\documentclass[11pt]{article}
\usepackage[utf8]{inputenc} 
\usepackage[T1]{fontenc}    
\usepackage{hyperref}       
\usepackage{url}            
\usepackage{booktabs}       
\usepackage{amsfonts}       
\usepackage{nicefrac}       
\usepackage{microtype}      
\usepackage{makecell}
\usepackage{graphicx}
\usepackage{enumitem}
\usepackage{booktabs}

\usepackage{multicol}

\usepackage{amssymb,amsmath,amsthm}
\usepackage{algpseudocode}
\usepackage{array}
\usepackage{cite}
\usepackage{fullpage}
\usepackage{graphicx}
\usepackage{url}
\usepackage{mathrsfs}
\usepackage{color}
\usepackage{verbatim}
\usepackage{mathtools}
\usepackage{algorithm}

\newtheorem{theorem}{Theorem}

\newtheorem{proposition}{Proposition}
\newtheorem{lemma}{Lemma}
\newtheorem{corollary}{Corollary}
\newtheorem{example}{Example}
\newtheorem{assumption}{Assumption}
\newtheorem{remark}{Remark}
\newtheorem{innercustomthm}{Theorem}
\newenvironment{customthm}[1]
  {\renewcommand\theinnercustomthm{#1}\innercustomthm}
  {\endinnercustomthm}
  
  \newtheorem{innercustomlem}{Lemma}
\newenvironment{customlem}[1]
  {\renewcommand\theinnercustomlem{#1}\innercustomlem}
  {\endinnercustomlem}
  
\makeatletter
\def\fixedlabel#1#2{%
  \@bsphack%
  \protected@write\@auxout{}%
         {\string\newlabel{#1}{{#2}{\thepage}}}%
  \@esphack}
\makeatother

\DeclareMathOperator*{\argmin}{argmin}
\DeclareMathOperator*{\ri}{ri}

\newcommand{\Sum}{\sum\limits_{i=1}^n}
\newcommand{\Sumj}{\sum\limits_{j=1}^m}
\newcommand{\EE}{\mathbb{E}}
\newcommand{\MM}{\mathcal{M}}

\newcommand{\meanj}{\frac{1}{m}\Sumj}

\newcommand{\eqdef}{\coloneqq}
\newcommand{\R}{\mathbb{R}}

\newcommand{\cX}{{\cal X}}
\newcommand{\PP}{\Pi_\cX}
\newcommand{\PPj}{\Pi_{\cX_j}}
\newcommand{\xleps}{x_{\lambda, \varepsilon}}

\renewcommand{\le}{\leqslant}
\renewcommand{\ge}{\geqslant}
\renewcommand{\leq}{\leqslant}
\renewcommand{\geq}{\geqslant}

\graphicspath{ {plots/} }

\usepackage[colorinlistoftodos,bordercolor=orange,backgroundcolor=orange!20,linecolor=orange,textsize=scriptsize]{todonotes}

\begin{document}


\title{A Stochastic Penalty Model for Convex and Nonconvex Optimization with Big Constraints}
\author{
  Konstantin Mishchenko${}^{1}$ \qquad Peter Richt\'{a}rik${}^{1,2,3}$ \\ \\
  ${}^{1}$ King Abdullah University of Science and Technology \\  ${}^{2}$University of Edinburgh \\ ${}^{3}$Moscow Institute of Physics and Technology\\
}
\maketitle




\begin{abstract}
The last decade witnessed a rise in the importance of supervised  learning applications involving {\em big data} and {\em big models}.  Big data refers to  situations where the amounts of training data  available and needed causes difficulties in the training phase of the pipeline. Big model refers to  situations where large dimensional  and over-parameterized models are  needed for the application at hand. Both of these phenomena lead to a dramatic increase in research activity aimed at taming the issues via the design of new sophisticated optimization algorithms.  In this paper we turn attention to the  {\em big constraints} scenario and argue that elaborate machine learning systems of the future will necessarily need to account for a large number of  real-world constraints, which will need to be incorporated  in the training process. This line of work is largely unexplored, and provides ample opportunities for future work and applications.  To handle the {\em big constraints} regime, we propose a {\em stochastic penalty} formulation which {\em reduces the problem to the well understood big data regime}. Our formulation has many interesting  properties which relate it  to the original problem in various ways, with  mathematical guarantees. We give a number of results specialized to nonconvex loss functions, smooth convex functions, strongly convex functions and convex constraints. We show through experiments that our approach can beat  competing approaches by several orders of magnitude when a medium accuracy solution is required.

\end{abstract}


\section{Introduction}

Supervised machine learning models are typically trained by minimizing an empirical risk objective, which has the form of an average of  $n$  loss functions $f_i$, where  $f_i(x)$ measures the loss associated with model $x\in \R^d$ when applied to data point $i$ of a training set:
\begin{align}
\label{eq:pb}
\min_{x \in \cX \subseteq \R^d} f(x) \eqdef \frac{1}{n}\sum_{i=1}^n f_i(x).
\end{align}
In the {\em big data} regime $n$ is very large, and is the source of issues when training the model, i.e., when searching for $x$ that minimizes $f$. In the big model regime $d$ is very large, which also causes considerable issues (e.g., cost of each iteration / backpropagation and communication). In modern deep learning applications, both $d$ and $n$ are large, and their relative sizes depend on the application. 

In the big data regime ($n$ is large), the state-of-the-art methods are based on {\em stochastic gradient descent (SGD)} \cite{RobbinsMonro:1951}, enhanced with additional tricks such as  minibatching \cite{pegasos2}, acceleration \cite{allen2017katyusha}, importance sampling \cite{csiba2016importance} and variance reduction \cite{SVRG, SAG, SAGA}.  In the big model regime ($d$ is large), the state-of-the-art methods are based on {\em randomized coordinate descent (RCD)} \cite{Nesterov:2010RCDM}, also typically enhanced with additional tricks such as minibatching \cite{PCDM}, acceleration \cite{APPROX} and importance sampling \cite{allen2016even}.\footnote{Note that variance reduction is not needed for CD methods as they are variance-reduced by design. Indeed, as $x\to x^*$, where $x^*=\arg \min_x f(x)$, all partial derivatives of $f$ at $x$ converge to zero. This is to be contrasted with SGD, where it is not true that $\nabla f_i(x) \to 0$ as $x\to x^*$, which necessitates the incorporation of explicit variance reduction strategies.} 

\subsection{Constrained optimization} 

All ``standard'' variants of SGD-type and CD-type methods can be extended, to a certain degree, to handle a constrained version of the above problem. In particular, if  $\cX\neq \R^d$, the  basic idea is to perform a step of the standard method, followed by a projection onto $\cX$~\cite{nesterov2013introductory, SAGA, xiao2014proximal, shalev2012proximal, pedregosa2017breaking}:
\begin{equation}\label{eq:PG}
	x^{k + 1} = \PP(x^k - \omega_k g^k),
\end{equation}
where $\PP$ denotes projection operator onto the set $\cX$, $\omega_k>0$ is a stepsize and $g^k$ is a descent direction. That is, the basic idea behind projected gradient descent is utilized. The situation is more complicated with CD type methods as currently they only work for separable or block separable constraints (block CD methods are needed for block separable  constraints \cite{UCDC}).  Convergence properties of SGD and CD-type algorithms  are typically unaffected by the inclusion of the projection step. However, what is affected is the cost of each iteration, which depends on the structure of $\cX$.  

The problem with constraints is further exacerbated by the fact that the success of SGD and CD methods lies in their very cheap iterations. Indeed, if a cheap iteration is to be followed by an expensive projection step, the advantages of using a stochastic method over, say, gradient descent will be reduced, and may completely disappear once the relative cost of a projection compared to the cost of performing one SGD or CD step exceeds a certain threshold. At the moment, very little is known about how to best handle this regime. We argue, however, that this regime is important, and will become increasingly important in the future with $\cX$ representing real-life constraints imposed by the environment in which the ML system will be operating.

\subsection{Contributions}

We develop a novel approach to solving problem \eqref{eq:pb} in the case when $\cX$ is described as the intersection of a very large number of constraints, each of which admits a cheap projection operator, while iterative projection onto $\cX$ is prohibitively expensive. In particular, we develop a novel stochastic penalty reformulation of the problem, and prove an array of theoretical results connecting exact and approximate solutions of the reformulation to the 
original problem \eqref{eq:pb} in various ways, including distance to the optimal solution, distance to feasibility, function/loss suboptimality and so on. This is done for both smooth nonconvex and convex problems. Moreover, we develop a new increasing penalty method, which uses an arbitrary inner solver as a subroutine, and establish its convergence rate. We show through  experiments that our approach and methods can outperform other existing approaches by large margins.

 \section{Stochastic  Projection Penalty Approach to Big Constraints}
 
In this work we are specifically interested in the constrained version of \eqref{eq:pb}; that is, we  consider the case $\cX\neq \R^d$. We shall assume that the problem is solvable, i.e.,  there exists $x^*\in \cX$ such that $f(x^*) \leq f(x)$ for all $x\in \cX$. Furthermore, we shall assume throughout that $f$ is lower bounded on $\R^d$, and that it achieves its minimum on $\R^d$. In particular, let $x^*_0$ denote a minimizer of $f$ on $\R^d$. 

Crucially, we assume that  performing projected iterations of the type \eqref{eq:PG} is prohibitive because $\cX$ is so complex that the projection step is much more computationally expensive than computing $g^k$. 
There are several different structural reasons for why projecting onto $\cX$ might be difficult, and in this work we focus on one of them, described next.

 
 \subsection{Big constraints}

 In this work we specifically address the situation when $\cX$ arises as the intersection of a big number of simpler constraints, 
 \begin{align}
 \label{eq:structure}
 	\emptyset \neq \cX \eqdef \bigcap_{j=1}^m \cX_j,
 \end{align}
 each of which admits a cheap projection $\Pi_{\cX_j}(\cdot)$. 


 The departure point for our work is the observation that the {\em feasibility} problem associated with \eqref{eq:structure} (i.e., the  problem: find $x\in \cX$) admits a {\em stochastic optimization reformulation} of the form 
\begin{equation}\label{eq:min_h} 
 \min_{x\in \R^d} \left[h(x) \eqdef \frac{1}{m} \sum_{j=1}^m h_j(x)\right], 
\end{equation}
where
\begin{equation}\label{eq:h_j}
 h_j(x) \eqdef \frac{1}{2} \|x - \PPj(x) \|^2,
 \end{equation}
and $\|\cdot\|$ is the standard Euclidean norm.  The name ``stochastic'' comes from interpretation of $h$ as the expectation of $h_j$, with $j$  picked uniformly at random. Note that $h(x)=0$ for all $x\in \cX$. In particular, $h(x^*)=0$.


If the sets $\cX_j$ are closed and convex, then $h_j$ is convex and $1$-smooth\footnote{That is, $\|\nabla h_j(x) - \nabla h_j(y)\| \leq \|x-y\|$ for all $x,y\in \R^d$. This follows from the formula $\nabla h_j(x) = x - \Pi_{\cX_j}(x)$ and from nonexpansiveness of the projection operator.}\cite{Ned:10}, and hence problem \eqref{eq:min_h} can in principle be solved by popular methods such as SGD, or any of its many variants. Indeed, it was recently shown in \cite{necoara2018randomized} that, under a {\em stochastic linear regularity} condition (see Assumption~\ref{as:lin_reg} below) on the sets $\{\cX_j\}_{j=1}^m$, {\em SGD  with unit stepsize} and uniform selection of sets, applied to \eqref{eq:min_h},  i.e., 
\[x^{k+1} = x^k - \nabla h_j(x^k) = \Pi_{\cX_j}(x^k), \]
converges at the linear rate $\EE \|x^k-\Pi_\cX(x^k)\|^2 \leq (1-\gamma)^k \|x^0-\Pi_\cX(x^0)\|^2$, where $0<\gamma\leq 1$ is defined below.  This is quite remarkable as $h$ is not necessarily strongly convex.

\begin{assumption}[Stochastic linear regularity~\cite{Ned:10,necoara2018randomized}]
\label{as:lin_reg}
There is a constant $\gamma > 0$ such that  the following inequality holds for all $x\in \R^d$:
    \begin{align}
    \label{eq:lin_reg}
 	\meanj \| x - \PPj(x)\|^2 \ge \gamma \|x - \PP (x)\|^2.
    \end{align}
\end{assumption}
It can be easily seen that, necessarily, $\gamma\leq 1$.
Inequality \eqref{eq:lin_reg} implies that if $h(x)=0$, then $x\in \cX$. Put together with what we said above, $x\in \cX$ if and only if $h(x)=0$. We shall enforce the linear regularity assumption throughout the paper as many of our results depend on it.\footnote{As a rule of thumb, a result that does not refer to $\gamma$ does not depend on this assumption.} In Appendix~\ref{sec:LR-suppl} we include some known and new examples of (not necessarily convex) sets for which Assumption~\ref{as:lin_reg} is satisfied.

Minibatch and variable stepsize extensions of SGD for \eqref{eq:min_h}  have been studied  as well \cite{necoara2018randomized}. 
For alternative developments  and connections to linear feasibility problems, duality and quasi-Newton updates, we refer the reader to \cite{SIMAX2015, SDA, inverse, gower2018accelerated}.
 
\subsection{Stochastic penalty reformulation}
  
Motivated by the above  considerations, we propose to reformulate 
 \eqref{eq:pb}+\eqref{eq:structure} into the problem
\begin{align}
\label{eq:pb_relaxed}
	\min_{x\in \mathbb{R}^d} f(x) + \lambda h(x),
\end{align}
where $h(x) $ is as in \eqref{eq:min_h}. That is, we have transformed the constrained problem \eqref{eq:pb} into an unconstrained one, with penalty $h(x)$ and penalty parameter $\lambda>0$. For $f\equiv 0$, \eqref{eq:pb_relaxed} reduces to \eqref{eq:min_h}, which is now well understood. From this perspective, one can think of  \eqref{eq:pb_relaxed} as a regularized version of \eqref{eq:min_h}. 

Since both $f$ and  $h$ are of a finite-sum structure, problem \eqref{eq:pb_relaxed} is solvable by modern stochastic gradient-type methods which operate well in the regime when $n+m$ is big. In other words, we have reduced the {\em big constraint} problem \eqref{eq:pb}+\eqref{eq:structure} into a finite-sum (or {\em big data}) problem, where the functions $h_j$ play the role of extra loss functions associated with the  constraints $\cX_j$. 


To the best of our knowledge, problem \eqref{eq:pb_relaxed} was not studied in this generality before  (except for the case $f \equiv 0$ in \cite{Ned:10, necoara2018randomized}, and $m=1$ case with convex nonsmooth $f$ in \cite{tran2017proximal}).

\subsection{Solving the reformulation}

We propose that instead of solving the original constrained problem, one solves the reformulation \eqref{eq:pb_relaxed}.  In particular, we propose two generic solution approaches:
\begin{enumerate} 
\item Solve \eqref{eq:pb_relaxed}, obtaining solution  $x^*_\lambda$. Output $x^*_\lambda$.
\item Solve  \eqref{eq:pb_relaxed}, but output  $\Pi_{\cX}(x^*_\lambda)$.
\end{enumerate}

The main focus of this work is to understand how good the points $x^*_\lambda$ and $\Pi_{\cX}(x^*_\lambda)$ are as solutions of the original problem  \eqref{eq:pb}+\eqref{eq:structure}. Our second approach always outputs a feasible point, and this comes at the cost of  computing a single projection onto $\cX$. This is obviously dramatically fewer projections than the iterative projection scheme \eqref{eq:PG} requires, and hence this approach makes sense in situations where computing a single projection is not prohibitive. With this approach we would like to obtain bounds on the difference between $f(x^*)$ and $f(\Pi_{\cX}(x^*_\lambda))$. The first proposed approach can't guarantee that $x^*_\lambda$ is feasible for 
 \eqref{eq:pb}. Hence, besides function suboptimality, we need to argue about distance of $x^*_\lambda$ to $\cX$, or its distance to $x^*$. 
 
\paragraph{Inexact solution.} In practice, one would use an iterative method for solving \eqref{eq:pb_relaxed} and hence it is not reasonable to assume that one can obtain the exact solution $x^*_\lambda$. To this effect, we also study the above two approaches in this inexact regime. In particular, we assume that we compute $x_{\lambda, \varepsilon}$ such that  
	\begin{equation}\label{eq:inex_sol}
    	f(\xleps) + \lambda h(\xleps) \le f(x_\lambda^*) + \lambda h(x_\lambda^*) + \varepsilon,
    \end{equation}
in case a deterministic method is used to obtain $\xleps$, or
$
    \EE	\left[f(\xleps) + \lambda h(\xleps)\right] \le f(x_\lambda^*) + \lambda h(x_\lambda^*) + \varepsilon
$
    in case a stochastic method is used and hence $\xleps$ is random.

\subsection{Assumptions on $f$}

For the sake of clarity, we shall first describe the exact solution theory (Section~\ref{sec:exact}), followed by the inexact solution theory (Section~\ref{sec:inexact}). In each case, we shall give an array of bounds, depending on assumptions.

We develop the theory under a variety of assumptions (and their combinations) on  the function $f$:
\begin{enumerate}
\item No assumptions on $f$ (e.g., $f$ could be nonconvex and nondifferentiable),
\item $f$ is $L$--smooth (but can be nonconvex)
\item $f$ is differentiable and convex or strongly convex
\end{enumerate}


\subsection{Other approaches}

One of the earliest applications of the function $h(\cdot)$ was in \cite{Ned:10}, where the authors used it with the zero objective $f\equiv 0$. Some of the results obtained there were later rediscovered by \cite{necoara2018randomized}.

More surprisigly, a few recent works tackle a problem similar to ours. For instance, in \cite{kundu2017convex} the authors consider an exact penalty approach and obtain methods with  $O(1/k^2)$ convergence rate. However, their approach suffers due to the need to compute full gradients,  including projections onto all sets at each iteration.  The work of \cite{ryu2017proximal} (S-PPG method) was  designed to tackle a similar problem as ours. Their approach, however, requires  storing  $\max(n, m)$ full-dimension variables, and does not provide any infeasibility guarantees. Convex objectives were studied in \cite{mahdavi2012stochastic}, while \cite{yang2017richer}  considers  a convex objective with inequality constraints. The approach of \cite{tran2017proximal} considers the same penalty model as ours, but for one set only ($m=1$) with a convex nonsmooth objective.

\section{Exact  Solution Theory} \label{sec:exact}

In this section we develop out {\em exact solution theory}. That is, we develop a series of results connecting $x^*_\lambda$  (the exact solution of \eqref{eq:pb_relaxed}) and  $\Pi_{\cX}(x^*_\lambda)$ to original problem \eqref{eq:pb}. In the rest of this section, we let $Opt_\lambda \eqdef f(x_\lambda^*) + \lambda h(x_\lambda^*) $.

\subsection{No assumption on $f$}

Our first result says, among other things, that the optimal value of the reformulated problem \eqref{eq:pb_relaxed} is always a lower bound on the optimal value of the original problem \eqref{eq:pb}. This result does not depend on Assumption~\ref{as:lin_reg}, nor on any assumption on $f$ such as differentiability or convexity.

\begin{lemma}\label{lem:bu98gd08g0s}
For all $\lambda \geq 0$ we have
  \begin{equation} \label{eq:opt_values_monotonicity1}
    Opt_{\lambda}
       \leq f(x^*). 
    \end{equation}
Moreover,    
	for any $\lambda \ge \theta \ge 0$ we have $f(x^*) \ge f(x_\lambda^*) \ge f(x_\theta^*)$ and $0\leq h(x_\lambda^*) \le h(x_\theta^*)$. 
\end{lemma}
\begin{proof}
See Appendix~\ref{sec:proofs_nof_assump}.
\end{proof}

The lemma says that $f(x^*_\lambda)$ is increasing in $\lambda$, while $h(x^*_\lambda)$ is decreasing. However, without further assumptions, it may not be the case that $f(x^*_\lambda)\to f(x^*)$ or $h(x^*_\lambda)\to 0$ as $\lambda\to \infty$. Still, in  situations where it is desirable to quickly find a rough  lower bound on $f(x^*)$,  and especially when $f$ is a difficult function, the above lemma can be of help. 

\subsection{$L$-smoothness}

We now describe our main result under the $L$-smoothness assumption: 
\[\|\nabla f(x) - \nabla f(y)\| \le L\|x - y\|, \quad x,y\in \R^d.\]
In particular, we allow $f$ to be nonconvex. Some of the results are a refinement  of those in Lemma~\ref{lem:bu98gd08g0s}.

\begin{theorem}
\label{thm:optimal_obj_values}
	If $f$ is $L$--smooth, then
\begin{itemize}	
\item[(i)]	 For all $\lambda\ge 2\frac{L}{\gamma}$,  $f(\PP(x_\lambda^*)) $ and $f(x^*)$ are related via
\[
    	f(x^*) \leq f(\PP(x_\lambda^*)) 
        \le  f(x^*) + \frac{2L(f(x^*) - f(x_0^*))}{\gamma \lambda}. 
\]
\item[(ii)] For all $\lambda\ge 2\frac{L}{\gamma}$ (lower bound) and all $\lambda\geq 0$ (upper bound), $Opt_\lambda$ and $f(x^*)$ are related via
	\[f(x^*) - \frac{2L(f(x^*)-f(x^*_0))}{\gamma \lambda}  \leq Opt_{\lambda} \le f(x^*) - \frac{\|\nabla f(x^*)\|^2}{2(L + \lambda)}.\]	
\item[(iii)]
	For all $\lambda>0$ we have
\[
 \frac{\gamma G}{L^2 + \gamma \lambda^2} \le h(x_\lambda^*) \le \frac{f(x^*) - f(x_0^*)}{\lambda},
\]
    where 
\[G \eqdef \frac{1}{4}\inf_{x\in \cX}\|\nabla f(x)\|^2 \leq \frac{1}{4}\|\nabla f(x^*)\|^2.
\]
\item[(iv)] The distance of $x^*_\lambda$ to $\cX$ is for all $\lambda\geq 0$ bounded by
\[
    	\frac{2 G}{L^2 + \lambda^2}\le \|x_\lambda^* - \PP(x^*_\lambda)\|^2 \leq \frac{2(f(x^*) - f(x^*_0))}{\gamma \lambda}.
\]
\item[(v)] The distance of $x^*_\lambda$ to the optimal solution $x^*$ cannot be too small. In particular, for all $\lambda \geq 0$,
\[
      \frac{\|\nabla f(x^*)\|^2}{(L + \lambda)^2} \leq 	\|x_\lambda^* - x^*\|^2.
\]
      \end{itemize}
\end{theorem}
\begin{proof} See Appendix~\ref{sec:exact_L_smooth_proofs}.
\end{proof}

The first set of inequalities say that $\Pi_{\cX}(x^*_\lambda)$ is a $O(1/\lambda)$ overapproximation of $f(x^*)$. In particular, $0 \leq f(\PP(x_\lambda^*)) - f(x^*) \leq \varepsilon$ as long as $\lambda = O(1/\varepsilon)$. Since $\Pi_{\cX}(x^*_\lambda) \in \cX$, there is no issue with feasibility. One should not expect a lower bound on  $f(\PP(x_\lambda^*))$ that would strictly separate it from $f^*$. In Appendix~\ref{sec:1dexample} we give a simple example of $\cX$, and a smooth and strongly convex $f$  for which $f(\PP(x_\lambda^*))  =f(x^*)$, while $\|\nabla f(x^*)\|>0$. The lower bound  on $Opt_\lambda$ says that $Opt_\lambda$ can not be much smaller than $f(x^*)$, while the upper bound says that $Opt_\lambda$ cannot be too close $f(x^*)$ either, which sharpens \eqref{eq:opt_values_monotonicity1}. Note that $|Opt_\lambda - f(x^*)| = O(1/\lambda)$, and hence $|Opt_\lambda \to f(x^*)| \leq \varepsilon$ as long as $\lambda = \Omega(1/\varepsilon)$. In the unconstrained case ($\cX = \R^d$) we have $G=0$ in (iii), and the lower bound on $h(x^*_\lambda)$ is zero, which is naturally expected. Inequalities (iv) give an $\Omega(1/\lambda^2)$ and $O(1/\lambda)$ lower and upper bounds on the (squared) distance of $x^*_\lambda$ from $\cX$, respectively. As we shall see in Theorem~\ref{thm:u98sgbd}, the upper bound can be improved to $O(1/\lambda^2)$ under convexity. Finally, (v) is a negative result; it says that $x^*_\lambda$ and $x^*$ cannot be too close, unless $\nabla f(x^*)=0$.

\subsection{Differentiability and convexity}

Our second main result is an analogue of Theorem~\ref{thm:optimal_obj_values}, but with the $L$-smoothness assumption  replaced by  differentiability and convexity.

\begin{theorem}\label{thm:u98sgbd} If $f$ is differentiable and convex, then for all $\lambda>0$
\begin{itemize}
\item[(i)] The values $f(x^*_\lambda)$ and $f(x^*)$ are related via
\[
    0\leq f(x^*) - f(x_\lambda^*) \leq  \frac{2}{\gamma \lambda}  \|\nabla f(x^*)\|^2 . 
\]
\item[(ii)]       
$$
        h(x^*_\lambda) \leq  \frac{2 \|\nabla f(x^*)\|^2}{\gamma \lambda^2}.  
$$
\item[(iii)]    The distance of $x^*_\lambda$ to $\cX$ is bounded above by
\[
     	\|x_\lambda^* - \PP(x^*_\lambda)\|^2 \leq \frac{4\|\nabla f(x^*)\|^2}{\gamma^2 \lambda^2}.
\]
\item[(iv)]    	If $f$ is $L$-smooth and if $f+ \lambda h$ is  $\mu$-strongly convex with $\mu>0$ (for this it suffices for $f$ to be $\mu$-strongly convex), then the distance of $x^*_\lambda$ from $x^*$ is bounded above by
\[
      \|x^*_\lambda - x^*\|^2 \leq \frac{L + \lambda - \gamma\lambda}{\gamma\mu\lambda(L + \lambda)}\|\nabla f(x^*)\|^2.
\]
    In particular, if $f$ is differentiable but not $L$-smooth (i.e., $L=+\infty$), then the bound simplifies to
\[
        \|x^*_\lambda - x^*\|^2 \leq \frac{1}{\gamma\mu\lambda}\|\nabla f(x^*)\|^2.
\]
\end{itemize}
\end{theorem}
\begin{proof}
See Appendix~\ref{sec:proofs_convex}.
\end{proof}

Note that the size of the gradient at optimum dictated the upper bounds. In the case when $\nabla f(x^*)=0$, we recover the expected results: $f(x^*)=f(x^*_\lambda)$, $h(x^*_\lambda)=0$, $x^*_\lambda \in \cX$ and $x^*_\lambda = x^*$. Further, note that compared to the $L$-smoothness results, convexity allowed us to disassociate $f_\lambda^*$ and $h_\lambda^*$ from each other, and thus enabled us to get a $O(1/\lambda)$ bound on $f(x^*)-f(x^*_\lambda)$ in (i). Further,  in (iii) we get an $O(1/\lambda^2)$ upper bound on the (squared) distance of $x^*_\lambda$ from $\cX$; this is an order of magnitude better than the $O(1/\lambda)$ bound obtained through $L$-smoothness. Finally, under strong convexity, we  get an $O(1/\lambda)$ upper bound on the (squared) distance of $x^*_\lambda$ from $x^*$ (see (iv)). Note that Theorem~\ref{thm:optimal_obj_values} provides a $\Omega(1/\lambda^2)$ lower bound on the same quantity. 
Observe that as long as $\lambda < \tfrac{L}{1 - \gamma}$, the upper bound improves to $O(\tfrac{1}{\lambda^2})$. In particular, in the extreme case when the stochastic linear regularity parameter $\gamma$ is equal to $1$
(which is possible only if all constraints are either $\R^d$ or $\cX$),  the improved upper bound holds for all $\lambda>0$.

\subsection{Summary of key results}

A brief summary of the key results obtained in this section
is provided in Table~\ref{tbl:summary_of_results}.

\begin{table}[t]
\begin{center}
\begin{tabular}{m{1.9cm} m{2.6cm} m{1.9cm}}
 Lower bound & Quantity & Upper bound\\
 \hline
 $  \Omega\left( \frac{1 }{L + \lambda}\right)^{*}$ & $f^* - (f_\lambda^* + \lambda h_\lambda^*)$ & $ O\left( \frac{1}{\lambda}\right)^{*}$ \\
 \hline
  $ \Omega\left( \frac{1 }{L + \lambda}\right)^{*}$ & $f^* - f_\lambda^*$ & $ O\left( \frac{1}{\lambda}\right)^{\dag}$ \\
 \hline
 $ \Omega\left( \frac{1}{L^2 + \lambda^2}\right)^{*}$  & $h_\lambda^*$ & $ O\left( \frac{1}{\lambda}\right)$, $ O\left( \frac{1}{\lambda^2}\right)^{\dag}$ \\ 
 \hline
  $ \Omega\left( \frac{1}{(L + \lambda)^2}\right)^{*}$ & $\|x_\lambda^* - x^*\|^2$  & $O(\frac{1}{\lambda})^{\ddag}$, $O(\frac{1}{\lambda^2})^{*, \ddag, \natural}$ \\
 \hline
  $ \Omega\left( \frac{1}{L^2 + \lambda^2}\right)^{*}$ & $\|x_\lambda^* - \PP(x^*_\lambda)\|^2$ & $ O\left( \frac{1}{\lambda^2}\right)^{\dag}$ \\
 \hline
   $0$ & $f(\PP(x^*_\lambda)) - f^*$ & $O\left( \frac{1}{\lambda}\right)^{*, \S}$ \\
 \hline
\end{tabular}
\caption{Lower and upper bounds for different measures of solution's quality ($f^*=f(x^*)$, $f_\lambda^* \eqdef f(x_\lambda^*)$, $h_\lambda^* \eqdef h(x_\lambda^*)$). Superscripts refer to assumptions used to prove the bound: * ($L$-smoothness), $\dag$ (convexity), $\ddag$ (strong convexity), $\S$ ($\lambda \gg L$), $\natural$ ($\lambda \le \frac{L}{1 - \gamma} $). }
\label{tbl:summary_of_results}
\end{center}
\end{table}

It is clear from a brief look at the table that in contrast to exact penalty reformulation approaches, such as \cite{kundu2017convex}, solving our reformulation does not produce a solution of problem \eqref{eq:pb}, but rather a point ``close'' to it. Moreover, our lower bounds say that it is not reasonable to solve the reformulation to exact accuracy. On the contrary, based on our theoretical results, and based on our computational experiments, it is best to apply fast but not necessarily very accurate methods such as SGD in order to get a proper approximate solution. Doing so suffices to ensure proximity to feasibility, while incurring loss  $f$ {\em smaller} than $f(x^*)$.

\subsection{Inexact solution theory} \label{sec:inexact}

 Our inexact solution theory contains results which are direct  analogues of the results contained in Theorems~\ref{thm:optimal_obj_values} and \ref{thm:u98sgbd}, but  apply to vector $\xleps$ satisfying \eqref{eq:inex_sol}. Because of this, and due to space constraints, the results can be found in Appendix~\ref{sec:approx}.

\section{Total Complexity}

We now compare the total cost (iteration complexity times cost of each iteration) of several well-known methods when applied to the original {\em big constraint} problem \eqref{eq:pb}, and when applied to our stochastic penalty reformulation~\eqref{eq:pb_relaxed}.

We choose $\lambda = 1/\sqrt{\delta}$, where $\delta$ is the desired accuracy of feasibility of $x$.  Assuming convexity of $f$, in view of Theorem~\ref{thm:u98sgbd} (parts (i) and (iii)) we get 
\[ f(x^*_\lambda) \leq f(x^*), \qquad  \|x^*_\lambda - \Pi_{\cX}(x^*_\lambda)\|^2 = O(\delta).\]
Similar results hold in the inexact case (see Appendix~\ref{sec:approx}).
 
Let $C_\Pi$  be the cost of projecting onto $\cX$, while $c_{\Pi}$ is the (average) cost of projecting onto $\cX_j$. In general, these costs relate as $C_\Pi \gg m c_{\Pi}$. Moreover, $c_{\nabla}$ is the (average) cost of computing a single stochastic gradient of $f$, i.e., of $\nabla f_i(x)$.

The total complexities for selected methods, including gradient descent (GD) in the strongly convex${}^{\ddag}$ and convex${}^{\dag}$ settings, SGD (see Theorem~\ref{thm:SGD} for the computation of the complexity), and variance-reduced methods SVRG~\cite{SVRG}, SARAH~\cite{SARAH} and SAGA~\cite{SAGA} (all in the strongly convex setting) are summarized in Table~\ref{tab:complexities}. 

\begin{table}[t]
\begin{center} \small
\begin{tabular}{| c | c | c |}
 \hline
 Method & Applied to \eqref{eq:pb} & Applied to \eqref{eq:pb_relaxed}\\
 \hline
 && \\ 
  $\textrm{GD}^{\ddag}$ & $\displaystyle  L(nc_{\nabla} + C_{\Pi})\log\frac{1}{\delta}$ &
  $\displaystyle \left(L + \frac{1}{\sqrt{\delta}}\right) (nc_{\nabla} + mc_{\Pi})\log\frac{1}{\delta}$ \\
     && \\
  \hline
     && \\
  $\textrm{GD}^{\dag}$ & $\displaystyle L(c_\nabla + C_\Pi)\frac{1}{\delta} $ & $\displaystyle \left(L + \frac{1}{\sqrt{\delta}}\right) (nc_{\nabla} + mc_{\Pi})\frac{1}{\sqrt{\delta}} $ \\
 && \\  
 \hline
 && \\
 SGD & $\displaystyle \frac{B^2D^2}{\delta^2}(c_\nabla + C_\Pi)$ & $\displaystyle \frac{\left(B^2 + \frac{D^2}{\delta}\right)D^2}{\delta}(c_\nabla + c_\Pi)$  \\
   && \\
 \hline 
&& \\
 SVRG \& SARAH & $\displaystyle  \left(n + \frac{L}{\mu} C_\Pi \right)c_\nabla \log\frac{1}{\delta}$ & $\displaystyle \left(n + \frac{L + \tfrac{1}{\sqrt{\delta}}}{\mu}\right)(c_\nabla + c_\Pi)\log\frac{1}{\delta} $ \\
   && \\ 
 \hline
    && \\ 
 SAGA & $\displaystyle \left(n + \frac{L}{\mu}\right) (c_\nabla + C_\Pi) \log\frac{1}{\delta}$ & $\displaystyle \left(n + \frac{L + \tfrac{1}{\sqrt{\delta}}}{\mu}\right)(c_\nabla + c_\Pi)\log\frac{1}{\delta} $ \\
    && \\ 
    \hline
\end{tabular}
\end{center}
\caption{Total complexities of selected methods when applied to the original {\em big constraint} problem \eqref{eq:pb} and to its stochastic penalty reformulation \eqref{eq:pb_relaxed}. By $\delta$ we denote the target error tolerance. $B$ and $D$ are method-specific constants (see, e.g., Theorem~\ref{thm:SGD}.)}
\label{tab:complexities}
\end{table}

Notice that the total complexity of each method applied to  our stochastic penalty reformulation~\eqref{eq:pb_relaxed} can be vastly better than when applied to the original  problem \eqref{eq:pb}. This depends on the relationship between $C_\Pi$, $c_{\Pi}$, $c_{\nabla}$, $L$, $\delta$ and, in some cases, other quantities. For instance, 
in the regime when $\tfrac{1}{\delta}=O(L)$ and $m c_{\nabla} \ll C_\Pi$, GD (look at first line of the table) applied to \eqref{eq:pb_relaxed} will be much faster than GD applied to \eqref{eq:pb}. Similar insights can be gained by inspecting the total complexities of other methods in the table (and methods we did not include in the table).


\subsection{SGD for \eqref{eq:pb_relaxed}}

Algorithm~\ref{alg:lambda_sgd} is SGD adapted to problem~\eqref{eq:pb_relaxed}. We sample $i$ and $j$ uniformly at random and independently, from $\{1,2,\dots,n\}$ and $\{1,2,\dots,m\}$, respectively.

\begin{algorithm}[h]
   \caption{SGD for problem \eqref{eq:pb_relaxed}}
   \label{alg:lambda_sgd}
\begin{algorithmic}
   \State {\bfseries Input:} Point $x^1 \in \R^d$, number of iterations $K$, penalty $\lambda>0$, $\alpha > \tfrac{1}{\mu}$
   \For{$k=1$ {\bfseries to} $K$}
   \State Sample $i\sim U(1, \dotsc, n)$, $j\sim U(1,\dotsc, m)$
   \State $\omega_k = \frac{\alpha}{2\alpha (L + \lambda) + k}$
   \State $x^{k+1} = x^k - \omega_k (\nabla f_i(x^k) + \lambda\nabla h_j(x^k))$
   \EndFor
\end{algorithmic}
\end{algorithm}

 To illustrate how the total complexity results in Table~\ref{tab:complexities} were computed, we provide an example, in the case of SGD, in the next result (proof in  Appendix~\ref{sec:SGD_proof}).
 
\begin{theorem}\label{thm:SGD}
	Assume that $f$ is $\mu$-strongly convex and $L$-smooth, and sequence $\{x_k\}_{k\geq 0}$ is generated by Algorithm \ref{alg:lambda_sgd}. Then, we have $\EE\|x^k - x_\lambda^*\|^2 \le \frac{c}{\lambda}$, with some constant $c > 0$, after at most
\[
    	\frac{\lambda (2\alpha (L + \lambda) + 1)}{c}\left(2\|x^0 - x^*\|^2 + \frac{2\|\nabla f(x^*)\|^2}{\gamma\mu \lambda} \right)       + \left(\frac{\lambda (2\alpha (L + \lambda) + 1)}{4c(L + \lambda)^2} + \frac{\lambda\alpha^2}{c(\alpha\mu - 1)}\right) N = O\left( \lambda^2\right)
\]
    iterations, where
$
    	N \leq \frac{4}{n}\Sum \|\nabla f_i(x_\lambda^*)\|^2 + \frac{16 \|\nabla f(x^*)\|^2}{\gamma}.
$
\end{theorem}

\section{Increasing Penalty Method}

In this section we describe a new (meta) method (Algorithm~\ref{alg:increasing_lambda}) designed to solve  a sequence of stochastic penalty reformulations of the form \eqref{eq:pb_relaxed}
with a sequence of increasing values of penalty parameters $\lambda_k$.  The method runs an arbitrary but fixed algorithm $\cal M$ applied to the problem
\[\min f(x) + \lambda_k h(x)\]
for $N_k$ iterations  in an inner loop, started from $x^k$. The meta-method decides how $\lambda_k$ are updated, and performs iterations of the form \[x^{k + 1} = y^k - \omega_{k+1} \nabla h(y^k),\] producing points $\{x^{k}\}$. Note that for computing $\nabla h(y^k)$ we only need cheap projections onto $\cX_1, \dotsc, \cX_m$ instead of the more expensive projection onto $\cX$.


\begin{algorithm}[h]
   \caption{Increasing Penalty Method}
   \label{alg:increasing_lambda}
\begin{algorithmic}
   \State {\bfseries Input:} Method $\MM$, point $x^1\in \R^d$, number of outer iterations $K$, penalties $\lambda_1, \lambda_2, \dotsc, \lambda_K$, numbers of inner iterations $N_1, N_2, \dotsc, N_K$
   \For{$k=1$ {\bfseries to} $K$}
   \State $y^k = \MM(x^k, \lambda_k, N_k)$ \Comment{Apply  to $\min f + \lambda_k h$}
   \State Option I: $\omega_{k + 1} = \frac{\lambda_{k+1} - \lambda_k}{\gamma\lambda_{k+1}}$
   \State Option II:  $\displaystyle \omega_{k + 1} = \argmin_{\stackrel{x = y^k - \omega \nabla h(y^k)}{\omega \le 2}}f(x) + \lambda_{k+1}h(x)$
   \State $x^{k + 1} = y^k - \omega_{k+1} \nabla h(y^k)$
   \EndFor
\end{algorithmic}
\end{algorithm}

\subsection{Non-accelerated variant}

For algorithms such as SAGA \cite{SAGA}, SVRG \cite{SVRG} and SARAH \cite{SARAH} serving as method $\MM$ in Algorithm~\ref{alg:increasing_lambda}, the following theorem gives the rate of convergence of non-accelerated variant of our meta-method (Algorithm~\ref{alg:increasing_lambda}):

\begin{theorem}\label{thm:incr_lambda}
	    
	    Assume $f$ is $L$-smooth, $\mu$-strongly convex and that the constraints $\cX_1, \dotsc, \cX_m$ are closed and convex. Choose any method $\MM$ that takes as input a problem $F$, an initial point $x$, number of iterations $N$, the smoothness of the problem $L_F$ and possibly strong convexity constant $\mu$. Set $\lambda_{k+1} = \beta k + \nu$ for some $\beta > 0$, $\nu \ge \frac{\beta(1 - \gamma)}{\gamma}$ and $\omega_{k+1}=\frac{\beta}{\gamma(\beta(k +1) + \nu)}$. If for any $x$ method $\MM$ returns a point $y$ satisfying 
	    \[
	    F(y) - F^*\le \rho(F(x) - F^*)
	    \] 
	    after at most 
	    $C_\MM\frac{L_F}{\mu}\log \frac{1}{\rho}$
	    iterations, then Algorithm~\ref{alg:increasing_lambda} provides $(\lambda,\varepsilon)$-accurate solution, where $\lambda=\beta k + \nu$ and $\varepsilon = \tfrac{\theta}{k}$, after at most
        \[C_\MM \frac{L + \beta}{\mu} \left( 1 + \frac{L(f(x^*) - f(x_0^*))}{\theta\gamma\beta}\right)k + C_\MM \frac{L + \nu}{\mu}\max\left\{0, \log\left(\frac{f(x^1) - f(x_0^*)}{\theta}\right) \right\} = O(k)\]
    iterations in total.
\end{theorem}

\subsection{Accelerated variant}

In Appendix~\ref{sec:acc} we develop an accelerated variant of the Increasing Penalty Method. This method  after $k$ outer iterations outputs point $y^k$ with  enjoying the following guarantees:
	\begin{eqnarray*}
		f(\PP(y^k)) - f(x^*) & \le & \frac{1}{k^2}\left( 2L + \frac{f(x^*) - f(x_0^*)}{2\gamma } \right),\\
		\|y^k - x^*\|^2 & \le & \frac{4}{k^2}\left( \frac{L}{\mu} + \frac{\|\nabla f(x^*)\|^2}{\gamma \mu L} \right),\\
		\|y^k - \PP(y^k)\|^2 & \le &\frac{1}{k^4}\left( \frac{\|\nabla f(x^*)\|^2}{2\gamma^2 L^2} + \frac{2}{\gamma} \right).
	\end{eqnarray*}
We do not describe it here for space reasons.

\section{Experiments}


We performed several experiments with L2-regularized logistic regression with regularization parameter $1/n$,   on datasets A1a, Mushrooms and Madelon. 

We randomly generate linear constraints: half of them are inequalities and half are equalities. For A1a and Mushrooms datasets we use $m=40, 60, 100$ constraints in total, and for Madelon it is $m=100, 200, 400$. Although it may look like $m$ is not big, each full projection onto $\cX$ becomes very expensive as it requires solving a separate auxiliary optimization problem.  We run the projected versions of classical SGD and SVRG to show the slowdown due to the expensive projection. We also run two algorithms designed specifically for avoiding full projection step: EPAPD of \cite{kundu2017convex} and S-PPG of \cite{ryu2017proximal}. 

We compare the mentioned methods to our stochastic penalty reformulation approach with different methods under the hood as solvers. Specifically, we run SVRG as method $\cal M$ inside Algorithm~\ref{alg:increasing_lambda} on the reformulated problem \eqref{eq:pb_relaxed}  with linearly increasing $\lambda_k$, and SGD on the reformulated problem  \eqref{eq:pb_relaxed} with fixed $\lambda$ set to $100L$. We used $x^0=0$. We measure the distance from the problem's optimum $x^*$, i.e., $\|x^k-x^*\|$, as well as the objective suboptimality of the iterates after hard projection onto $\cX$, i.e., $f(\Pi_{\cX}(x^k)) - f^*$.  Our results show the superiority of our stochastic penalty approach for medium accuracy targets, in some cases  by {\em several orders of magnitude.}


We provide additional experimental results in  Appendix~\ref{sec:exp1} and Appendix~\ref{sec:exp2}.

\begin{table*}[ht]
  \begin{center}
  \begin{tabular}{c@{\quad}cccc}
    &
     $m=40$  & $m=60$& $m=100$ \\
   \makecell{\tiny Distance from \\ \tiny the optimum}
      & \raisebox{-\totalheight / 2}{\includegraphics[scale=0.30]{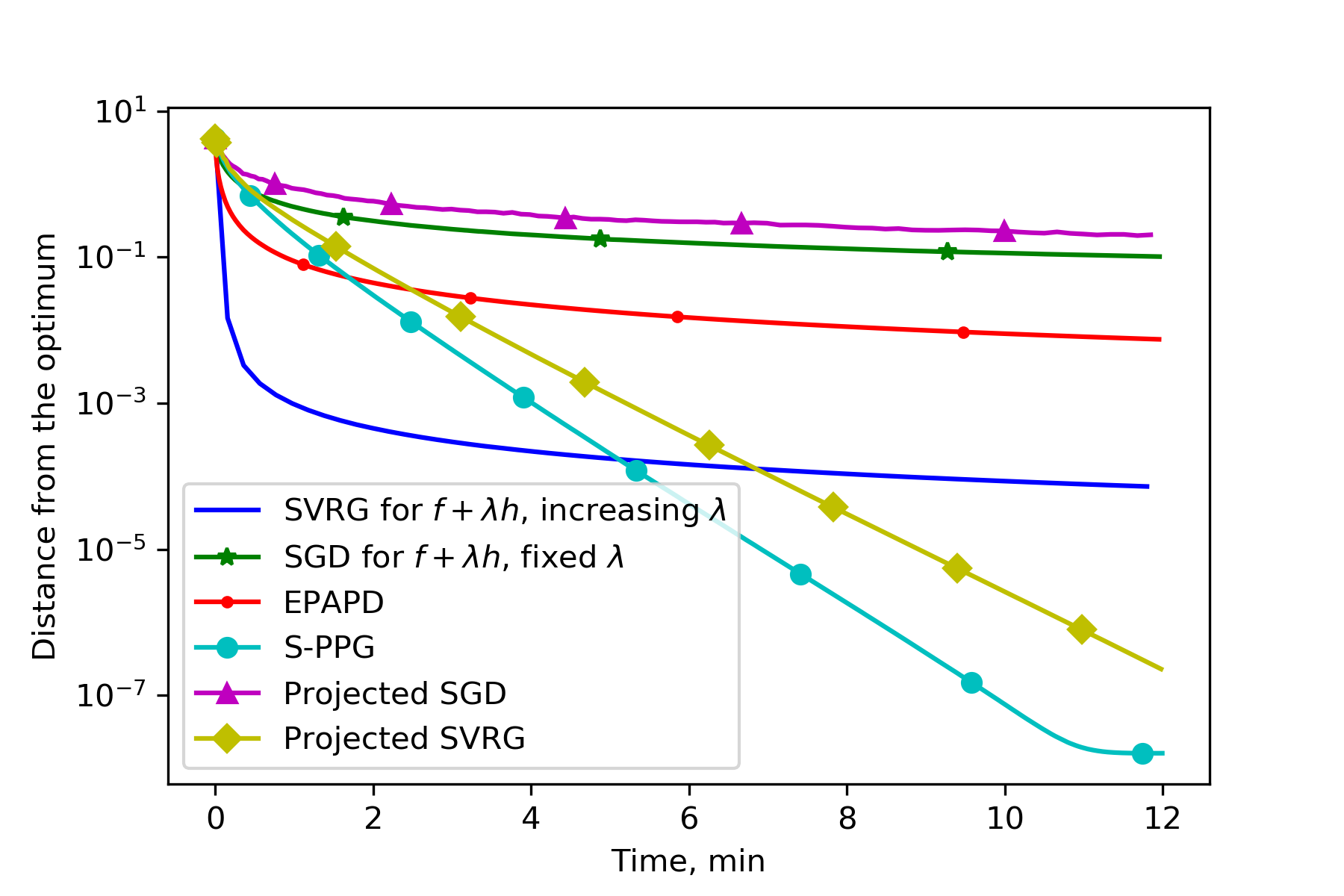}}
      & \raisebox{-\totalheight / 2}{\includegraphics[scale=0.30]{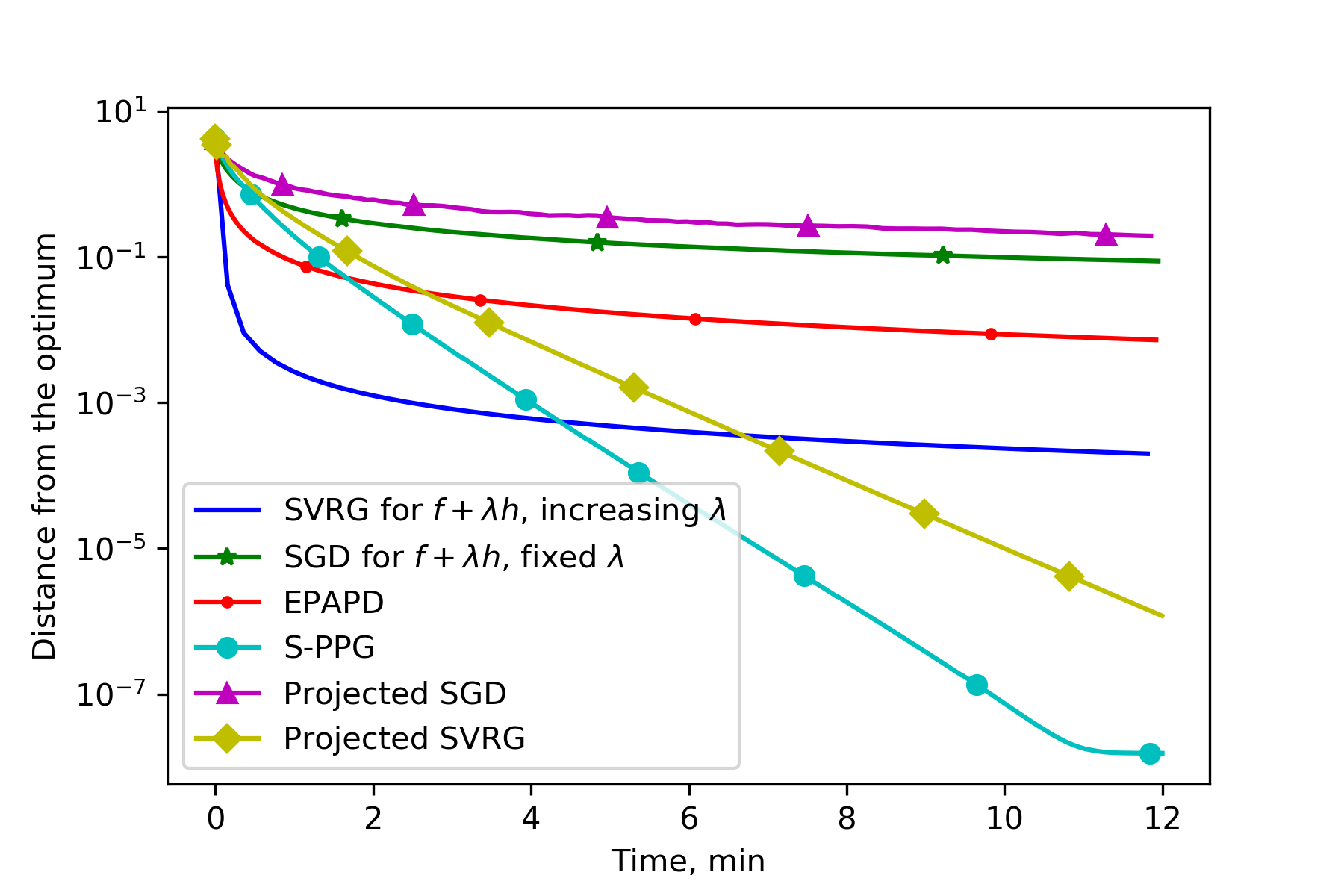}}
      & \raisebox{-\totalheight / 2}{\includegraphics[scale=0.30]{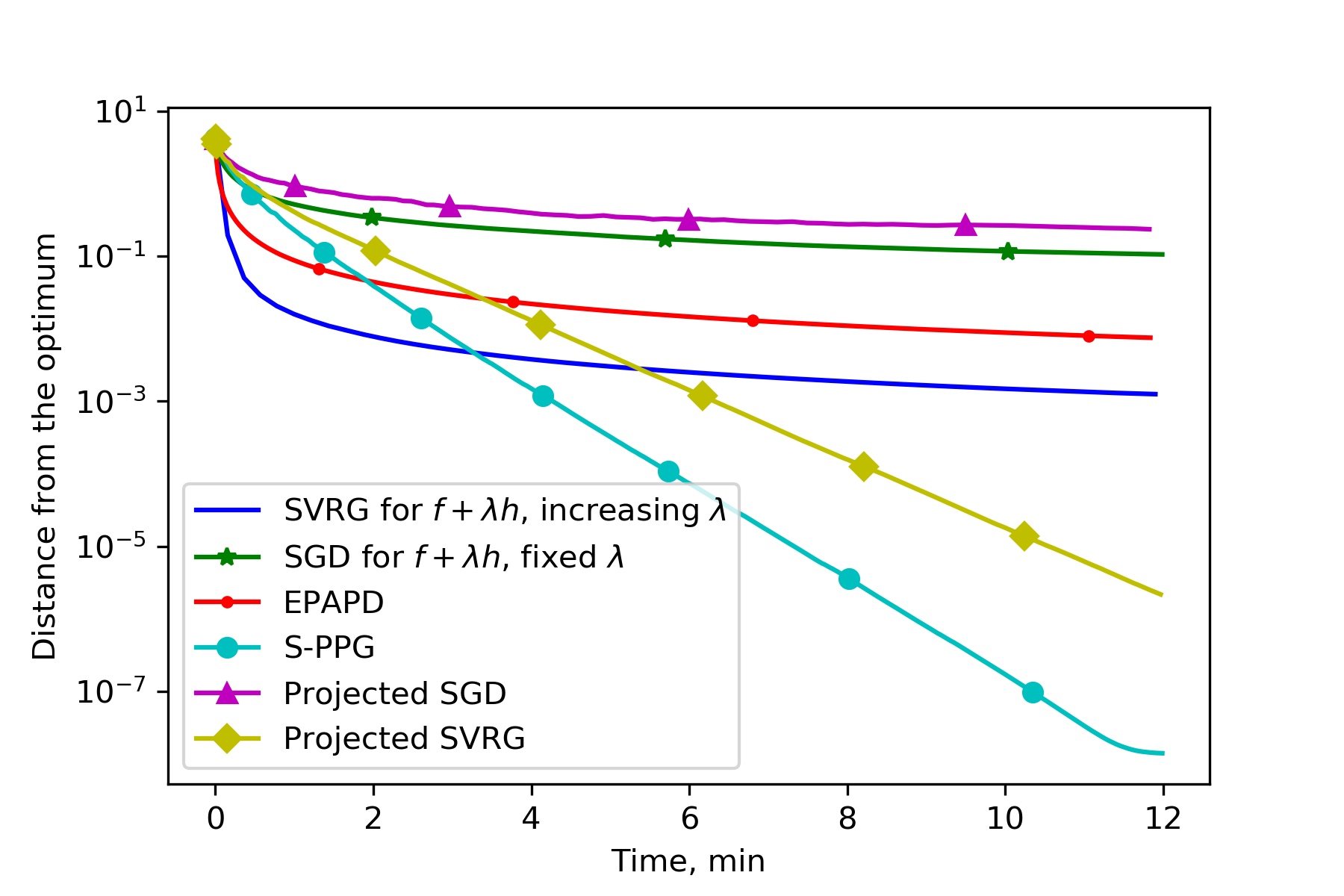}}\\ 
      
    \makecell{\tiny Objective value \\ \tiny of projected \\ \tiny iterates} 
      & \raisebox{-\totalheight / 2}{\includegraphics[scale=0.30]{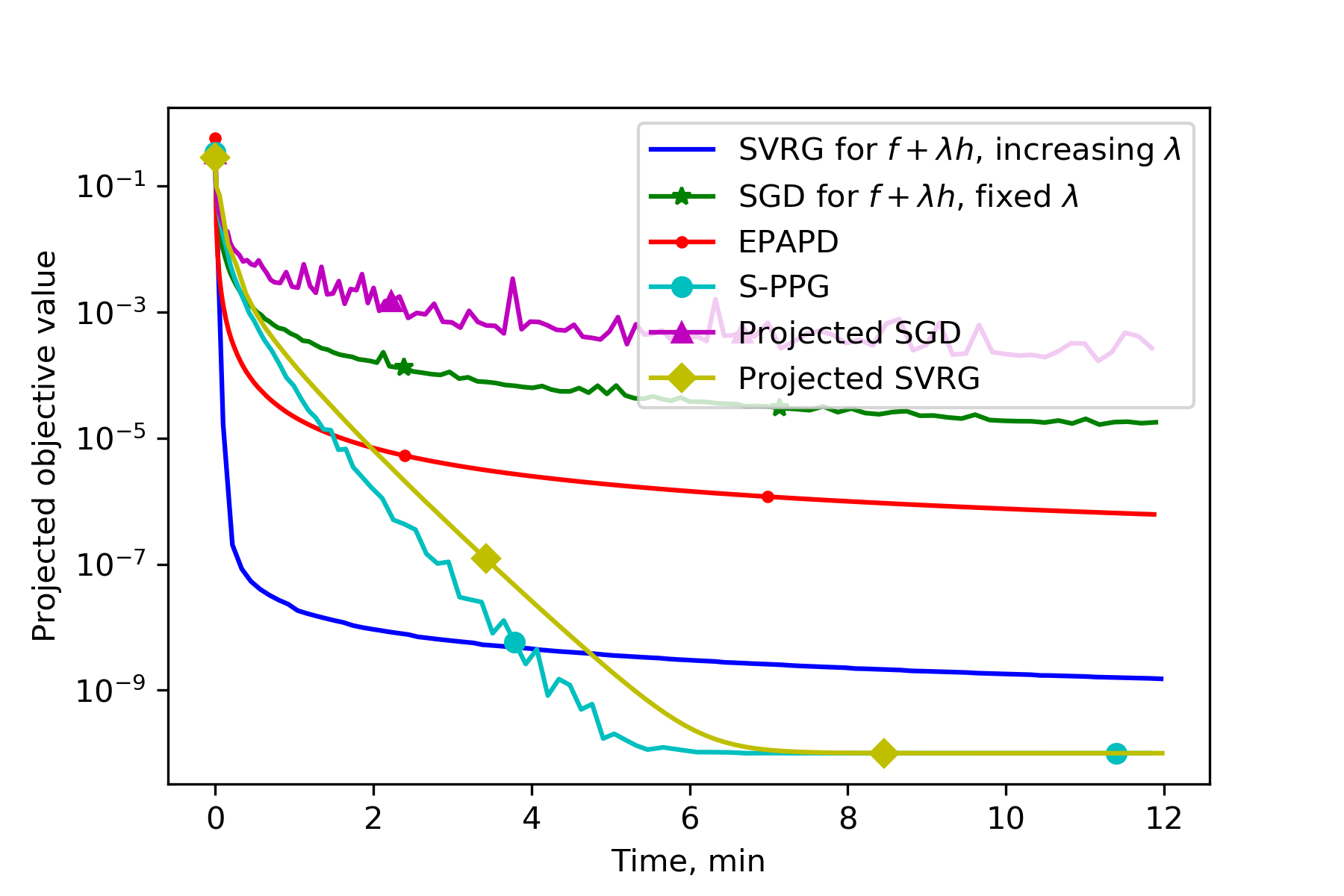}}
      & \raisebox{-\totalheight / 2}{\includegraphics[scale=0.30]{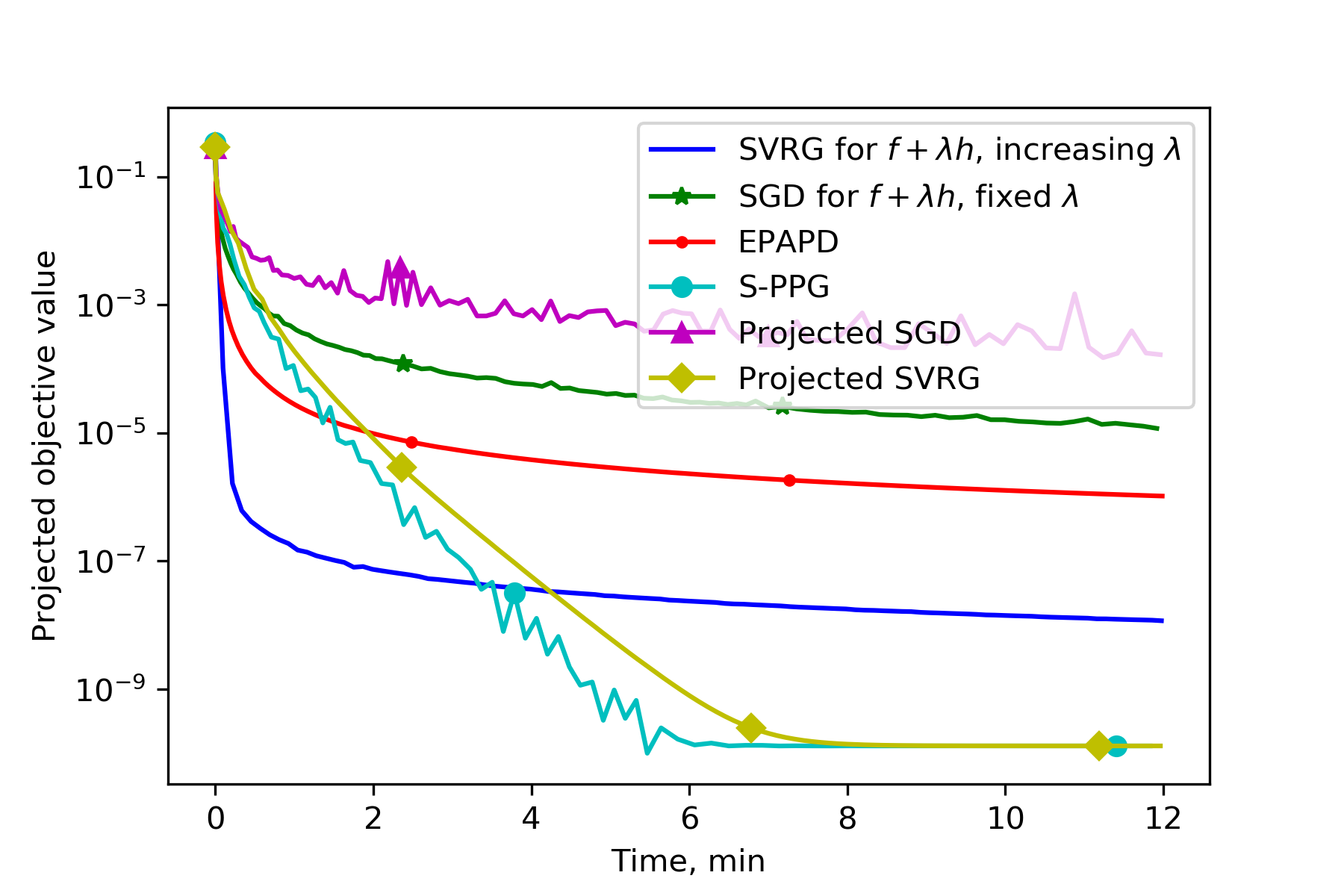}}
      & \raisebox{-\totalheight / 2}{\includegraphics[scale=0.30]{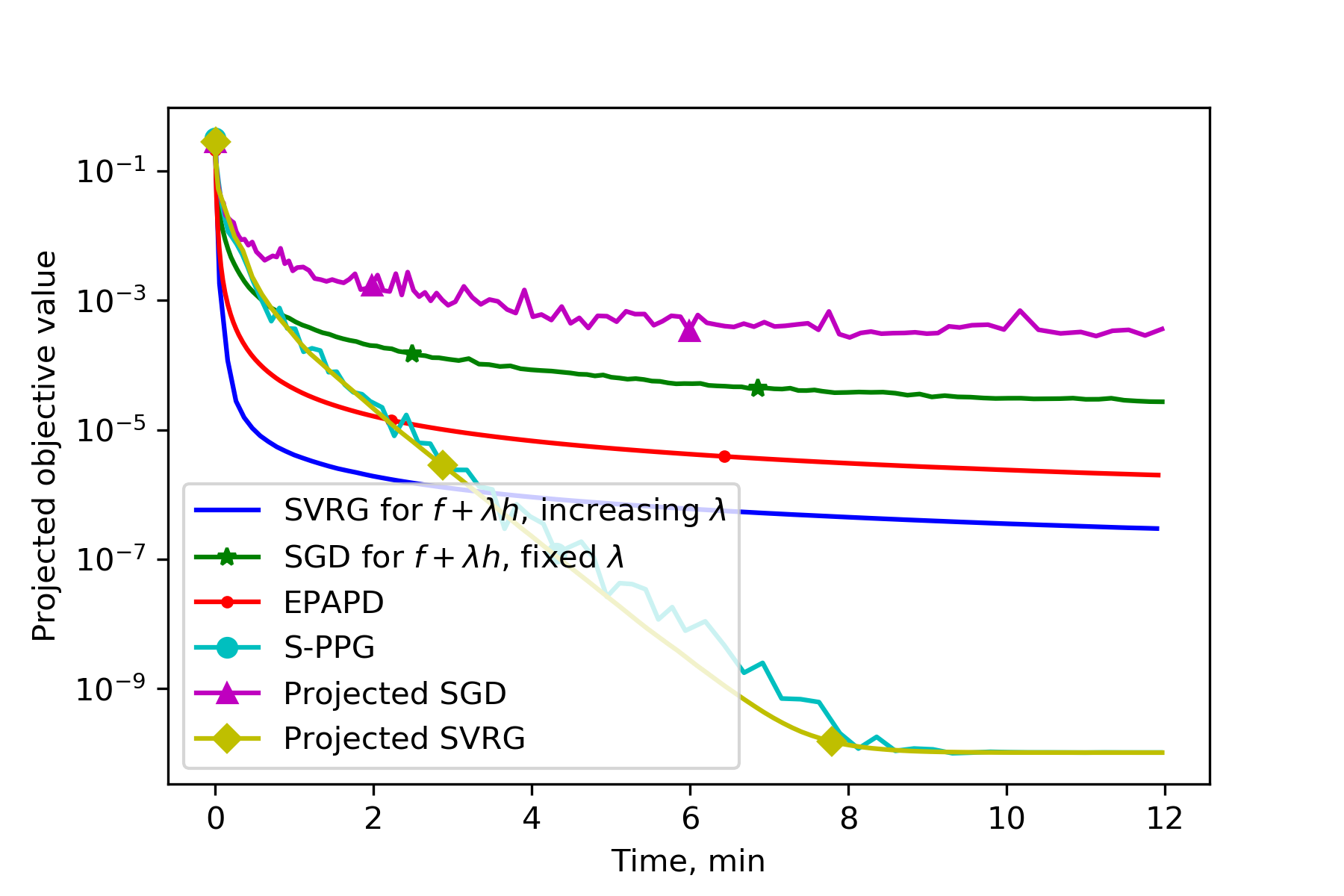}}
     \end{tabular}

  \begin{tabular}{c@{\quad}cccc}
  \hline
    & $m=40$  & $m=60$& $m=100$ \\
    \makecell{\tiny Distance from \\ \tiny the optimum}
      & \raisebox{-\totalheight / 2}{\includegraphics[scale=0.30]{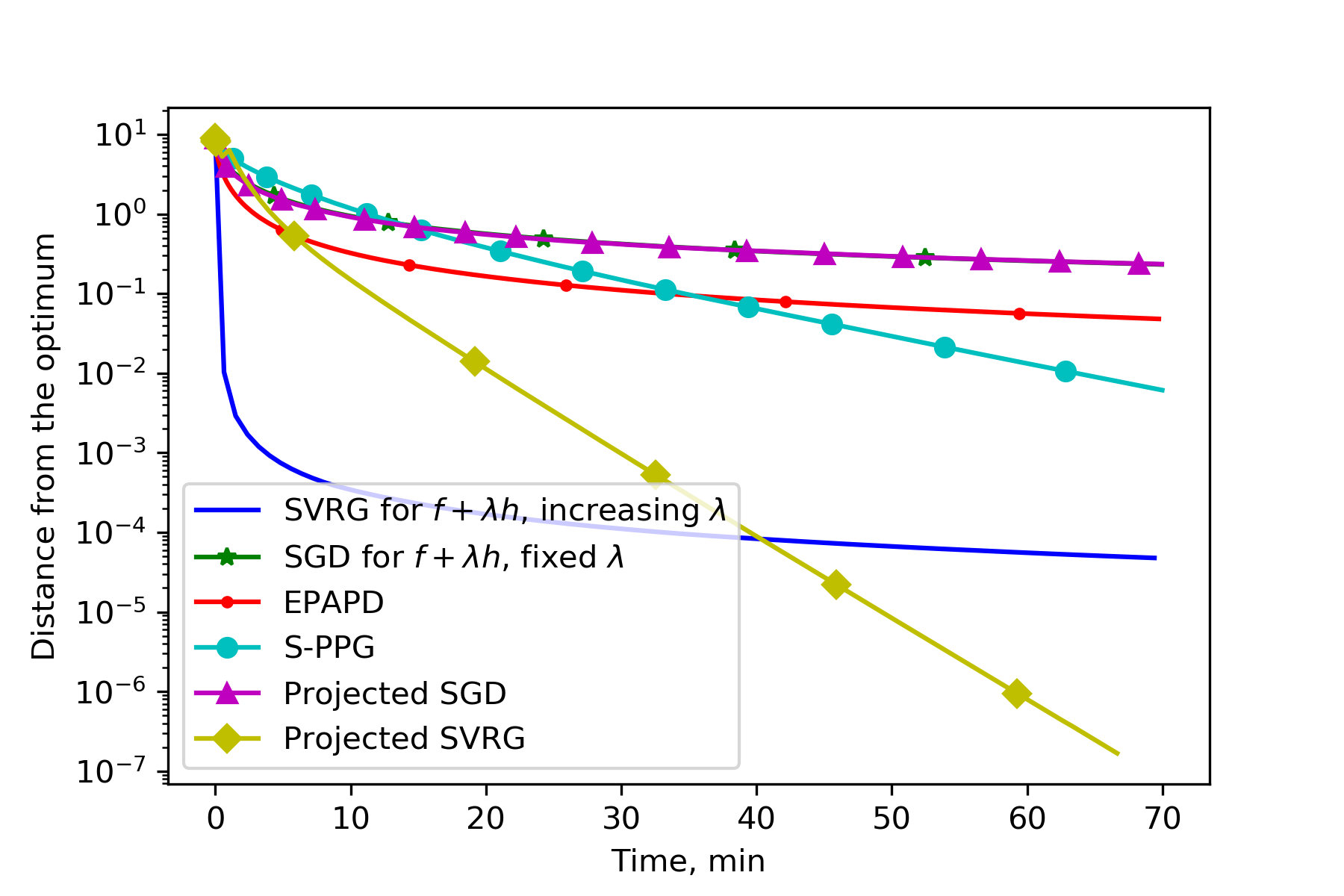}}
      & \raisebox{-\totalheight / 2}{\includegraphics[scale=0.30]{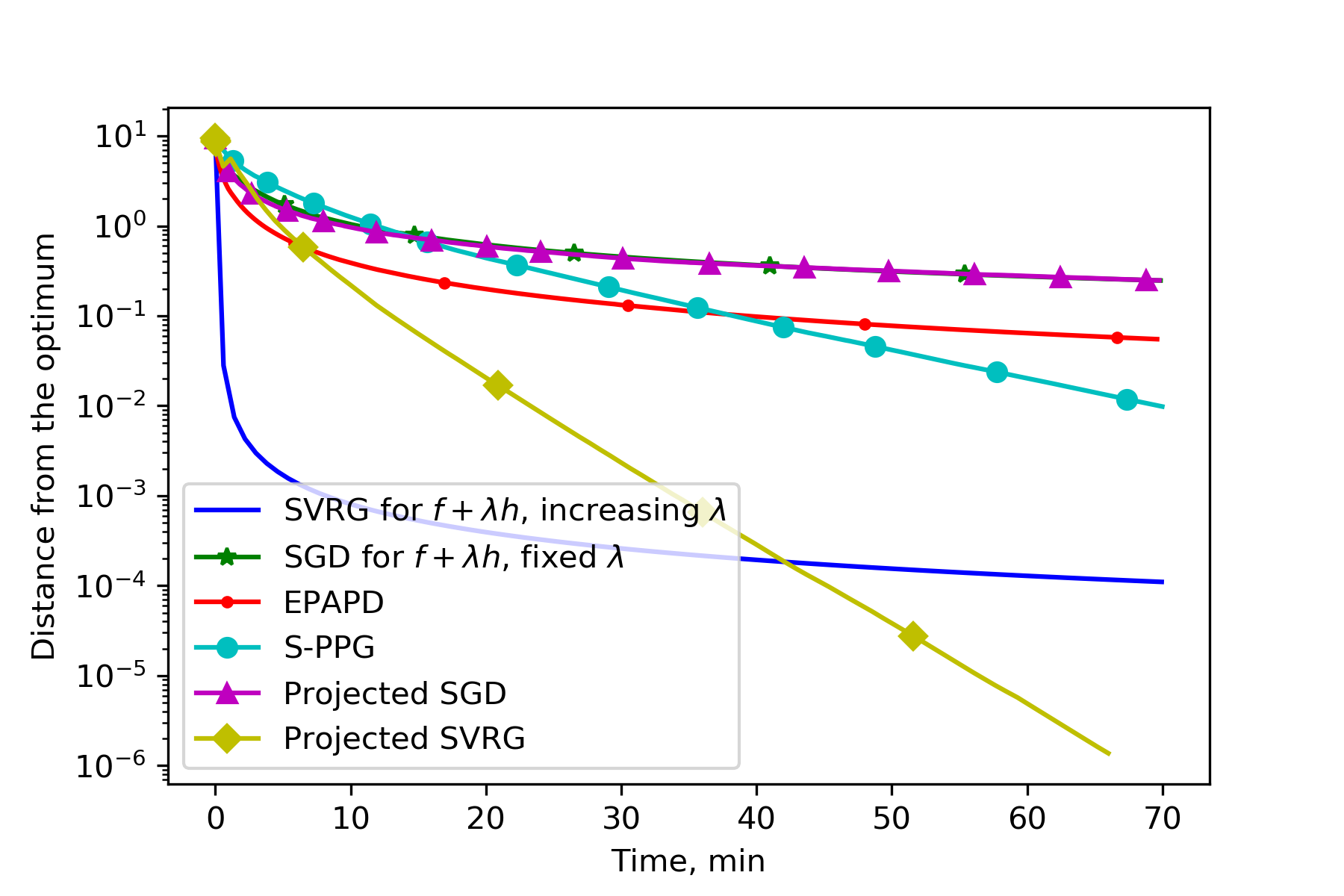}}
      & \raisebox{-\totalheight / 2}{\includegraphics[scale=0.30]{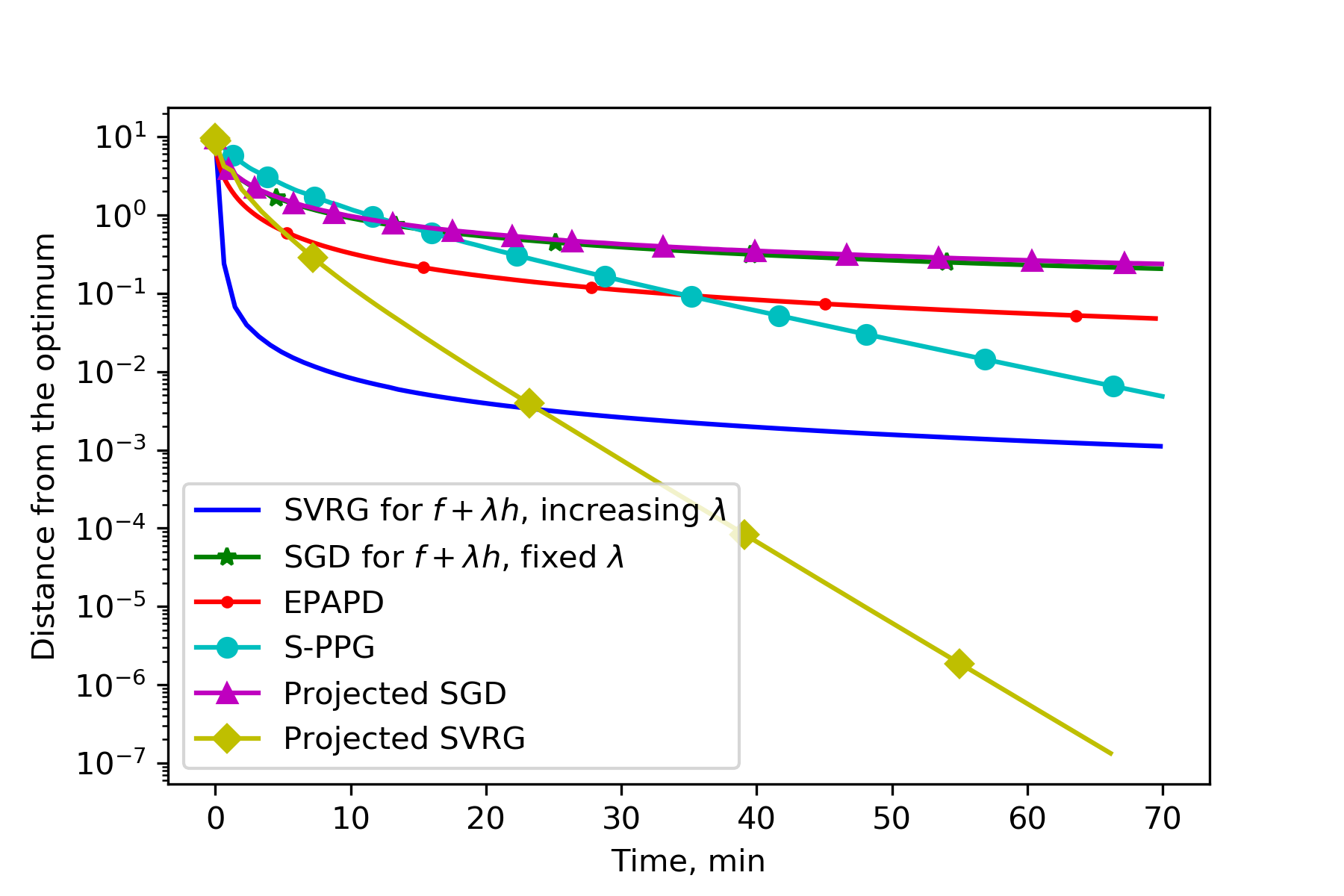}}\\ 
      
    \makecell{\tiny Objective value \\ \tiny of projected \\ \tiny iterates} 
      & \raisebox{-\totalheight / 2}{\includegraphics[scale=0.30]{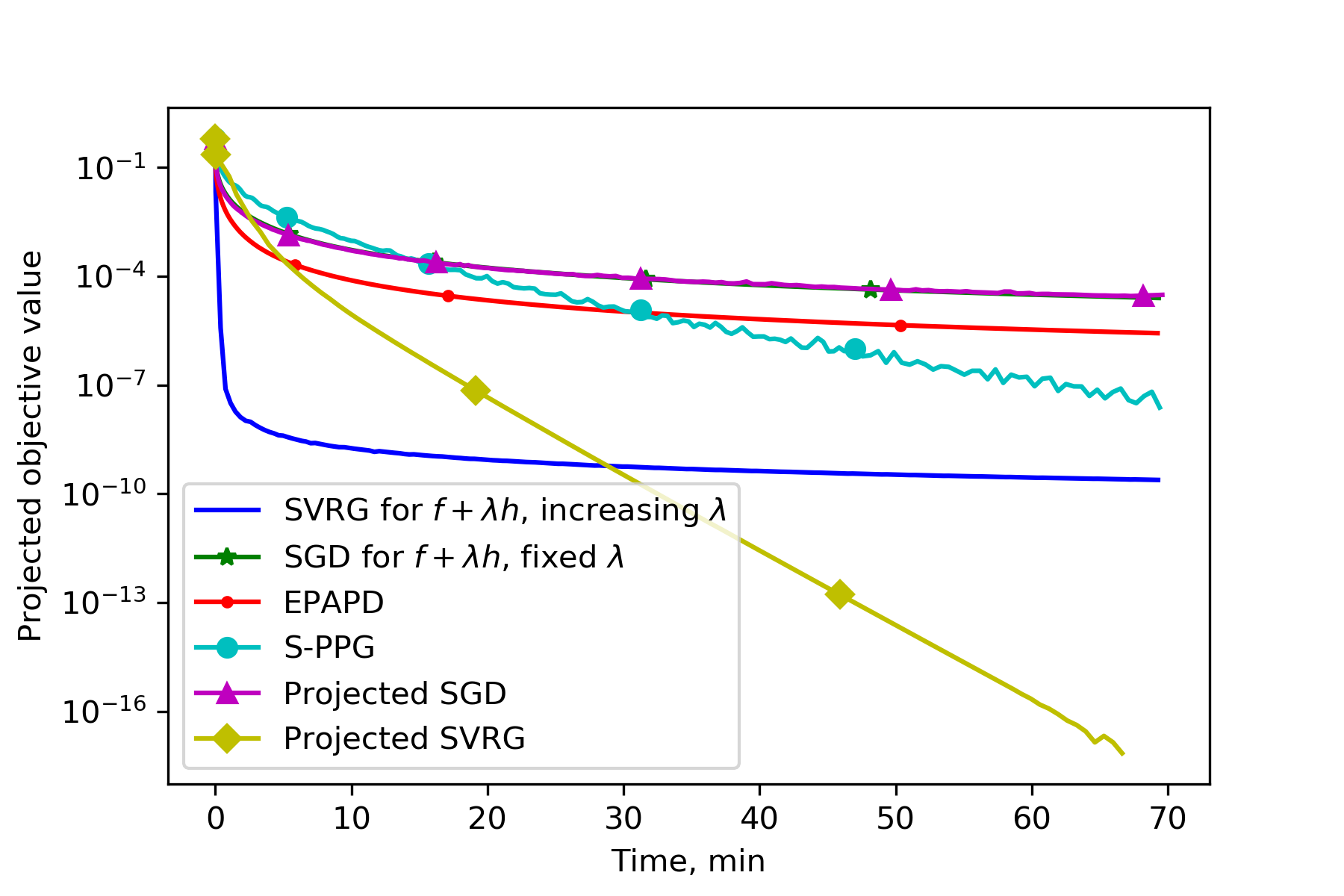}}
      & \raisebox{-\totalheight / 2}{\includegraphics[scale=0.30]{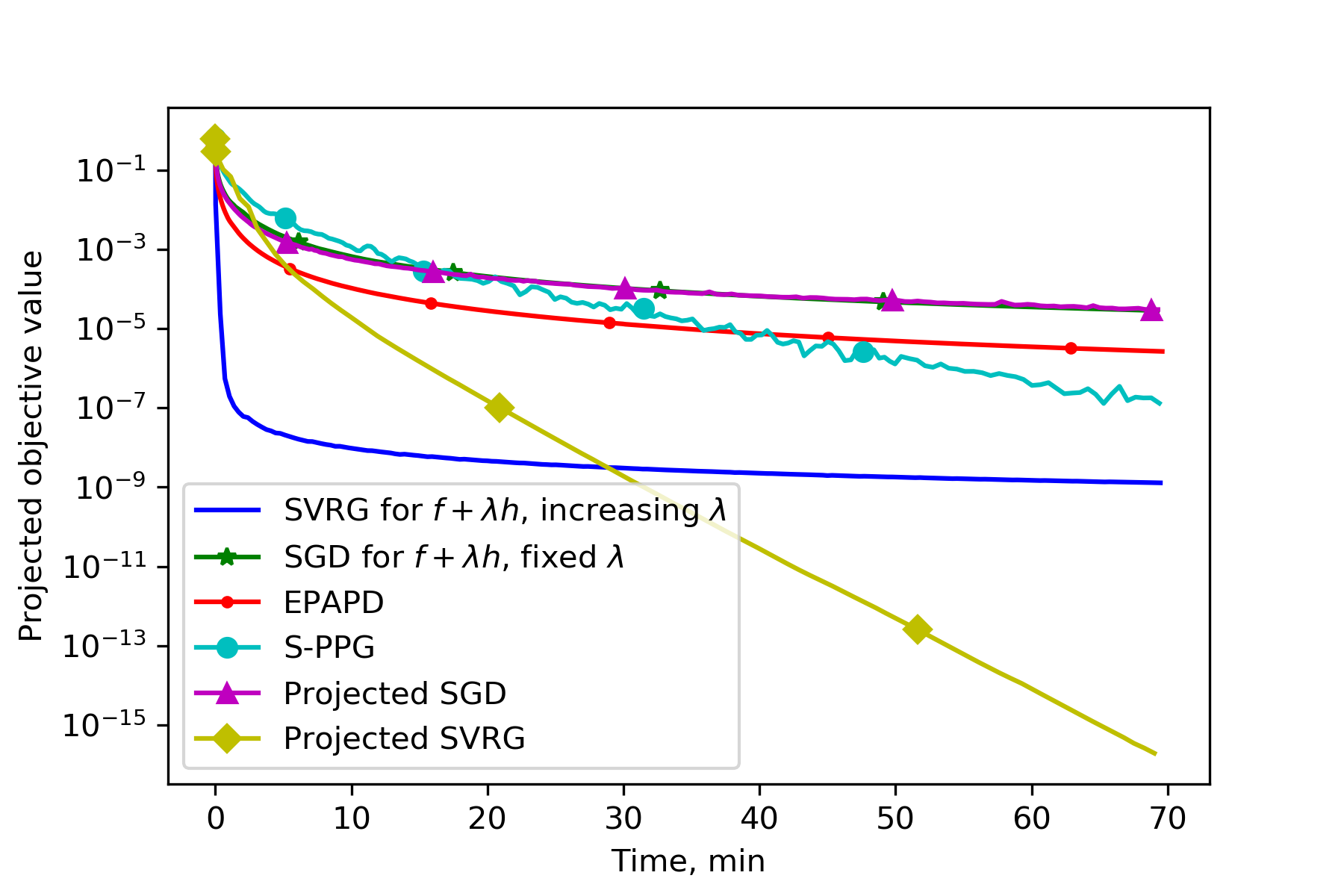}}
      & \raisebox{-\totalheight / 2}{\includegraphics[scale=0.30]{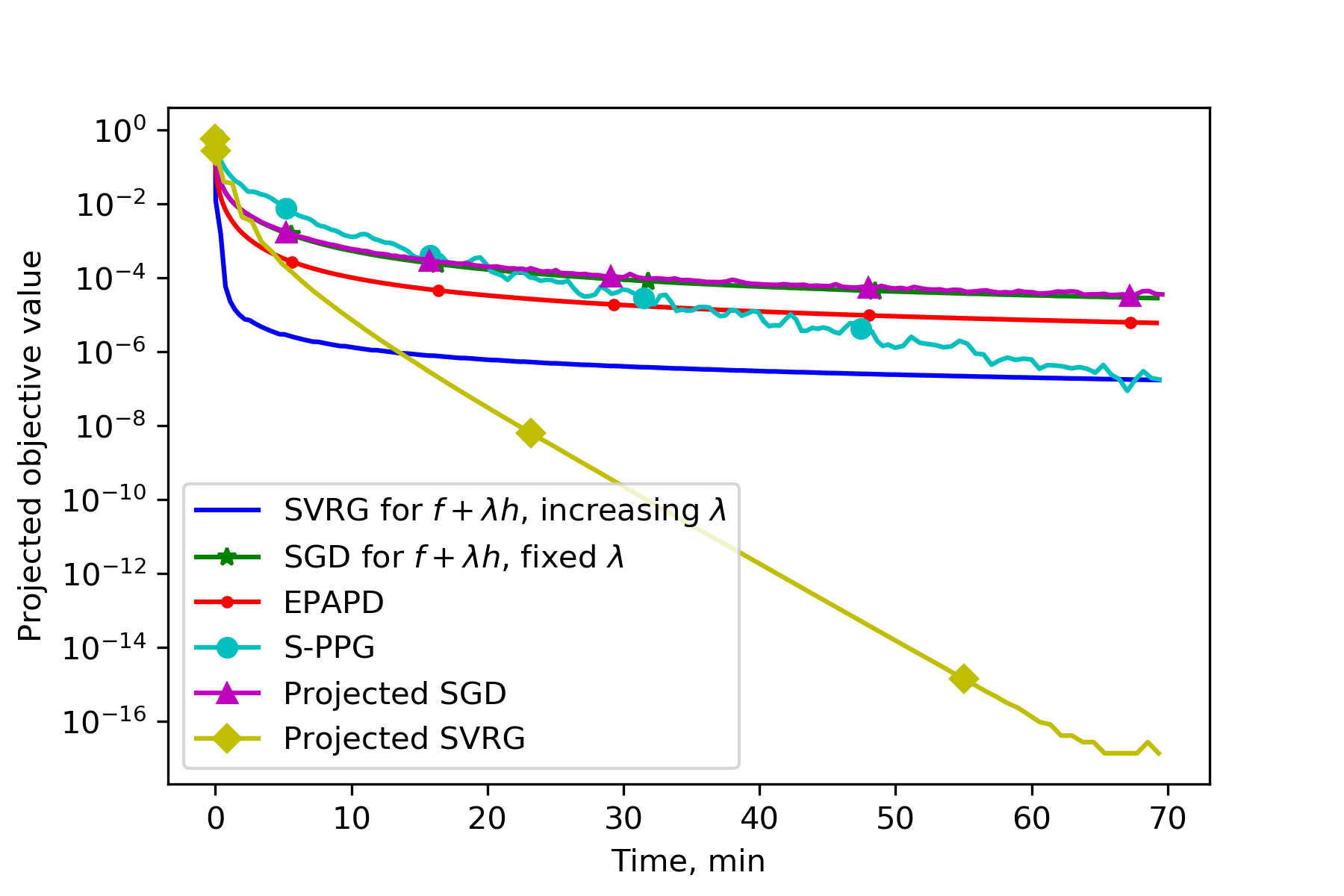}}
    \end{tabular}
   
  \begin{tabular}{c@{\quad}ccc}
      \hline
   &  $m=100$  & $m=200$& $m=400$ \\
    \makecell{\tiny Distance from \\ \tiny the optimum}
      & \raisebox{-\totalheight / 2}{\includegraphics[scale=0.30]{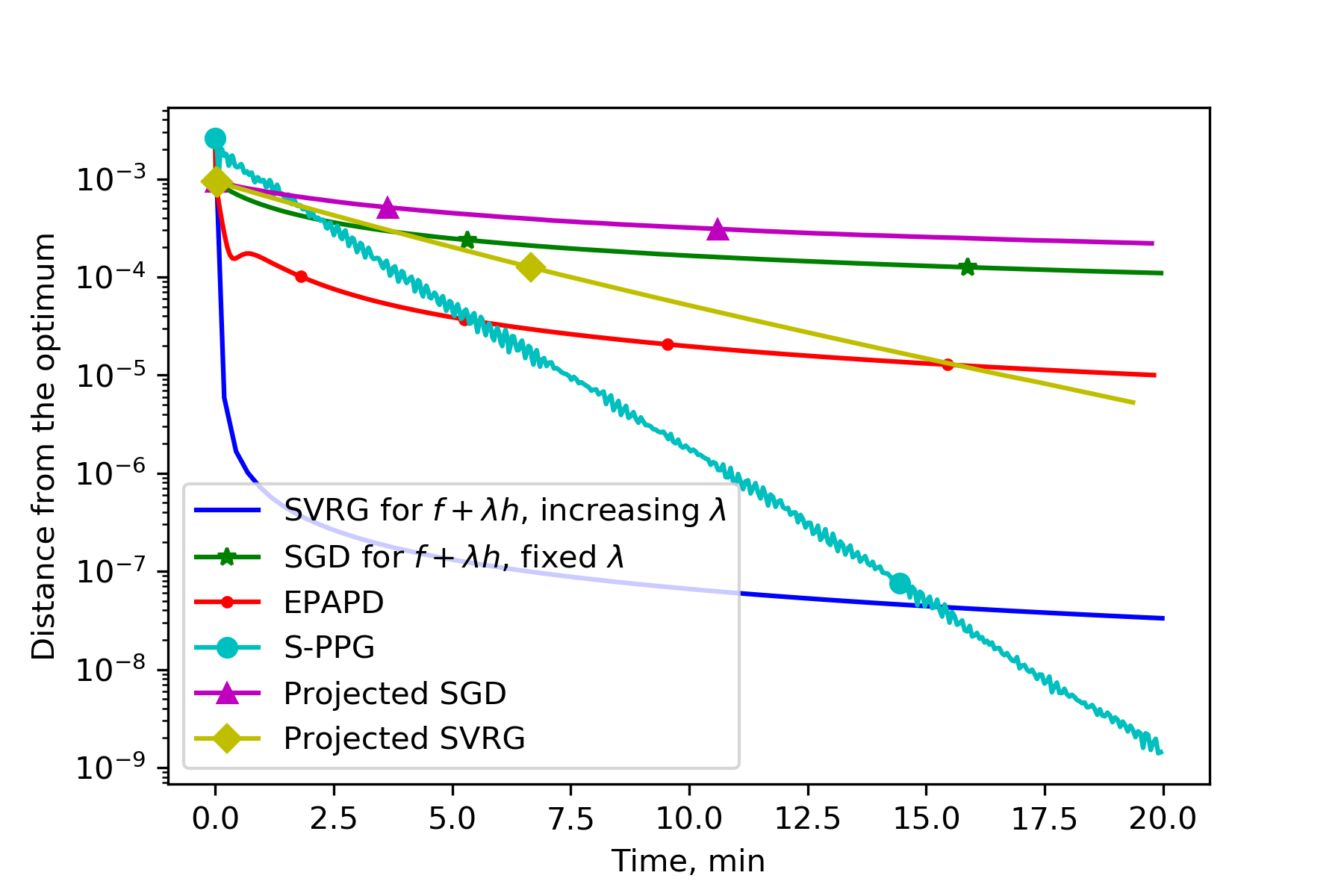}}
      & \raisebox{-\totalheight / 2}{\includegraphics[scale=0.30]{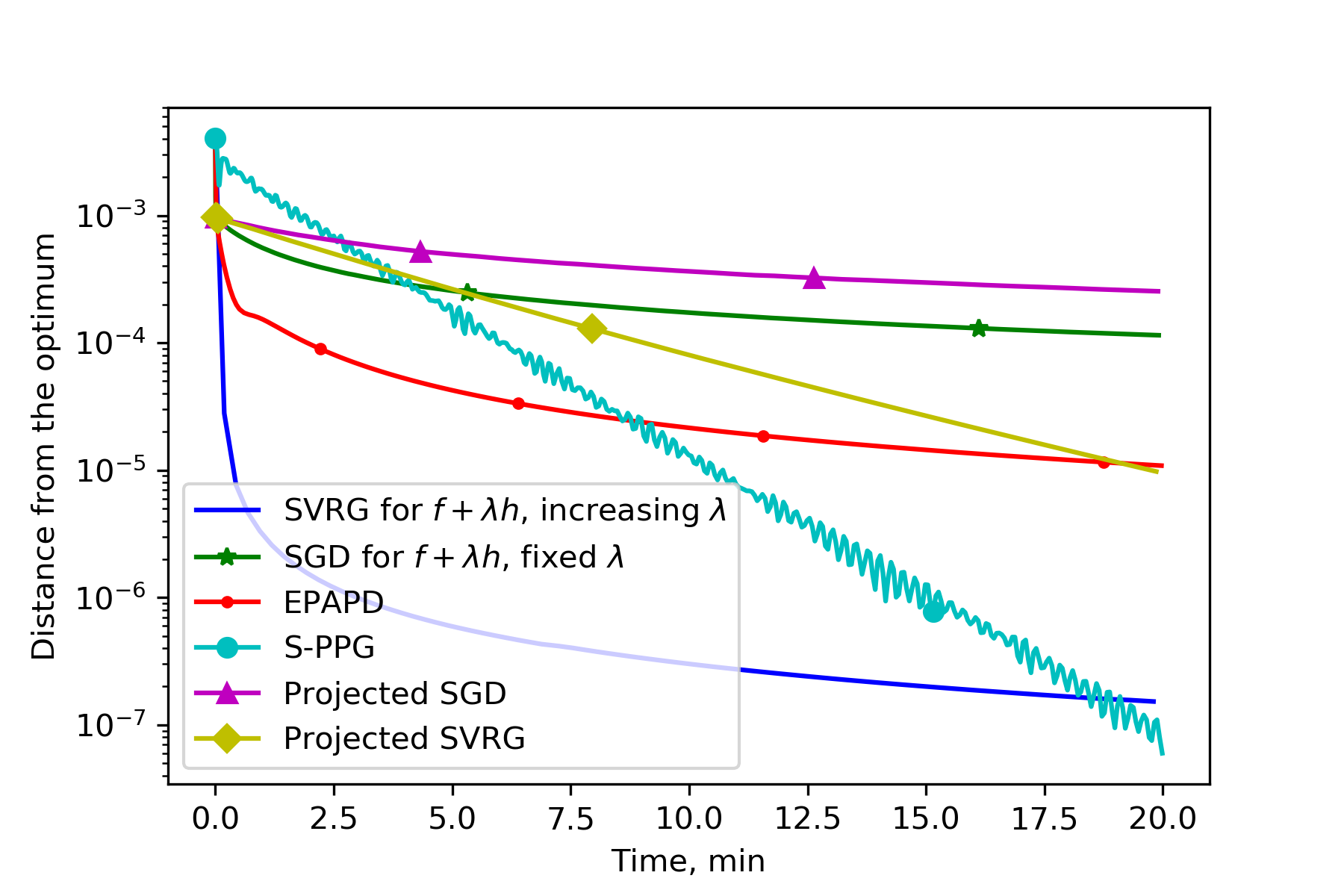}}
      & \raisebox{-\totalheight / 2}{\includegraphics[scale=0.30]{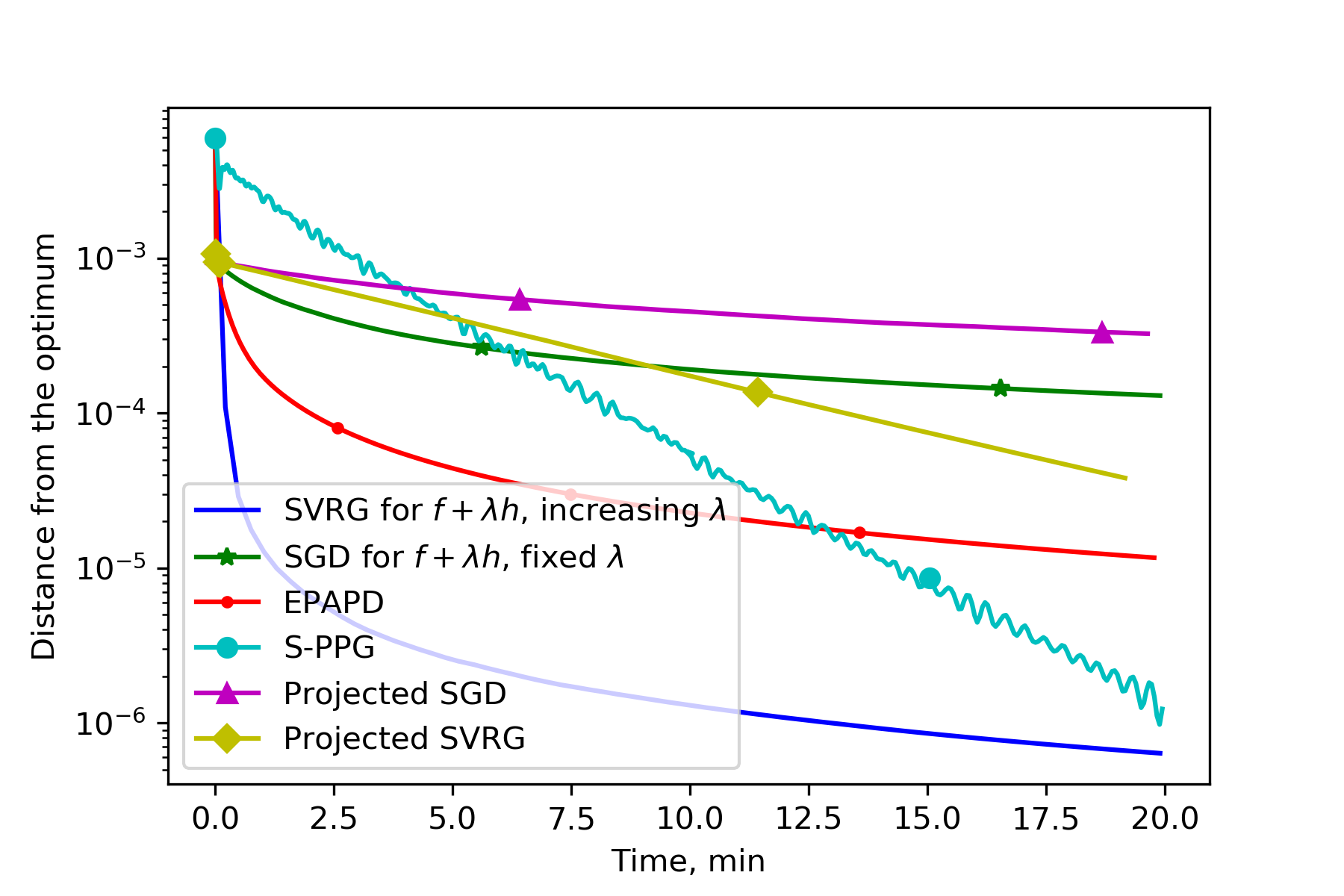}}\\ 
      
    \makecell{\tiny Objective value \\ \tiny of projected \\ \tiny iterates} 
      & \raisebox{-\totalheight / 2}{\includegraphics[scale=0.30]{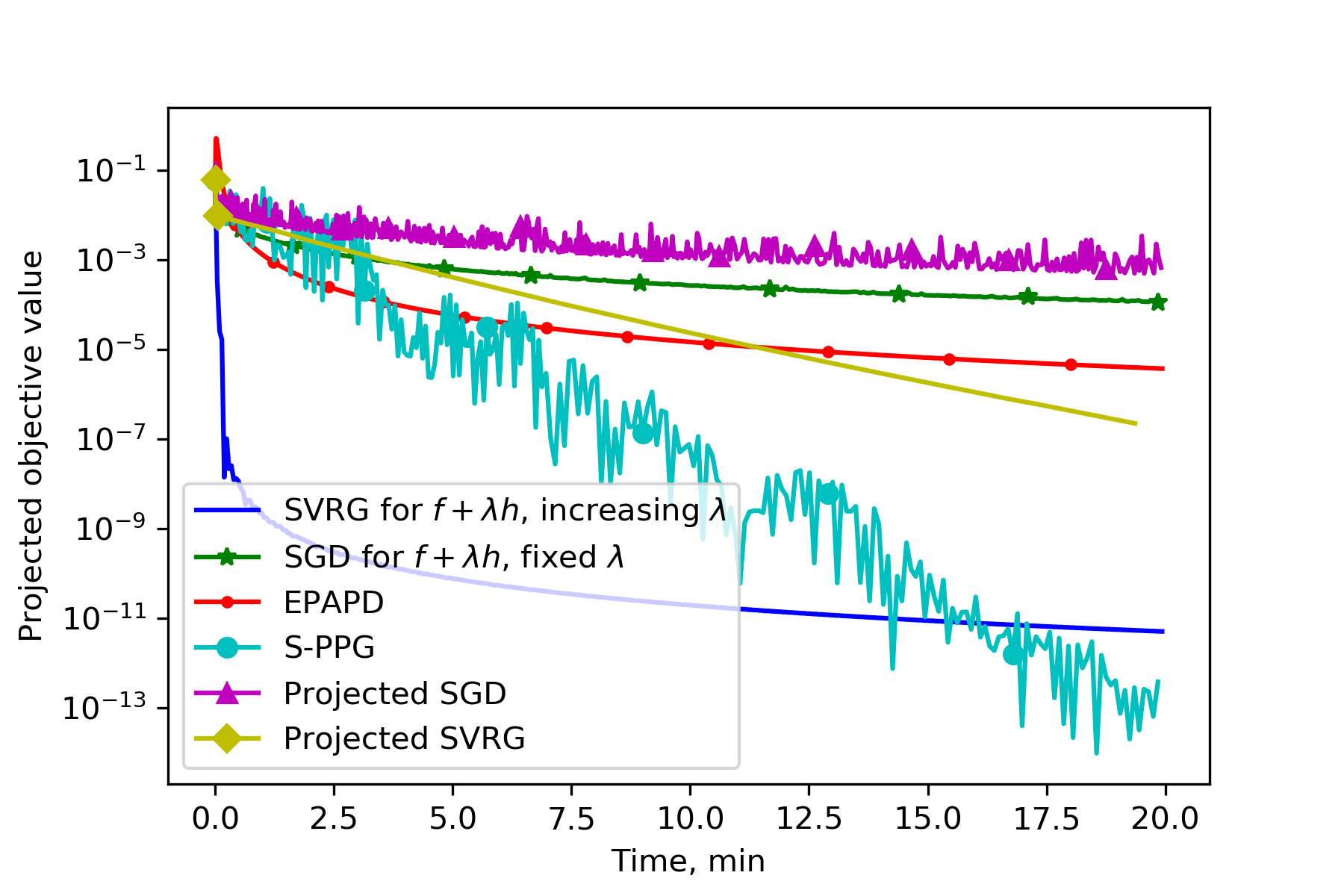}}
      & \raisebox{-\totalheight / 2}{\includegraphics[scale=0.30]{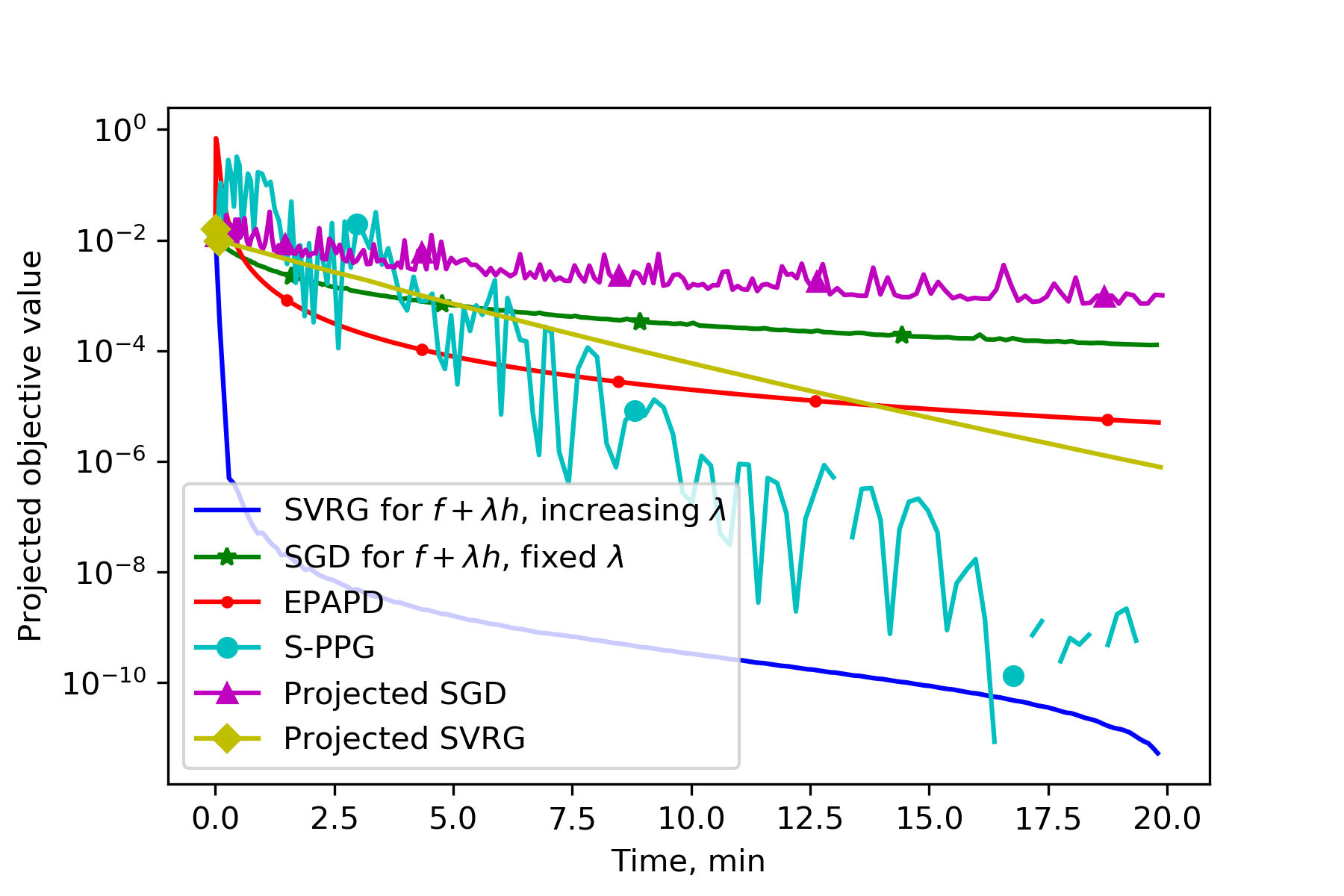}}
      & \raisebox{-\totalheight / 2}{\includegraphics[scale=0.30]{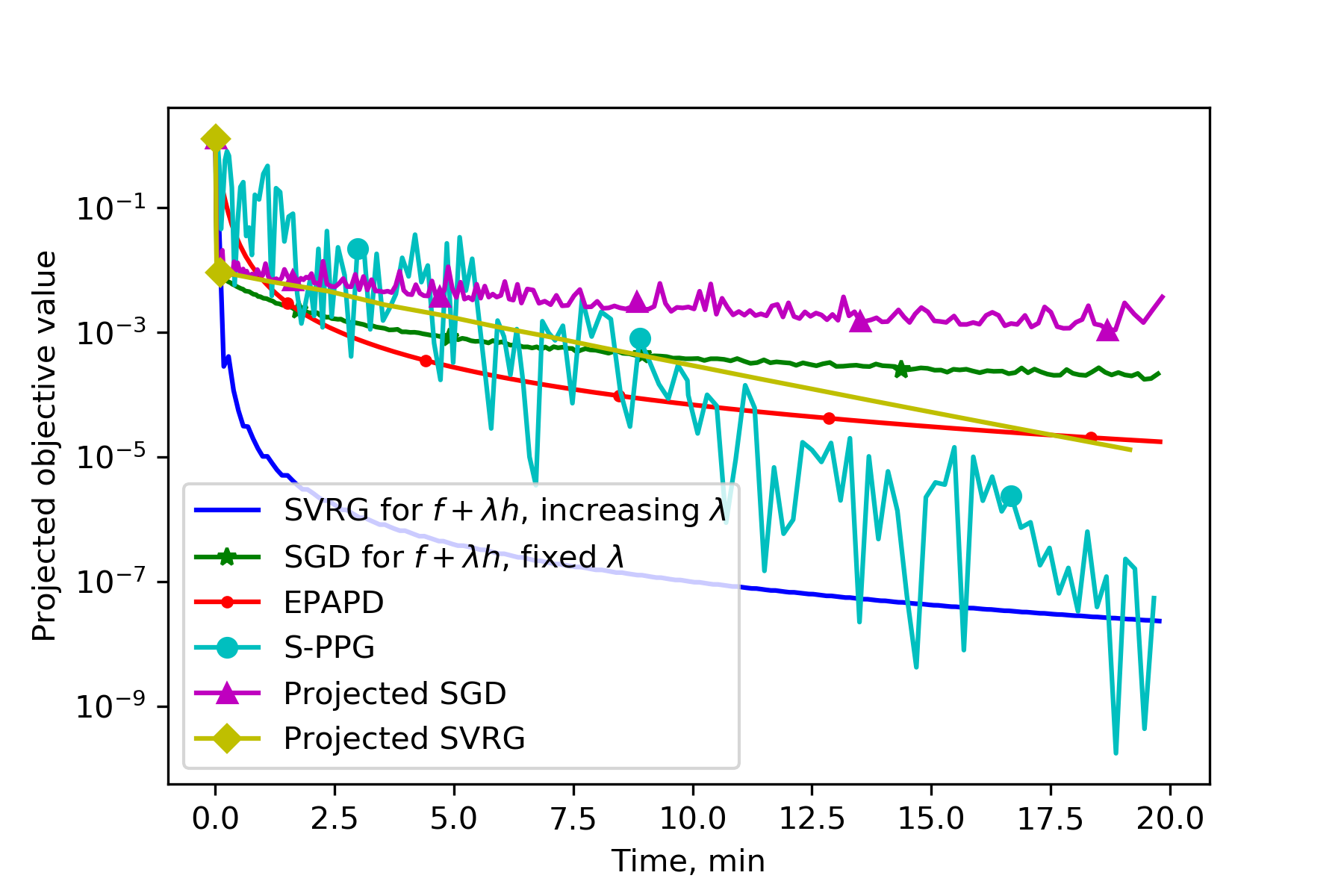}}
  \end{tabular}
  \caption{Experiments on logistic regression with $m$ randomly generated constraints, half of which are linear equality and half linear inequality constraints. Tha datasets used: A1a (top 2 rows), Mushrooms (middle 2 rows) and Madelon (bottom 2 rows). For each dataset and each each method we plot (squared) distance from the optimum ($\|x^k-x^*\|^2$) and  objective value of projected iterates ($f(\Pi_{\cX}(x^k))-f(x^*)$), where $\{x^k\}$ are the iterates produced by each method. }\label{tab:logreg}
  \end{center}
\end{table*}


\clearpage

\bibliographystyle{plain}
\bibliography{penalty-arxiv}

\clearpage
\onecolumn
\appendix

\part*{Supplementary Material}
            
\section{Basic facts and notation}

The Cauchy-Shwarz inequality states that for any $a, b \in \mathbb{R}^d$, 
\begin{align}
	\langle a, b \rangle \le \|a\|\|b\|. \label{eq:cauchy-schwarz}
\end{align}

\begin{proposition}
	If $f$ is $L$-smooth, than for any $x$ and $y$
    \begin{align}
    	|f(y) - f(x) + \langle \nabla f(x), x - y \rangle| &\le \frac{L}{2}\|x - y\|^2, \label{eq:smooth_bregman}\\
        \|\nabla f(x) - \nabla f(y)\| &\le L\|x - y\| \label{eq:smooth_norms}.
    \end{align}
    If $f$ is also convex, then
    \begin{align}
        \|\nabla f(x) - \nabla f(y)\|^2 \le 2L(f(y) - f(x) - \langle \nabla f(x), y - x\rangle). \label{eq:smooth_bregman_lower}
    \end{align}
\end{proposition}

\begin{proposition}
	If $f$ is $\mu$-strongly convex, then for any $x$ and $y$
    \begin{align}
    \label{eq:str_convx_means_firmly_nonexpansive}
    	\mu\|x - y\|^2 \le \langle x - y, \nabla f(x) - \nabla f(y) \rangle.
    \end{align}
    
\end{proposition}

\begin{proposition}
	For any $u, v$ and a convex set $C$ it is satisfied that
    \begin{align}
    \label{eq:proj_is_firmly_nonexpansive}
    	\|\Pi_C(u) - \Pi_C(v) \|^2 + \|u - \Pi_C(u) - (v - \Pi_C(v))\|^2 \le \|u - v\|^2,
    \end{align}
    and, as a consequence,
    \begin{align}
    \label{eq:proj_is_nonexpansive}
    	\|\Pi_C(u) - \Pi_C(v) \|^2 \le \|u - v\|^2.
    \end{align}
\end{proposition}

\subsection{Optimality conditions}

The first order necessary optimality condition for \eqref{eq:pb} is
\begin{equation} \label{eq:ih098h0d909nos9(}
	\langle \nabla f(x^*), x- x^* \rangle \geq 0, \qquad \text{for all} \qquad x\in \cX, 
\end{equation}
where the $\langle \nabla f(z),t\rangle $ denotes the directional derivative of $f$ at $z$ in direction $t$. If $f$ is convex, then this condition is also sufficient. A similar necessary condition is that
\begin{align}
\label{eq:x_star_is_a_stationary_point}
	x^* = \PP(x^* - \omega \nabla f(x^*))
\end{align}
for any $\omega \ge 0$.

Finally, the first order necessary condition for $x_\lambda^*$ is 
\begin{align}
\label{eq:first_order_for_relaxed_pb}
    	\nabla (f(x_\lambda^*) + \lambda h(x_\lambda^*)) = \nabla f(x_\lambda^*) + \lambda \meanj (x_\lambda^* - \PPj (x_\lambda^*)) = 0.
\end{align}
			
\section{Exact Theory under No Assumptions}\label{sec:proofs_nof_assump}

We will often write $f^* = f(x^*)$, $f^*_\lambda = f(x^*_\lambda)$, and  $h^*_\lambda = h(x^*_\lambda)$.

\begin{customlem}{\ref{lem:bu98gd08g0s}}
For all $\lambda \geq 0$ we have
  \begin{equation} \label{eq:opt_values_monotonicity}
    f_\lambda^*  +	\lambda h_\lambda^* 
       \leq f^*. 
    \end{equation}
Moreover,    
	for any $\lambda \ge \theta \ge 0$ we have $f^*\ge f_\lambda^* \ge f_\theta^*$ and $0\leq h_\lambda^* \le h_\theta^*$. 
\end{customlem}

\begin{proof}
Note that $x^*\in \cX$ and hence $h(x^*)=0$.	Since  $x_\lambda^*$ minimizes $f + \lambda h$ and $h_\lambda^*\ge 0$,  we get 
\[
    	f^* = f(x^*) = f(x^*) + \lambda h(x^*) \ge f_\lambda^* + \lambda h_\lambda^* \ge f_\lambda^*,
\]
which implies \eqref{eq:opt_values_monotonicity} and $f^* \geq f^*_\lambda$.    Further,	from the definitions of $x_\lambda^*$ and $x_\theta^*$ we see that
    \begin{align}
    \label{eq:lambda_theta_opt}
    	f(x_\lambda^*) + \lambda h(x_\lambda^*) 
        &\le f(x_\theta^*) + \lambda h(x_\theta^*),\\
        f(x_\theta^*) + \theta h(x_\theta^*) 
        &\le f(x_\lambda^*) + \theta h(x_\lambda^*) \nonumber.
    \end{align}
    Multiplying the first inequality by $\theta$, the second inequality by $\lambda$, and adding them up, we obtain
    \begin{align*}
    	\theta f(x_\lambda^*) + \lambda f(x_\theta^*) \le \theta f(x_\theta^*) + \lambda f(x_\lambda^*),
    \end{align*}
    which implies that $f^*_\lambda \geq f^*_\theta$. The final statement follows  immediately from \eqref{eq:lambda_theta_opt}.
\end{proof}

\section{Exact Theory under $L$-smoothness} \label{sec:exact_L_smooth_proofs}

\begin{lemma}
\label{lem:f_of_projection}
	Let $f$ be $L$-smooth, choose $y\in \R^d$ and set $x = \PP(y)$. Then for all $\lambda > \frac{L}{\gamma}$,
	\begin{align}
	\label{eq:f_of_projection1}
		f(x) \le f(y) + \lambda h(y) + \frac{L}{\gamma \lambda - L}(f(y) - f_0^*).
	\end{align}

\end{lemma}

\begin{proof}
In view of \eqref{eq:smooth_bregman}, $L$-smoothness of function $f$ implies
\begin{equation}\label{eq:89ys8ghs}
    	f(x) 
        \le f(y) + \langle \nabla f(y), x - y \rangle + \frac{L}{2}\|y - x\|^2.
\end{equation}
     Furthermore, from inequality $\langle a, b \rangle \le \frac{\alpha}{2}\|a\|^2 + \frac{1}{2\alpha}\|b\|^2$ (which holds for all vectors $a,b$ and $\alpha>0$) and linear regularity we deduce
     \begin{eqnarray}
     	f(x) 
        &\overset{\eqref{eq:89ys8ghs}}{\le} & f(y) + \frac{\alpha}{2}\| y - \PP (y) \|^2 + \frac{L}{2}\|y - \PP(y)\|^2 + \frac{1}{2\alpha}\| \nabla f(y) \|^2 \notag\\
       &\overset{\eqref{eq:lin_reg}}{\le} & f(y) + \frac{\alpha + L }{2\gamma} \meanj\|y - \PPj (y) \|^2 + \frac{1}{2\alpha}\| \nabla f(y) \|^2.\label{eq:bu8g7dff}
     \end{eqnarray}
     In addition, we can bound the last term  as follows:
\begin{equation}   \label{eq:bui87gf7gbf}       	\|\nabla f(y)\|^2 \le 2L(f(y) - f_0^*).
\end{equation}
Indeed, 
     \begin{eqnarray*}
     	f^*_0 \quad \eqdef \quad f(x^*_0) &=& \arg \min_{x} f(x)  \\
     	&\le & f(y - \frac{1}{L}\nabla f(y))  \\
     	&\stackrel{\eqref{eq:89ys8ghs}}{\le}& f(y) + \langle \nabla f(y), y - \frac{1}{L} \nabla f(y) - y\rangle + \frac{L}{2}\|y - \frac{1}{L} \nabla f(y) - y\|^2\\
     	& =& f(y) - \frac{1}{2L}\|\nabla f(y)\|^2.
     \end{eqnarray*}
Combining \eqref{eq:bu8g7dff} and \eqref{eq:bui87gf7gbf}, we get
     \begin{align*}
     	f(x) 
       &\le f(y) + \frac{\alpha + L }{\gamma}h(y) + \frac{L}{\alpha}(f(y) - f_0^*).
     \end{align*}
	It suffices  to plug $\alpha = \gamma\lambda - L$ into this inequality.

\end{proof}

\begin{customthm}{\ref{thm:optimal_obj_values}}
	Assume that $f$ is $L$-smooth. Then 
\begin{itemize}
\item[(i)]	For all $\lambda\ge 2\frac{L}{\gamma}$, the $f(\PP(x_\lambda^*)) $ and $f(x^*)$ are related via
\begin{equation}\label{eq:nb989dbffs}
    	f^* \leq f(\PP(x_\lambda^*)) 
        \le f_\lambda^* + \lambda h_\lambda^* + \frac{2L}{\gamma \lambda}(f^* - f_0^*) \le f^* + \frac{2L}{\gamma \lambda}(f^* - f_0^*) 
\end{equation}
\item[(ii)] For all $\lambda\ge 2\frac{L}{\gamma}$ (lower bound) and all $\lambda\geq 0$ (upper bound), $Opt_\lambda$ and $f(x^*)$ are related via
    \begin{align}
    \label{eq:f_star_minus_fh_lambda_star}
    	  f^* - \frac{2L}{\gamma \lambda}(f^* - f_0^*)
          \le f_\lambda^* + \lambda h_\lambda^* 
          \le f^* - \frac{\|\nabla f(x^*)\|^2}{2(L + \lambda)}.
    \end{align}
    The upper bound in \eqref{eq:f_star_minus_fh_lambda_star} holds for all $\lambda\geq 0$.
 \item[(iii)]   
For all $\lambda>0$, the value $h(x^*_\lambda)$ is bounded as
    \begin{align}
    	\frac{\gamma G}{L^2 + \gamma \lambda^2} &\le h_\lambda^* \le \frac{f^* - f_0^*}{\lambda},\label{eq:h_lambda_star_bounds}
    \end{align}
    where 
    \begin{align}
    \label{eq:def_of_G}
    	G \eqdef \frac{1}{4}\inf_{x\in \cX}\|\nabla f(x)\|^2.
    \end{align}
\item[(iv)]      The distance of $x^*_\lambda$ to $\cX$ is for all $\lambda>0$ ($\lambda$ can be 0 in the lower bound) bounded by
\begin{equation}\label{eq:h08h00fff}
    	\frac{2 G}{L^2 + \lambda^2}\le \|x_\lambda^* - \PP(x^*_\lambda)\|^2 \leq \frac{2(f(x^*) - f(x^*_0))}{\gamma \lambda}.
\end{equation}
\item[(v)]  	The distance of $x^*_\lambda$ to the optimal solution $x^*$ cannot be too small. In particular, for all $\lambda \geq 0$,
    \begin{equation}\label{eq:nb98fg8fds}
      \frac{\|\nabla f(x^*)\|^2}{(L + \lambda)^2} \leq 	\|x_\lambda^* - x^*\|^2.
      \end{equation}
\end{itemize}
\end{customthm}

\begin{proof}
\begin{itemize}
\item[(i)]
The first inequality in \eqref{eq:nb989dbffs} follows  since $\Pi_{\cX}(x^*_\lambda)$ is feasible for \eqref{eq:pb}. 
Since $\gamma \lambda - L \ge \gamma \frac{1}{2}\lambda$, we have $\lambda > L/\gamma$ and we can therefore apply Lemma~\ref{lem:f_of_projection}. Using it with $x=\Pi_{\cX}(x^*_\lambda)$ and $y=x^*_\lambda$, we obtain
    \[
     	f(\PP(x_\lambda^*)) 
        \le f(x_\lambda^*) + \lambda h(x_\lambda^*) + \frac{L}{\gamma \lambda - L}(f_\lambda^* - f_0^*).
    \]
The second inequality follows by observing that  $f_\lambda^*\le f^*$, and by replacing $\gamma \lambda - L$ with $ \frac{1}{2}\lambda$. To prove the third inequality, it is enough to use  \eqref{eq:opt_values_monotonicity}.

\item[(ii)] The lower bound in \eqref{eq:f_star_minus_fh_lambda_star} follows from \eqref{eq:nb989dbffs}. To prove the upper bound, let us consider the point $z = x^* - \omega\nabla f(x^*)$. From $L$-smoothness of $f$ we get
	\begin{eqnarray}
	    f(z) &\stackrel{\eqref{eq:smooth_bregman}}{\le} & f(x^*) + \langle \nabla f(x^*), z - x^* \rangle + \frac{L}{2}\|z - x^*\|^2 \notag\\
	    & = & f^* - \omega \|\nabla f(x^*) \|^2 + \frac{L}{2}\omega^2\|\nabla f(x^*)\|^2.\label{eq:nbi87gf8}
	\end{eqnarray}
	Moreover, since $x^*\in \cX$, we get $\|z - \PPj(z)\| \le \|z - \PP(z)\| \le \|z - x^*\| = \omega\|\nabla f(x^*)\|$, we  get
	\begin{equation} \label{eq:nbg9f8f} h(z) \overset{\eqref{eq:min_h}+\eqref{eq:h_j} }{=} \frac{1}{2m}\sum_{j=1}^m \|z -\PPj(z)  \|^2 \leq \frac{\omega^2}{2}\|\nabla f(x^*)\|^2, \end{equation}
	and hence
	\begin{align*}
	   f(x^*_\lambda) + \lambda h(x^*_\lambda) \leq  f(z) + \lambda h(z) \overset{\eqref{eq:nbi87gf8}+\eqref{eq:nbg9f8f}}{\le} f^* - \omega \|\nabla f(x^*) \|^2 + \frac{L}{2}\omega^2\|\nabla f(x^*)\|^2 + \frac{\lambda}{2}\omega^2\|\nabla f(x^*)\|^2,
	\end{align*}
After	plugging $\omega = (L + \lambda)^{-1}$, we obtain the right-hand side of \eqref{eq:f_star_minus_fh_lambda_star}. 
	
\item[(iii)]		The upper bound in  \eqref{eq:h_lambda_star_bounds} follows directly from \eqref{eq:opt_values_monotonicity} and the fact that  $f_0^*\le f_\lambda^*$.     To get the lower bound in \eqref{eq:h_lambda_star_bounds}, let us use inequality $\|a\|^2\ge \frac{1}{2} \|b\|^2 - \|a - b\|^2$ and smoothness of $f$ to write
    \begin{eqnarray}
    	\lambda^2 \left\|x_\lambda^* - \meanj \PPj(x_\lambda^*) \right\|^2   
    	&    \stackrel{\eqref{eq:first_order_for_relaxed_pb}}{=} & \|\nabla f(x_\lambda^*)\|^2 \nonumber \\
       & \geq & \frac{1}{2}\|\nabla f(\PP(x_\lambda^*))\|^2 - \|\nabla f(x_\lambda^*) - \nabla f(\PP(x_\lambda^*))\|^2 \nonumber\\
       &  \stackrel{\eqref{eq:def_of_G} + \eqref{eq:smooth_bregman}}{\geq} &  2G - L^2\|x_\lambda^* - \PP (x_\lambda^*)\|^2 \label{eq:first_step_in_proving_low_bd_on_hlam} \\
        & \overset{\eqref{eq:lin_reg}}{\geq} & 2G - \frac{L^2}{\gamma}\meanj\|x_\lambda^* - \PPj (x_\lambda^*)\|^2. \label{eq:nb98g9dd}
    \end{eqnarray}
    Rearranging \eqref{eq:nb98g9dd} and applying Jensen's inequality, we  deduce
 \[
    	2G 
        \le \left(\frac{L^2}{\gamma} + \lambda^2 \right)\meanj\|x_\lambda^* - \PPj (x_\lambda^*)\|^2 
        =  2\left(\frac{L^2}{\gamma} + \lambda^2 \right) h_\lambda^*,
\]
which gives the lower bound in \eqref{eq:h_lambda_star_bounds}.

\item[(iv)]
Recall that by \eqref{eq:first_step_in_proving_low_bd_on_hlam},
\begin{align*}
    2G \le 	\lambda^2 \left\|x_\lambda^* - \meanj \PPj(x_\lambda^*) \right\|^2 + L^2\|x_\lambda^* - \PP (x_\lambda^*)\|^2.
\end{align*}
Clearly, we can bound the first term via Jensen's inequality as
\begin{align*}
    \left\|x_\lambda^* - \meanj \PPj(x_\lambda^*) \right\|^2 \le \meanj \left\|x_\lambda^* - \PPj(x_\lambda^*) \right\|^2 \le \left\|x_\lambda^* -  \PP(x_\lambda^*) \right\|^2,
\end{align*}
which gives the lower bound in \eqref{eq:h08h00fff}. The upper bound in \eqref{eq:h08h00fff} follows by combining the upper bound in
\eqref{eq:h_lambda_star_bounds} with linear regularity \eqref{eq:lin_reg}.

\item[(v)]  Since $f$ is $L$-smooth and $h$ is 1-smooth,  $F_\lambda\coloneqq f + \lambda h$ is  $(L + \lambda)$-smooth. Therefore,
    \begin{eqnarray*}
        (L + \lambda)^2\|x^* - x_\lambda^*\|^2 &\overset{\eqref{eq:smooth_bregman}}{\ge} & \|\nabla F_\lambda (x^*) - \underbrace{\nabla F_\lambda (x_\lambda^*)}_{=0}\|^2 = \|\nabla f(x^*)\|^2,
    \end{eqnarray*}
    where the last identity also uses the fact that $h(x^*)=0$.
\end{itemize}
\end{proof}

\section{Exact Theory under Convexity} \label{sec:proofs_convex}

\begin{lemma}
\label{lem:f_star_minus_f_upper}
	If $f$ is $\mu$-strongly convex, then for any $x\in \R^d$,
	\begin{align}
    	f(x) \ge f^* -  \|\nabla f(x^*)\| \sqrt{\frac{2}{\gamma} h(x) } + \frac{\mu}{2}\|x - x^*\|^2.\label{eq:f_star_minus_f_lambda_upper_strcvx}
    \end{align}
	In particular, if $f$ is just convex ( i.e.\ $\mu = 0$), then
    \begin{align}
    	f(x) \ge f^* -  \|\nabla f(x^*)\| \sqrt{\frac{2}{\gamma} h(x) }.\label{eq:f_star_minus_f_lambda_upper}
    \end{align}
\end{lemma}

\begin{proof}
	Strong convexity of $f$ yields
	\begin{eqnarray*}
f(x) & \geq & f(x^*) + \langle \nabla f(x^*), x - x^* \rangle  + \frac{\mu}{2}\|x - x^*\|^2 \\
	& = & f(x^*) + \langle \nabla f(x^*), \PP(x) - x^* \rangle  + \langle \nabla f(x^*), x - \PP(x)  \rangle  + \frac{\mu}{2}\|x - x^*\|^2 \\
	&\overset{\eqref{eq:ih098h0d909nos9(}  }{\geq} & f(x^*) + \langle \nabla f(x^*), x - \PP(x)  \rangle  + \frac{\mu}{2}\|x - x^*\|^2 \\
	&\overset{\eqref{eq:cauchy-schwarz}}{\geq} & f(x^*) - \|\nabla f(x^*)\| \|x - \PP(x)\|  + \frac{\mu}{2}\|x - x^*\|^2 \\
	&\overset{\eqref{eq:lin_reg} }{\ge}&f(x^*) - \|\nabla f(x^*)\| \sqrt{\frac{2}{\gamma}h(x)}  + \frac{\mu}{2}\|x - x^*\|^2.
	\end{eqnarray*}
\end{proof}

\begin{customthm}{\ref{thm:u98sgbd}}
If $f$ is differentiable and convex, then for all $\lambda>0$, we have

\begin{itemize}
\item[(i)] The values $f(x^*_\lambda)$ and $f(x^*)$ are related via
    \begin{align}\label{eq:0h0f9h9ff} 
    0\leq f^* - f_\lambda^* \leq  \frac{2}{\gamma \lambda}  \|\nabla f(x^*)\|^2 . 
    \end{align}
\item[(ii)]     $h(x^*_\lambda)$ is bounded above by     
\begin{equation}\label{eq:hx98g8gd}
        h^*_\lambda \leq  \frac{2 \|\nabla f(x^*)\|^2}{\gamma \lambda^2}.  
        \end{equation}

\item[(iii)]    The distance of $x^*_\lambda$ to $\cX$ is bounded above by
    \begin{align}
    \label{eq:upp_bound_on_dist_to_projection}
     	\|x_\lambda^* - \PP(x^*_\lambda)\|^2 \leq \frac{4\|\nabla f(x^*)\|^2}{\gamma^2 \lambda^2}.
    \end{align}
    
\item[(iv)]    	If $f$ is $L$-smooth and if $f+ \lambda h$ is  $\mu$-strongly convex with $\mu>0$ (for this it suffices for $f$ to be $\mu$-strongly convex), then the distance of $x^*_\lambda$ from $x^*$ is bounded above by
\[
      \|x^*_\lambda - x^*\|^2 \leq \frac{L + \lambda - \gamma\lambda}{\gamma\mu\lambda(L + \lambda)}\|\nabla f(x^*)\|^2.
\]

    In particular, if $f$ is differentiable but not $L$-smooth  (i.e., $L=+\infty$), then the bound simplifies to
    \begin{equation}\label{eq:nb98fg8f}
        \|x^*_\lambda - x^*\|^2 \leq \frac{1}{\gamma\mu\lambda}\|\nabla f(x^*)\|^2.
    \end{equation}
\end{itemize}
        
\end{customthm}

\begin{proof}
 
\begin{itemize}
\item[(i)] 	 By Lemma~\ref{lem:f_star_minus_f_upper}, convexity of $f$ gives
\begin{eqnarray*}
f(x^*_\lambda) 
&\overset{\eqref{eq:f_star_minus_f_lambda_upper} }{\geq}&f(x^*) - \|\nabla f(x^*)\| \sqrt{\frac{2}{\gamma}h(x^*_\lambda)}\\
&\overset{\eqref{eq:opt_values_monotonicity}}{\geq} &  f(x^*) - \|\nabla f(x^*)\| \sqrt{\frac{2}{\gamma \lambda} (f(x^*) - f(x^*_\lambda)) }.
\end{eqnarray*}
Rearranging the resulting inequality gives \eqref{eq:0h0f9h9ff}.
\item[(ii)]
Next,
\[ 
h^*_\lambda \enskip \overset{\eqref{eq:opt_values_monotonicity}}{\leq } \enskip \frac{f^* - f^*_\lambda}{\lambda}
\enskip \overset{\eqref{eq:0h0f9h9ff} }{\leq} \enskip \frac{2}{\gamma \lambda^2}  \|\nabla f(x^*)\|^2. 
\]
\item[(iii)] It suffices to combine linear regularity \eqref{eq:lin_reg} with \eqref{eq:hx98g8gd}.

\item[(iv)] Using Lemma~\ref{lem:f_star_minus_f_upper}, we get
    \begin{eqnarray*}
        f(x^*_\lambda)
        &\overset{\eqref{eq:f_star_minus_f_lambda_upper_strcvx}}{\geq} & f(x^*) - \|\nabla f(x^*)\| \sqrt{\frac{2}{\gamma}h(x^*_\lambda)} + \frac{\mu}{2}\|x_\lambda^* - x^*\|^2\\
        &\overset{\eqref{eq:f_star_minus_fh_lambda_star}}{\geq} &  f(x^*) - \|\nabla f(x^*)\| \sqrt{\frac{2}{\gamma \lambda} \left( f(x^*) - f(x^*_\lambda) - \frac{\|\nabla f(x^*)\|^2}{2(L + \lambda)} \right) }  + \frac{\mu}{2}\|x_\lambda^* - x^*\|^2.
    \end{eqnarray*}
    Denote $r_\lambda = f^* - f_\lambda^*  - \frac{\|\nabla f(x^*)\|^2}{2(L + \lambda)}$. Then the inequality above is equivalent to
    \begin{align*}
        \frac{\mu}{2}\|x_\lambda^* - x^*\|^2 
        &\le \|\nabla f(x^*)\| \sqrt{\frac{2}{\gamma \lambda} r_\lambda } - (f^* - f_\lambda^*) \\
        &= \|\nabla f(x^*)\| \sqrt{\frac{2}{\gamma \lambda} r_\lambda } - \left( f^* - f_\lambda^* - \frac{\|\nabla f(x^*)\|^2}{2(L + \lambda)} \right) - \frac{\|\nabla f(x^*)\|^2}{2(L + \lambda)}\\
        &= \|\nabla f(x^*)\| \sqrt{\frac{2}{\gamma \lambda} r_\lambda } - r_\lambda - \frac{\|\nabla f(x^*)\|^2}{2(L + \lambda)}\\
        &= - \left(\sqrt{r_\lambda} - \|\nabla f(x^*)\|\sqrt{\frac{1}{2\gamma\lambda}} \right)^2 + \frac{\|\nabla f(x^*)\|^2}{2\gamma\lambda} -  \frac{\|\nabla f(x^*)\|^2}{2(L + \lambda)}\\
        &\le \frac{L + \lambda - \gamma\lambda}{2\gamma\lambda(L + \lambda)}\|\nabla f(x^*)\|^2.
    \end{align*}
    The result \eqref{eq:nb98fg8f} follows by taking the limit $L\rightarrow +\infty$.
\end{itemize}
\end{proof}

The following result follows immediately by combining Lemma~\ref{lem:bu98gd08g0s} and Equations \eqref{eq:0h0f9h9ff}  and \eqref{eq:hx98g8gd} from Theorem~\ref{thm:u98sgbd}.

\begin{remark}
    If one also denotes $z \coloneqq x^* - \frac{1}{L + \lambda}\nabla f(x^*)$ and $\widetilde\gamma \eqdef \frac{m^{-1}\sum_j \|z - \PPj(z)\|^2}{\|z - \PP(z)\|^2}\le 1$, then \eqref{eq:nb98fg8f} can be improved to
    \begin{align*}
        \frac{\mu}{2}\|x_\lambda^* - x^*\|^2  \le \frac{L + \widetilde\gamma\lambda - \gamma\lambda}{2\gamma\lambda(L + \widetilde\gamma\lambda)}\|\nabla f(x^*)\|^2.
    \end{align*}
\end{remark}

\section{Counterexample for Missing Lower Bounds} \label{sec:1dexample}

    Consider the following example: $x\in \mathbb{R}^d$, $d=1$, $f(x) = \tfrac{L}{2} x^2$, $\cX = [1,2]$. Then, $x^* = 1$ and $L/2=f(x^*)\ge f(x_\lambda^*)=(L/2)(x_\lambda^*)^2$, so $\left|x_\lambda^*\right|\le 1$ and $\PP(x_\lambda^*) = 1 = x^*$.   Thus,
    \begin{align*}
        \|\PP(x_\lambda^*) - x^*\|^2 = 0, \qquad \text{and} \qquad f(\PP(x_\lambda^*)) - f^* = 0.
    \end{align*}
    Note also that $\|\nabla f(x^*)\| >0$. 	The considered function $f$ is $L$-smooth and $L$-strongly convex, so with these assumptions we can not hope for nontrivial (i.e., strictly positive) lower bounds on $\|\PP(x_\lambda^*) - x^*\|^2$ and $f(\PP(x_\lambda^*)) - f^*$. 

\section{Algorithm Complexities: SGD} \label{sec:SGD_proof}

The proposition below was proved in \cite{pmlr-v80-nguyen18c}.
\begin{proposition}[Convergence of SGD]
\label{pr:sgd_convergence}
	Assume that a function $F(x)\equiv \EE_\xi [F(x;\xi)]$ is $\mu$-strongly convex and for every $\xi$ is $L$-smooth. Let sequence $\{x^k\}_k$ be generated by the update rule
    \begin{align*}
    	x^{k+1} = x^k - \omega_k \nabla F(x^k; \xi^k),
    \end{align*}
    where $\omega_k = \frac{\alpha}{2\alpha L + k} \le \frac{1}{2L}$ with some $\alpha > \mu^{-1}$, $\EE_{\xi^k} [\nabla F(x^k; \xi^k)] = \nabla F(x^k)$. Then, for $x_F^*\coloneqq \argmin F(x)$ one has
    \begin{align*}
    	\EE \|x^k - x_F^*\|^2 \le \frac{M}{2\alpha L + k},
    \end{align*}
    where $M = \max(I, J)$, and
    \begin{align*}
    	I &= (2\alpha L + 1)\left[\left(1 - \frac{\mu}{2L}\right)\|x^0 - x_F^*\|^2 + \frac{N}{4L^2}\right],\\
        J &= \frac{\alpha^2 N}{\alpha\mu - 1},\\
        N &= 2\EE[\|\nabla F(x_F^*;\xi)\|^2] \leq \frac{4}{n}\Sum \|\nabla f_i(x_\lambda^*)\|^2 + \frac{16 \|\nabla f(x^*)\|^2}{\gamma}.
    \end{align*}
\end{proposition}

\begin{theorem}
	Assume that $f$ is $\mu$-strongly convex and $L$-smooth, and apply the update rule
    \begin{align*}
    	x^{k+1} = x^k - \omega_k (\nabla f_i(x^k) + \lambda\nabla h_j(x^k)),
    \end{align*}
    where $\omega_k = \frac{\alpha}{2\alpha (L + \lambda) + k}$, $i$ and $j$ are sampled uniformly from $\{1,\dotsc, n\}$ and $\{1, \dotsc, m\}$ at each iteration independently. Then, to guarantee $\EE\|x^k - x_\lambda^*\|^2 \le \frac{c}{\lambda}$ with some constant $c > 0$ we need no more than
    \begin{align*}
    	\frac{2\lambda \alpha (L + \lambda) + 1}{c}\left(2\|x^0 - x^*\|^2 + \frac{2\|\nabla f(x^*)\|^2}{\gamma\mu \lambda} + \frac{N}{4(L + \lambda)^2}\right) + \frac{\lambda \alpha^2 N}{c(\alpha\mu - 1)} = O(\lambda^2)
    \end{align*}
    iterations.
\end{theorem}

\begin{proof}
	Let us apply Proposition \ref{pr:sgd_convergence} to the function $F = f + \lambda h$ with $\nabla F(x^k; \xi) = \nabla f_i(x^k) + \lambda \nabla h_j(x^k)$. Clearly, $\EE[\nabla f_i(x^k) + \lambda \nabla h_j(x^k)] = \nabla f(x^k) + \lambda \nabla h(x^k)$ and
    \begin{eqnarray*}
    	N 
        &=& 2\EE[\|\nabla (f_i + \lambda h_j)(x_\lambda^*)\|^2] \\
        &\le & 4 \EE[\|\nabla f_i(x_\lambda^*)\|^2] + 4\lambda^2\EE [\|\nabla h_j(x_\lambda^*)\|^2] \\
        &=&  \frac{4}{n}\Sum \|\nabla f_i(x_\lambda^*)\|^2 + 4\frac{\lambda^2}{m}\Sumj \|x_\lambda^* - \PPj(x_\lambda^*)\|^2\\
        &=&  \frac{4}{n}\Sum \|\nabla f_i(x_\lambda^*)\|^2 + 8 \lambda^2h(x_\lambda^*)\\
        & \overset{\eqref{eq:hx98g8gd}}{\le} & \frac{4}{n}\Sum \|\nabla f_i(x_\lambda^*)\|^2 + \frac{16 \|\nabla f(x^*)\|^2}{\gamma}. 
    \end{eqnarray*}
    For any $k\ge M\lambda/c - 2\alpha(L + \lambda)$ we have
    \begin{align*}
    	\EE\|x^k - x_\lambda^*\|^2 
        \le \frac{M}{2\alpha (L + \lambda) + k }
        \le \frac{M}{2\alpha(L + \lambda) + M\lambda/c - 2\alpha(L + \lambda) }
        = \frac{c}{\lambda}.
    \end{align*}
    Moreover, 
    \begin{eqnarray*}
    	\frac{M\lambda}{c} - 2\alpha(L + \lambda) 
        &\le& \frac{\lambda(I + J)}{c} \\
        &\le& \frac{\lambda (2\alpha (L + \lambda) + 1)}{c}\left(\|x^0 - x_\lambda^*\|^2 + \frac{N}{4(L + \lambda)^2}\right) + \frac{\lambda \alpha^2 N}{c(\alpha\mu - 1)}\\
        &\le& \frac{\lambda (2\alpha (L + \lambda) + 1)}{c}\left(2\|x^0 - x^*\|^2 + 2\|x_\lambda^* - x^*\|^2 + \frac{N}{4(L + \lambda)^2}\right) + \frac{\lambda \alpha^2 N}{c(\alpha\mu - 1)}\\
        &\overset{\eqref{eq:nb98fg8f}}{\le} & \frac{\lambda (2\alpha (L + \lambda) + 1)}{c}\left(2\|x^0 - x^*\|^2 + \frac{2\|\nabla f(x^*)\|^2}{\gamma\mu \lambda} \right) \\
        && \qquad +  \left(\frac{\lambda (2\alpha (L + \lambda) + 1)}{4c(L + \lambda)^2} + \frac{\lambda\alpha^2}{c(\alpha\mu - 1)}\right) N .
    \end{eqnarray*}
\end{proof}

\section{Increasing $\lambda$}

The lemma below provides an explanation of the last step of Algorithm \ref{alg:increasing_lambda}. More precisely, it measures how the functional suboptimality changes when we replace penalty $\lambda_k$ with $\lambda_{k+1}$ and do a step with average projection.
    
\begin{lemma}
	Let $x^{k+1}=y^k - \omega_k\nabla h(y^k)$ and $\omega_k$ be defined as in Algorithm \ref{alg:increasing_lambda}, $\lambda_{k}\ge (1 - \gamma)\lambda_{k+1}$ (for example $\lambda_k \sim k$), $f$ be $L$-smooth and convex. Then
	\begin{align*}
		f(x^{k+1}) + \lambda_{k+1} h(x^{k+1}) \le f(y^k) + \lambda_k h(y^k) + L\frac{\lambda_{k+1} - \lambda_k}{\gamma\lambda_{k+1}^2}(f(x^{k+1}) - f_0^*).
	\end{align*}
\end{lemma}
In particular, since $f^*_{\lambda_{k}} + \lambda_{k} h^*_{\lambda_{k}}\le f^*_{\lambda_{k+1}} + \lambda_{k+1} h_{\lambda_{k+1}}^*$ for $\lambda_{k+1}\ge \lambda_k$, we get
\begin{align}\label{eq:lem_incr_lambda}
    f(x^{k+1}) + \lambda_{k+1} h(x^{k+1}) - (f^*_{\lambda_{k+1}} + \lambda_{k+1} h^*_{\lambda_{k+1}})
    &\le f(y^k) + \lambda_k h(y^k) - (f^*_{\lambda_{k}} + \lambda_{k} h^*_{\lambda_{k}}) \nonumber\\
    &\qquad + L\frac{\lambda_{k+1} - \lambda_k}{\gamma\lambda_{k+1}^2}(f(x^{k+1}) - f_0^*).
\end{align}

\begin{proof}
	Since $\omega_{k+1} = \frac{\lambda_{k+1} - \lambda_k}{\gamma\lambda_{k+1}}$, the condition $\lambda_k\ge (1 - \gamma)\lambda_{k+1}$ implies $\omega_k\le 1$. Therefore, the left-hand side of~\eqref{eq:lem_incr_lambda} with Option~II will be always not greater than for Option~II, and it suffices to consider only the latter.
	
	From convexity and inequality $\langle a, b\rangle \le \frac{\alpha}{2}\|a\|^2 + \frac{1}{2\alpha}\|b\|^2$ we obtain
    \begin{align*}
    	f(x^{k+1}) 
        &\le f(y^k) + \langle \nabla f(x^{k+1}), x^{k+1} - y^k \rangle \\
        &= f(y^k) + \omega_k\meanj \langle \nabla f(x^{k+1}), \PPj(y^k) - y^k\rangle \\
        &\le f(y^k) + \frac{\alpha\omega_k}{2}\meanj \|\PPj(y^k) - y^k\|^2 + \frac{\omega_k}{2\alpha}\|\nabla f(x^{k+1})\|^2 \\
        &\le f(y^k) + \alpha\omega_k h(y^k) + \frac{L\omega_k}{\alpha}(f(x^{k+1}) - f_0^*).
    \end{align*}
    Recall that $\omega_k\le 1$, $h$ is 1-smooth and $\gamma$-Polyak-\L{}ojasiewicz \cite{necoara2018randomized}, from which it follows that
    \begin{align*}
    	h(x^{k+1}) 
        &\le h(y^k) + \langle \nabla h(y^k), x^{k+1} - y^k\rangle + \frac{1}{2}\|x^{k+1} - y^k\|^2 \\
        &= h(y^k) - \omega_k\| \nabla h(y^k)\|^2 + \frac{\omega_k^2}{2}\|\nabla h(y^k)\|^2 \\
        &= h(y^k) - \omega_k(1 - \frac{\omega_k}{2})\| \nabla h(y^k)\|^2 \\
        &\le h(y^k) - \gamma\omega_k(2 - \omega_k) h(y^k).
    \end{align*}
    
    Combining the two bounds above,
    \begin{align}
    \label{eq:f_of_average_projection}
    	f(x^{k+1}) + \lambda_{k+1} h(y^{k+1}) \le f(y^k) + (\alpha \omega_k + (1 - 2\gamma\omega_k + \gamma\omega_k^2) \lambda_{k+1}) h(y^k) + \frac{L\omega_k}{\alpha}(f(x^{k+1}) - f_0^*).
    \end{align}
    We want to get $f(y^k) + \lambda_k h(y^k)$ in the right hand side, so we write
    \begin{align*}
    	\alpha \omega_k + (1 - 2\gamma\omega_k + \gamma\omega_k^2) \lambda_{k+1} &= \lambda_k,
    \end{align*}
    which gives $\alpha = (\lambda_{k} - \lambda_{k + 1} + 2\gamma\lambda_{k+1}\omega_k - \gamma\lambda_{k+1}\omega_k^2)/\omega_k$. The coefficient before $L(f(x^{k+1}) - f^*)$ in \eqref{eq:f_of_average_projection} is, thus, equal to
    \begin{align*}
    	\frac{\omega_k}{\alpha} = \frac{\omega_k^2}{-(\lambda_{k+1} - \lambda_{k}) + 2\gamma\lambda_{k+1}\omega_k - \gamma\lambda_{k+1}\omega_k^2}
    \end{align*}
    If we minimize it with respect to $\omega_k$, we get that the optimal choice is $\omega_k = \frac{\lambda_{k+1} - \lambda_k}{\gamma\lambda_{k+1}}$, which means that
    \begin{align*}
    	\frac{\omega_k}{\alpha} = \frac{\lambda_{k+1} - \lambda_k}{\gamma\lambda_{k+1}\lambda_k}.
    \end{align*}
\end{proof}

\begin{customthm}{\ref{thm:incr_lambda}}
	    Assume $f$ is $L$-smooth, $\mu$-strongly convex and that the constraints $\cX_1, \dotsc, \cX_m$ are closed and convex. Choose any method $\MM$ that takes as input a problem $F$, an initial point $x$, number of iterations $N$, the smoothness of the problem $L_F$ and possibly strong convexity constant $\mu$. Set $\lambda_{k+1} = \beta k + \nu$ for some $\beta > 0$, $\nu \ge \frac{\beta(1 - \gamma)}{\gamma}$ and $\omega_{k+1}=\frac{\beta}{\gamma(\beta(k +1) + \nu)}$. If for any $x$ method $\MM$ returns a point $y$ satisfying 
	    \[
	    F(y) - F^*\le \rho(F(x) - F^*)
	    \] 
	    after at most 
	    $$C_\MM\tfrac{L_F}{\mu}\log \tfrac{1}{\rho}$$
	    iterations, then Algorithm~\ref{alg:increasing_lambda} provides $(\lambda,\varepsilon)$-accurate solution, where $\lambda=\beta k + \nu$ and $\varepsilon = \tfrac{\theta}{k}$, after at most
        $C_\MM \tfrac{L + \beta}{\mu} \left( 1 + \tfrac{L(f^* - f_0^*)}{\theta\gamma\beta}\right)k$ + $C_\MM \frac{L + \nu}{\mu}\max\{0, \log\frac{f(x^1) - f_0^*}{\theta} \} = O(k)$
    iterations in total.
\end{customthm}
\begin{proof}
	In this proof, we will be referring to iterations that happen inside method $\MM$ as inner iterations, as opposed to outer iterations which give us sequences $\{x^k\}_k$ and $\{y^k\}_k$.

	By Lemma~\ref{lem:bu98gd08g0s} we have $f_\lambda^* + \lambda h_\lambda^* \le f^*$ for any $\lambda$, so
	\begin{align*}
		-f_0^* \le f^* - f_0^* - (f^*_{\lambda_{k+1}} + \lambda_{k+1} h^*_{\lambda_{k+1}}).
	\end{align*}
	Plugging this bound into \eqref{eq:lem_incr_lambda}, where condition $\lambda_k\ge (1 - \gamma)\lambda_{k+1}$ is satisfied by our assumption on $\nu$, implies
	\begin{align*}
		&f(x^{k+1}) + \lambda_{k+1} h(x^{k+1}) - (f^*_{\lambda_{k+1}} + \lambda_{k+1} h^*_{\lambda_{k+1}})\\		
    &\qquad \le f(y^k) + \lambda_k h(y^k) - (f^*_{\lambda_{k}} + \lambda_{k} h^*_{\lambda_{k}}) \\
    &\qquad\qquad + L\frac{\lambda_{k+1} - \lambda_k}{\gamma\lambda_{k+1}^2}(f(x^{k+1}) - (f^*_{\lambda_{k+1}} + \lambda_{k+1} h^*_{\lambda_{k+1}})) + L\frac{\lambda_{k+1} - \lambda_k}{\gamma\lambda_{k+1}^2}(f^* - f_0^*).
	\end{align*}
	By non-negativity of $h$ we also have
	\begin{align*}
		f(x^{k+1}) - (f^*_{\lambda_{k+1}} + \lambda_{k+1} h^*_{\lambda_{k+1}}) \le f(x^{k+1}) + \lambda_{k+1}h(x^{k+1}) - (f^*_{\lambda_{k+1}} + \lambda_{k+1} h^*_{\lambda_{k+1}})
	\end{align*}
	Rearranging the terms, we get
	\begin{align*}
		& f(x^{k+1}) + \lambda_{k+1} h(x^{k+1}) - (f^*_{\lambda_{k+1}} + \lambda_{k+1} h^*_{\lambda_{k+1}}) \\
		&\qquad \le \frac{1}{1 + L\frac{\lambda_{k+1} - \lambda_k}{\gamma\lambda_{k+1}^2}}\left(f(y^k) + \lambda_k h(y^k) - (f^*_{\lambda_{k}} + \lambda_{k} h^*_{\lambda_{k}}) + L\frac{\lambda_{k+1} - \lambda_k}{\gamma\lambda_{k+1}^2}(f^* - f_0^*) \right). 
	\end{align*}
	Now we apply method $\MM$ to get $y^{k+1}$ from $x^{k+1}$. The smoothness of function $F_\lambda = f + \lambda h$, to which we are applying the method, is not bigger than $L + \lambda$, so it takes at most $C_\MM \frac{L + \lambda}{\mu}\log \frac{1}{\rho}$ inner iterations to get
	\begin{align*}
		&f(y^{k+1}) + \lambda_{k+1} h(y^{k+1}) - (f^*_{\lambda_{k+1}} + \lambda_{k+1} h^*_{\lambda_{k+1}})\\
		 &\qquad \le \rho (f(x^{k+1}) + \lambda_{k+1} h(x^{k+1}) - (f^*_{\lambda_{k+1}} + \lambda_{k+1} h^*_{\lambda_{k+1}})).
	\end{align*}
	Denote $\varepsilon_k \eqdef f(y^k) + \lambda_k h(y^k) - (f^*_{\lambda_{k}} + \lambda_{k} h^*_{\lambda_{k}}) $. Then, we have proved that for any $\rho$, after the specified number of inner iterations, we get
	\begin{align}\label{eq:epsilon_recurrence}
		\varepsilon_{k+1} \le \rho \left(\varepsilon_k + L\frac{\lambda_{k+1} - \lambda_k}{\gamma\lambda_{k+1}^2}(f^* - f_0^*) \right).
	\end{align}
	If we choose $\varepsilon_k$ to be always not bigger than $\theta \frac{\lambda_{k+1} - \lambda_k}{\lambda_{k+1}} \le \frac{\theta}{k}$, then we first need $C_\MM \tfrac{L + \nu}{\mu}\log \tfrac{f(x^1) - f_\nu^*}{\theta}$ iterations to get this condition for $\varepsilon_1$ if $\theta < f(x_1) - f_\nu^*$ and 0 iterations otherwise. Afterwards, we will only need to improve from $\varepsilon_k$ to $\varepsilon_{k+1}$ and by inequality \eqref{eq:epsilon_recurrence} with
\begin{align*}
	\rho = \frac{\varepsilon_{k+1}}{\varepsilon_k + L\frac{\lambda_{k+1} - \lambda_k}{\gamma\lambda_{k+1}^2}(f^* - f_0^*)} = \frac{\theta(\lambda_{k+2} - \lambda_{k+1})}{\lambda_{k+2}(\theta \frac{\lambda_{k+1} - \lambda_k}{\lambda_{k+1}} + L\frac{\lambda_{k+1} - \lambda_k}{\gamma\lambda_{k+1}^2}(f^* - f_0^*))} 
\end{align*}	
	 it will take no more than
	\begin{align*}
		C_\MM \frac{L + \lambda_{k+1}}{\mu} \log \left(\frac{\lambda_{k+2}(\lambda_{k+1} - \lambda_k)}{\lambda_{k+1}(\lambda_{k+2} - \lambda_{k+1})} \left(1 + \frac{L(f^* - f_0^*)}{\theta \gamma\lambda_{k+1}} \right) \right)
	\end{align*}
	iterations. If $\lambda_{k+1} = \beta k + \nu$, it simplifies to
	\begin{align*}
		&C_\MM \frac{L + \beta k + \nu}{\mu} \left( \log \left(1 + \frac{1}{k + \frac{\nu}{\beta}} \right) + \log\left(1  + \frac{L(f^* - f_0^*)}{\theta\gamma (\beta k + \nu)} \right) \right)\\
		&\qquad\le C_\MM \frac{L + \beta k + \nu}{\mu} \left( \frac{1}{k + \frac{\nu}{\beta}} + \frac{L(f^* - f_0^*)}{\theta\gamma  \beta(k + \frac{\nu}{\beta})} \right)\\
		&\qquad\le C_\MM \frac{(L + \beta)(k + \frac{\nu}{\beta})}{\mu} \left( \frac{1}{k + \frac{\nu}{\beta}} + \frac{L(f^* - f_0^*)}{\theta\gamma\beta (k + \frac{\nu}{\beta})} \right)\\
		&\qquad= C_\MM \frac{L + \beta }{\mu} \left( 1 + \frac{L(f^* - f_0^*)}{\theta\gamma\beta} \right).
	\end{align*}
	Thus, the cost of one outer iteration is constant and the claim follows.
\end{proof}
\begin{corollary}
	Choose $\theta=\beta = L$, $\nu = L \frac{1 - \gamma}{\gamma}$, arbitrary $k>1$ and run Algorithm~\ref{alg:increasing_lambda} for 
	\[2C_\MM  \tfrac{L}{\mu} \left(\left( 1 + \tfrac{f^* - f_0^*}{\gamma L}\right)k + \frac{1}{\gamma}\max\left\{0, \log\frac{f(x^1) - f^*}{L}\right\}\right)
	\]
	iterations, including the ones in method $\MM$, to obtain $y^k$. Then, it satisfies
	\begin{eqnarray*}
		f(\PP(y^k)) - f^* &\stackrel{\eqref{eq:f(P(xleps))}}{\le} & \frac{2}{k}\left( L + \frac{1}{\gamma}(f^* - f_0^*) \right),\\
		\|y^k - x^*\|^2 &\stackrel{\eqref{eq:xleps_minus_xstar}}{\le} & \frac{1}{k}\left( 4\frac{L}{\mu} + 8 \frac{\|\nabla f(x^*)\|^2}{\gamma \mu L} \right),\\
		\|y^k - \PP(y^k)\|^2 &\stackrel{\eqref{eq:dist_to_proj_xleps}}{\le} &\frac{1}{k^2}\left( \frac{4\|\nabla f(x^*)\|^2}{\gamma^2 L^2} + \frac{2}{\gamma} \right),
	\end{eqnarray*}
	which are $O\left( \frac{1}{k} \right)$ and $O\left(\frac{1}{k^2} \right)$ convergence rates.
\end{corollary}

\subsection{Accelerated method}\label{sec:acc}
Let us now consider methods $\MM$ which give accelerated convergence, such as Catalyst \cite{lin2015universal}. In particular, we will assume that when solving a problem $F$, the complexity is proportional to $\sqrt{\frac{L_F}{\mu}}$ instead of $\frac{L_F}{\mu}$, where $L_F$ and $\mu$ are smoothness and strong convexity of $F$. For simplicity, we will assume that the output of $\MM$ is deterministic.

\begin{theorem}
	Take the same assumption as in Theorem~\ref{thm:incr_lambda} regarding $f$ and constraints and choose an accelerated method $\MM$. Set $\lambda_{k+1} = \beta k^2 + \nu$ for some $\beta > 0$, $\nu \ge \frac{\beta}{\gamma}(\frac{1}{\gamma} - 1)$ and $\omega_{k+1}=\tfrac{\beta}{\gamma(\beta(k + 1)^2 + \nu)}$. If for any $x$ method $\MM$ returns a point $y$ satisfying $F(y) - F^*\le \rho(F(x) - F^*)$ after at most $C_\MM\sqrt{\tfrac{L_F}{\mu}}\log \frac{1}{\rho}$ iterations, then Algorithm~\ref{alg:increasing_lambda} provides $(\beta k^2 + \nu, \tfrac{\theta}{k^2})$-accurate solution after at most
        $C_\MM \left(\sqrt{\tfrac{L + \beta}{\mu}} \left( 3 + \tfrac{L(f^* - f_0^*)}{\theta\gamma} \right)k +\sqrt{\frac{L + \nu}{\mu}}\max\{0, \log\frac{f(x^1) - f_\nu^*}{\theta} \} \right) $
    iterations in total.
\end{theorem}
\begin{proof}
	This time, let us aim at verifying that $\varepsilon_k$ is no more than $ \frac{\beta\theta}{\lambda_{k+1}}\le \frac{\theta}{k^2}$. The claim about the improvement from $\varepsilon_k$ to $\varepsilon_{k+1}$ is going to be exactly the same as in the previous theorem with the only change in the dependency on the conditioning. To wit, if we start from $x^k$, then after at most $C_\MM\sqrt{\tfrac{L_F}{\mu}}\log \frac{1}{\rho}$ inner iterations we have
	\begin{align*}
		\varepsilon_{k+1} \le \rho \left(\varepsilon_k + L\frac{\lambda_{k+1} - \lambda_k}{\gamma\lambda_{k+1}^2}(f^* - f_0^*) \right).
	\end{align*}
	Therefore, to get $\varepsilon_k \le  \frac{\theta}{\lambda_{k+1}}$, we need at most
	\begin{align*}
		 C_\MM \sqrt{\frac{L + \lambda_{k+1}}{\mu}} \log \left(\frac{\lambda_{k+2}}{\lambda_{k+1}} \left(1 + \frac{L(\lambda_{k+1} - \lambda_k)(f^* - f_0^*)}{\beta\theta\gamma\lambda_{k+1}} \right) \right)
	\end{align*}		
	iterations. After specifying $\lambda_{k+1}=\beta k^2 + \nu$, it reduces to
	\begin{align*}
		&C_\MM \sqrt{\frac{L + \beta k^2 + \nu}{\mu}} \left( \log \left(1 + \frac{2k + 1}{k^2 + \frac{\nu}{\beta}} \right) + \log\left(1  + \frac{L\beta(2k - 1)(f^* - f_0^*)}{\beta\theta\gamma (\beta k^2 + \nu)} \right) \right) \\
		&\qquad \le C_\MM \sqrt{\frac{(L + \beta)(k^2 + \frac{\nu}{\beta})}{\mu}} \left(  \frac{2k +1}{k^2 + \frac{\nu}{\beta}} + \frac{L(2k - 1)(f^* - f_0^*)}{\beta\theta\gamma (k^2 + \frac{\nu}{\beta})} \right) \\
		&\qquad \le C_\MM \sqrt{\frac{L + \beta}{\mu}} \left( 3 + 2\frac{L(f^* - f_0^*)}{\beta\theta\gamma k} \right).
	\end{align*}
\end{proof}
Notice that for $\nu = \frac{L(1 - \gamma)}{\gamma^2}$ we have $L + \nu = L\left(1 - \frac{1}{\gamma} + \frac{1}{\gamma^2} \right) \le \frac{L}{\gamma^2}$ since $\gamma \le 1$, so we get the following corollary.
\begin{corollary}
	Choose $\theta=L$, $\beta = 3L$, $\nu = \frac{L(1 - \gamma)}{\gamma^2}$, arbitrary $k>1$ and run Algorithm~\ref{alg:increasing_lambda} for
	\[2C_\MM  \sqrt{\tfrac{L}{\mu}} \left( \left(3 + \tfrac{f^* - f_0^*}{\gamma L} \right) k+ \frac{1}{\gamma}\max\left\{0, \log\frac{f(x^1) - f_\nu^*}{L}\right\}\right)
	\]
	iterations, including the ones in method $\MM$, to obtain $y^k$. Then, it satisfies
	\begin{eqnarray*}
		f(\PP(y^k)) - f^* &\stackrel{\eqref{eq:f(P(xleps))}}{\le} & \frac{1}{k^2}\left( 2L + \frac{1}{2\gamma}(f^* - f_0^*) \right),\\
		\|y^k - x^*\|^2 &\stackrel{\eqref{eq:xleps_minus_xstar}}{\le} & \frac{4}{k^2}\left( \frac{L}{\mu} + \frac{\|\nabla f(x^*)\|^2}{\gamma \mu L} \right),\\
		\|y^k - \PP(y^k)\|^2 &\stackrel{\eqref{eq:dist_to_proj_xleps}}{\le} &\frac{1}{k^4}\left( \frac{\|\nabla f(x^*)\|^2}{2\gamma^2 L^2} + \frac{2}{\gamma} \right),
	\end{eqnarray*}
	which are $O\left( \frac{1}{k^2} \right)$ and $O\left( \frac{1}{k^4} \right)$ convergence rates.
\end{corollary}

\section{Inexact Solution Theory}  \label{sec:approx}

We say that $\xleps$ is an $\varepsilon$-approximate solution of problem \eqref{eq:pb_relaxed} if
	\begin{align}
    \label{eq:def_xleps}
    	f(\xleps) + \lambda h(\xleps) \le f_\lambda^* + \lambda h_\lambda^* + \varepsilon.
    \end{align}
    
    Below we discuss how good an approximate solution is under different assumptions. It is clear after a short look at the results that they are very similar to that of exact solutions.
\subsection{No assumption}
\begin{theorem}
    Let $\xleps$ be a $(\lambda, \varepsilon)$-approximate solution of problem \ref{eq:pb}. Then,
    \begin{align*}
        \|\xleps - \PP(\xleps)\|^2 \le 2\frac{f^* - f_0^* + \varepsilon}{\gamma\lambda}.
    \end{align*}
\end{theorem}
\begin{proof}
    By definition of $\xleps$ we have
    \begin{align*}
        h(\xleps) \le \frac{f_\lambda^* + \lambda h_\lambda^* f(\xleps) + \varepsilon }{\lambda} \le \frac{f^* - f(\xleps) + \varepsilon}{\lambda}.
    \end{align*}
    Since $x_0^* = \argmin f(x)$, we also have $f(\xleps) \ge f_0^*$. In addition, we derive from linear regularity
    \begin{align*}
        \|\xleps - \PP(\xleps)\|^2 \le \frac{2 h(\xleps)}{\gamma} \le \frac{f^* - f_0^* + \varepsilon}{\lambda}.
    \end{align*}
\end{proof}

\subsection{Smooth objective}
\begin{theorem}
    Let $f$ be $L$-smooth and $\xleps$ be a $(\lambda, \varepsilon)$-approximate solution of problem \ref{eq:pb}. Then,
    \begin{align*}
        \gamma\frac{G - 2(L + \lambda)\varepsilon}{L^2 + 2\gamma\lambda^2} &\le \|\xleps - \PP(\xleps)\|^2.
    \end{align*}
\end{theorem}

\begin{proof}
    From $(L + \lambda)$-smoothness of $f + \lambda h$ we obtain
    \begin{eqnarray*}
        \|\nabla f(\xleps) + \lambda \nabla h(\xleps)\|^2 
        &\stackrel{\eqref{eq:first_order_for_relaxed_pb}}{=}& \|\nabla (f(\xleps) + \lambda h(\xleps)) - \nabla (f(x_\lambda^*) + \lambda h(x_\lambda^*))\|^2\\
        &\stackrel{\eqref{eq:smooth_bregman}}{\le}& 2(L + \lambda)(f(\xleps) + \lambda h(\xleps) - f_\lambda^* - \lambda h_\lambda^*) \\
        &\le& 2(L + \lambda) \varepsilon.
    \end{eqnarray*}
    Consequently,
    \begin{eqnarray*}
        \lambda^2 \left\|\xleps - \meanj \PPj(\xleps) \right\|^2 
        &=& \|-\nabla f(\xleps) +  \nabla (f(\xleps) + \lambda h(\xleps)) - \nabla (f(x_\lambda^*) + \lambda h(x_\lambda^*))\|^2 \\
        &\ge& \frac{1}{2}\|\nabla f(\xleps)\|^2 - \|\nabla (f(\xleps) + \lambda h(\xleps)) - \nabla (f(x_\lambda^*) + \lambda h(x_\lambda^*))\|^2 \\
        &\ge&  \frac{1}{2}\|\nabla f(\xleps)\|^2 - 2(L + \lambda) \varepsilon.
    \end{eqnarray*}
    Now, let us lower bound the first term in the last expression:
    \begin{eqnarray*}
        \|\nabla f(\xleps)\|^2
        &=& \|\nabla f(\xleps) - \nabla f(\PP(\xleps)) + \nabla f(\PP(\xleps))\|^2\\
        &\ge& \frac{1}{2} \|\nabla f(\PP(\xleps))\|^2 - \|\nabla f(\xleps) - \nabla f(\PP(\xleps))\|^2 \\
        &\ge& 2G - \|\nabla f(\xleps) - \nabla f(\PP(\xleps))\|^2\\
        &\ge& 2G - L^2\|\xleps - \PP(\xleps)\|^2\\
        &\ge& 2G - \frac{L^2}{\gamma} \meanj \|\xleps - \PPj(\xleps)\|^2.
    \end{eqnarray*}
    Plugging this bound into what we had before, we get
    \begin{align*}
        G - 2(L + \lambda)\varepsilon 
        &\le \lambda^2 \|\xleps - \meanj \PPj(\xleps)\|^2 
        + \frac{L^2}{\gamma} \meanj \|\xleps - \PPj(\xleps)\|^2\\
        &\le 2\left( \lambda^2 + \frac{L^2}{\gamma} \right) h(\xleps).
    \end{align*}
\end{proof}

\begin{theorem}
    Let $f$ be $L$-smooth and $\xleps$ be a $(\lambda, \varepsilon)$-approximate solution of problem \ref{eq:pb}. Then,
    \begin{align*}
        f(\xleps) &\le f^* + \varepsilon - \frac{\|\nabla f(x^*)\|^2}{2(L + \lambda)}.
    \end{align*}
\end{theorem}

\begin{proof}
    It follows from definition of $\xleps$ and Theorem \ref{thm:optimal_obj_values} that
\[
        f(\xleps) \le f(\xleps) + \lambda h(\xleps) 
        \le f_\lambda^* + \lambda h_\lambda^* + \varepsilon \le f^* + \varepsilon - \frac{\|\nabla f(x^*)\|^2}{2(L + \lambda)}.
\]
\end{proof}

\subsection{Smooth and convex}
\begin{theorem}
    Let $f$ be convex, $L$-smooth function and $\xleps$ be a $(\lambda, \varepsilon)$-approximate solution of problem \ref{eq:pb}. Then,
    \begin{align*}
        f^* - f(\xleps) \le \frac{2}{\gamma\lambda}\|\nabla f(x^*)\|^2
        + \max\left(0, 2\varepsilon - \frac{\|\nabla f(x^*)\|^2}{L + \lambda}\right) .
    \end{align*}
    
\end{theorem}

\begin{proof}
	By Lemma~\ref{lem:f_star_minus_f_upper}, convexity of $f$ implies
\begin{eqnarray*}
f(\xleps) 
&\overset{\eqref{eq:f_star_minus_f_lambda_upper} }{=}&f^* - \|\nabla f(x^*)\| \sqrt{\frac{2}{\gamma}h(\xleps)}\\
&\overset{\eqref{eq:def_xleps}}{\geq} &  f^* - \|\nabla f(x^*)\| \sqrt{\frac{2}{\gamma \lambda} (f_\lambda^* + \lambda h_\lambda^* + \varepsilon - f(\xleps)) }\\
&\overset{\eqref{eq:f_star_minus_fh_lambda_star}}{\geq} &  f^* - \|\nabla f(x^*)\| \sqrt{\frac{2}{\gamma \lambda} (f^* + \varepsilon - \frac{\|\nabla f(x^*)\|^2}{2(L + \lambda)}  - f(\xleps)) }.
\end{eqnarray*}
	Denote $A = f^* - f(\xleps)$, $B = \|\nabla f(x^*)\|$ and $C = \varepsilon - \frac{\|\nabla f(x^*)\|^2}{2(L + \lambda)}$. Then, the inequality above can be rewritten as $A \le B\sqrt{A} + C$. If $C< 0 $, it yields $A\le B^2$. Otherwise, we derive $(\sqrt{A} - \frac{B}{2})^2 \le C + \frac{B^2}{4}$, from which it follows $A \le ( \frac{B}{2} + \sqrt{C + \frac{B^2}{4}})^2 \le B^2 + 2 C$. Combining the two cases, we conclude $A \le B^2 + \max\{0, 2C\}$.
\end{proof}

\begin{theorem}
\label{thm:better_than_f_star}
	Let $\xleps$ be a $(\lambda, \varepsilon)$-approximate solution of problem \ref{eq:pb} with $\varepsilon\le \frac{\|\nabla f(x^*)\|^2}{2(L + \lambda)}$. Then,
    \begin{align*}
    	f(\xleps) &\le f^*.
    \end{align*}
    If, further, $f$ is convex, we also have
    \begin{align*}
    	h(\xleps) &\le \frac{2}{\gamma\lambda^2}\|\nabla f(x^*)\|^2 + \frac{\|\nabla f(x^*)\|^2}{2\lambda(L + \lambda)}.
    \end{align*}
\end{theorem}
\begin{proof}
	By definition of $\xleps$
    \begin{align}
    	f(\xleps) 
        \le f_\lambda^* + \lambda h_\lambda^* + \varepsilon
        \stackrel{\mathclap{\eqref{eq:f_star_minus_fh_lambda_star}}}{\le} f^* - \frac{\|\nabla f(x^*)\|^2}{2(L + \lambda)} + \varepsilon
        \le f^*.\label{eq:f_h_lambda_epsilon_le_f_star}
    \end{align}
    Moving on, convexity of $f$ brings us with the help of Lemma \ref{lem:f_star_minus_f_upper}
    \begin{align*}
    	f(\xleps) 
        &\ge f^* - \|\nabla f(x^*)\|\sqrt{\frac{2}{\gamma}h(\xleps)} \\
        &\stackrel{\eqref{eq:def_xleps}}{\ge} f^* - \|\nabla f(x^*)\|\sqrt{\frac{2}{\gamma\lambda}(f_\lambda^* + \lambda h_\lambda^* + \varepsilon - f(\xleps)) } \\ 
        &\stackrel{\eqref{eq:f_h_lambda_epsilon_le_f_star}}{\ge} f^* - \|\nabla f(x^*)\|\sqrt{\frac{2}{\gamma\lambda}(f^* - f(\xleps))},
    \end{align*}
    which finally gives 
    \begin{align}
    \label{eq:f_star_minus_fleps}
    	f^* - f(\xleps) \le \frac{2}{\gamma\lambda}\|\nabla f(x^*)\|^2.
    \end{align}
    Moreover,
    \begin{align}
    	h(\xleps) 
        \stackrel{\eqref{eq:def_xleps}}{\le}\frac{1}{\lambda}(f_\lambda^* + \lambda h_\lambda^* + \varepsilon - f(\xleps)) 
        \stackrel{\eqref{eq:f_h_lambda_epsilon_le_f_star}+\eqref{eq:f_star_minus_fleps}}{\le} \frac{2}{\gamma\lambda^2}\|\nabla f(x^*)\|^2 + \frac{\varepsilon}{\lambda}, \label{eq:hxleps}
    \end{align}
    which results in our claim after we plug $\varepsilon \le \frac{\|\nabla f(x^*)\|^2}{2(L + \lambda)}$.
\end{proof}

\begin{theorem}
	Let $\xleps$ be an $ \varepsilon$-approximate solution of problem \eqref{eq:pb_relaxed}  with $\lambda> \frac{L}{\gamma}$. If $f$ is $L$-smooth, then
    \begin{align}
    		f(\PP(\xleps))
		\le f^* +  \frac{\gamma \lambda}{\gamma \lambda - L} \varepsilon +  \frac{L}{\gamma \lambda - L} (f^* - f_0^*).\label{eq:f(P(xleps))}
    \end{align}
\end{theorem}

Also note that from~\eqref{eq:hxleps} it also follows 
\begin{align}
	\|\xleps - \PP(\xleps)\|^2 \le \frac{4}{\gamma^2\lambda^2}\|\nabla f(x^*)\|^2 + \frac{2\varepsilon}{\gamma\lambda}. \label{eq:dist_to_proj_xleps}
\end{align}

\begin{proof}
	Let us apply Lemma \ref{lem:f_of_projection} to point $\xleps$. Then, we get
	\begin{align*}
		f(\PP(\xleps))
		& \le f(\xleps) + \lambda h(\xleps) + \frac{L}{\gamma \lambda - L} (f(\xleps) - f_0^*) \\
		&\le f_\lambda^* + \lambda h_\lambda^* + \varepsilon +  \frac{L}{\gamma \lambda - L} (f(\xleps) - f^*) +  \frac{L}{\gamma \lambda - L} (f^* - f_0^*) \\
		&\le f^* + \varepsilon +  \frac{L}{\gamma \lambda - L} \varepsilon +  \frac{L}{\gamma \lambda - L} (f^* - f_0^*).
	\end{align*}
\end{proof}

\subsection{Strongly convex}
\begin{theorem}
    Let $f$ be $\mu$-strongly convex and $\xleps$ be an $ \varepsilon$-approximate solution of problem \eqref{eq:pb_relaxed}. Then,
    \begin{align}
        \|\xleps - x^*\|^2 \le \frac{4\varepsilon}{\mu} + 8\frac{\|\nabla f(x^*)\|^2}{\gamma \mu \lambda}. \label{eq:xleps_minus_xstar}
    \end{align}
\end{theorem}

\begin{proof}
    It follows from strong convexity of $f + \lambda h$ that
    \begin{align*}
        \|\xleps - x_\lambda^*\|^2 \le \frac{2}{\mu} (f(\xleps) + \lambda h(\xleps) - (f_\lambda^* + \lambda h_\lambda^*)) \le \frac{2\varepsilon}{\mu}.
    \end{align*}
    Combining this inequality with Theorem~\ref{thm:u98sgbd} (iv)  yields the claim.
\end{proof}
\clearpage

\section{Experiments with a nonconvex $f$ in 2 dimensions} \label{sec:exp1}

In Figure~\ref{fig:nonconvex_obj_rate}  we show how different quantities may depend on $\lambda$ in the case where objective function is nonconvex.  To be precise, we used two-dimensional function \[f(x, y)=x^2y^2\] to be able to find global minima by grid search. 

The plots show the rate $O(\frac{1}{\lambda^2})$ for $f(\PP(x_\lambda^*)) - f^*$, $\|x_\lambda^* - x^*\|^2$ and $\|x_\lambda^* - \PP(x_\lambda^*)\|^2$, which is not covered by our theory for nonconvex objectives. In contrast, except for the last quantity, we are only able to show $O(\frac{1}{\lambda})$ upper bound, and lower bounds of order $\Omega(\frac{1}{\lambda^2})$ for the last two quantities. This leads us to  conjecture that, under suitable assumptions, one can prove upper bounds having $O(\frac{1}{\lambda^2})$ rate for nonconvex objectives. 

We also note that under the extreme assumption that $\lambda\le \frac{L}{1- \gamma}$, at least for local $\gamma$  around $x^*$, we obtain by Theorem~\ref{thm:optimal_obj_values}(v) the desired rate for $\|x_\lambda^* - x^*\|^2$, although only for strongly convex problems.

\begin{figure}[h]
\begin{center}
	\includegraphics[scale=0.40]{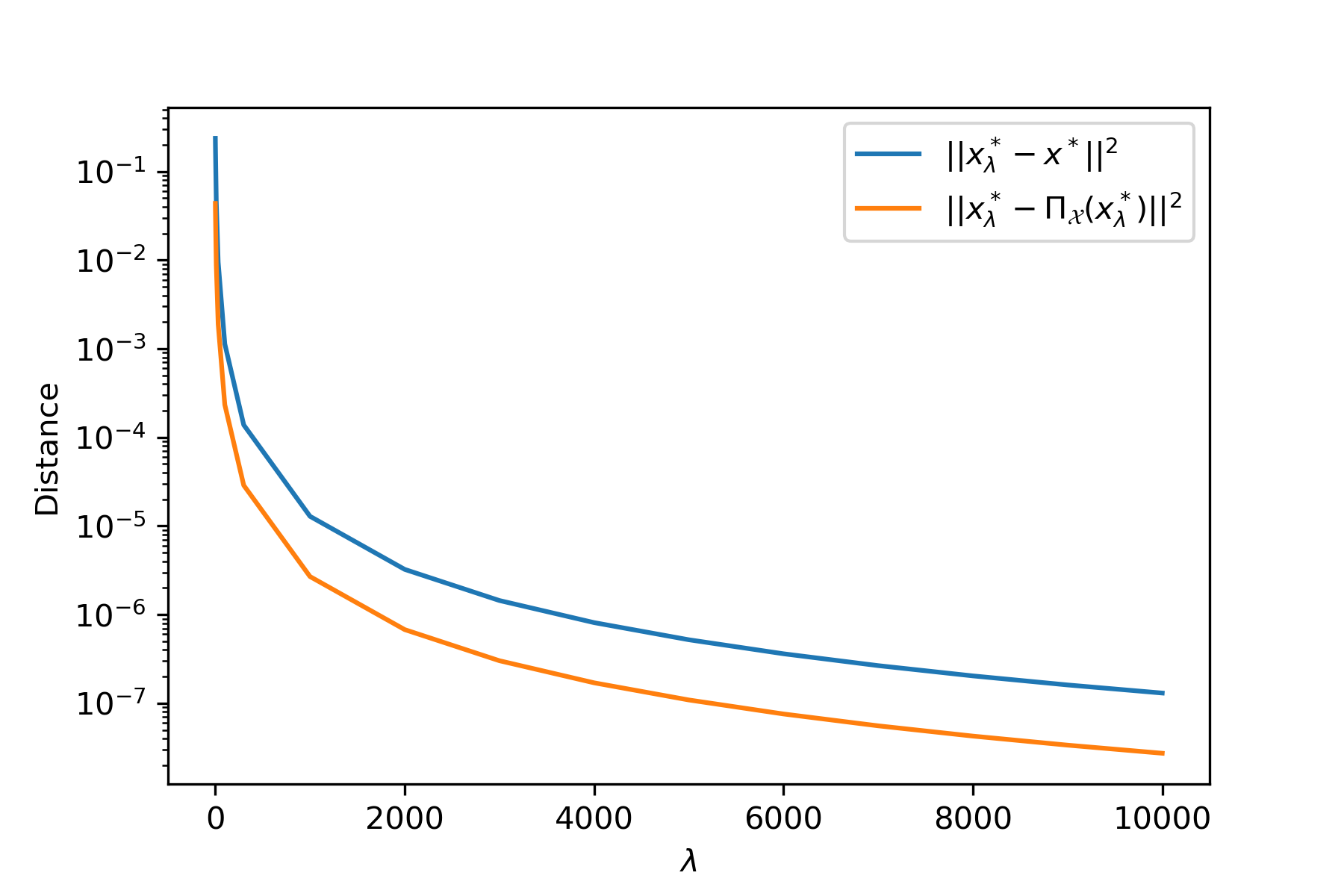}
	\includegraphics[scale=0.40]{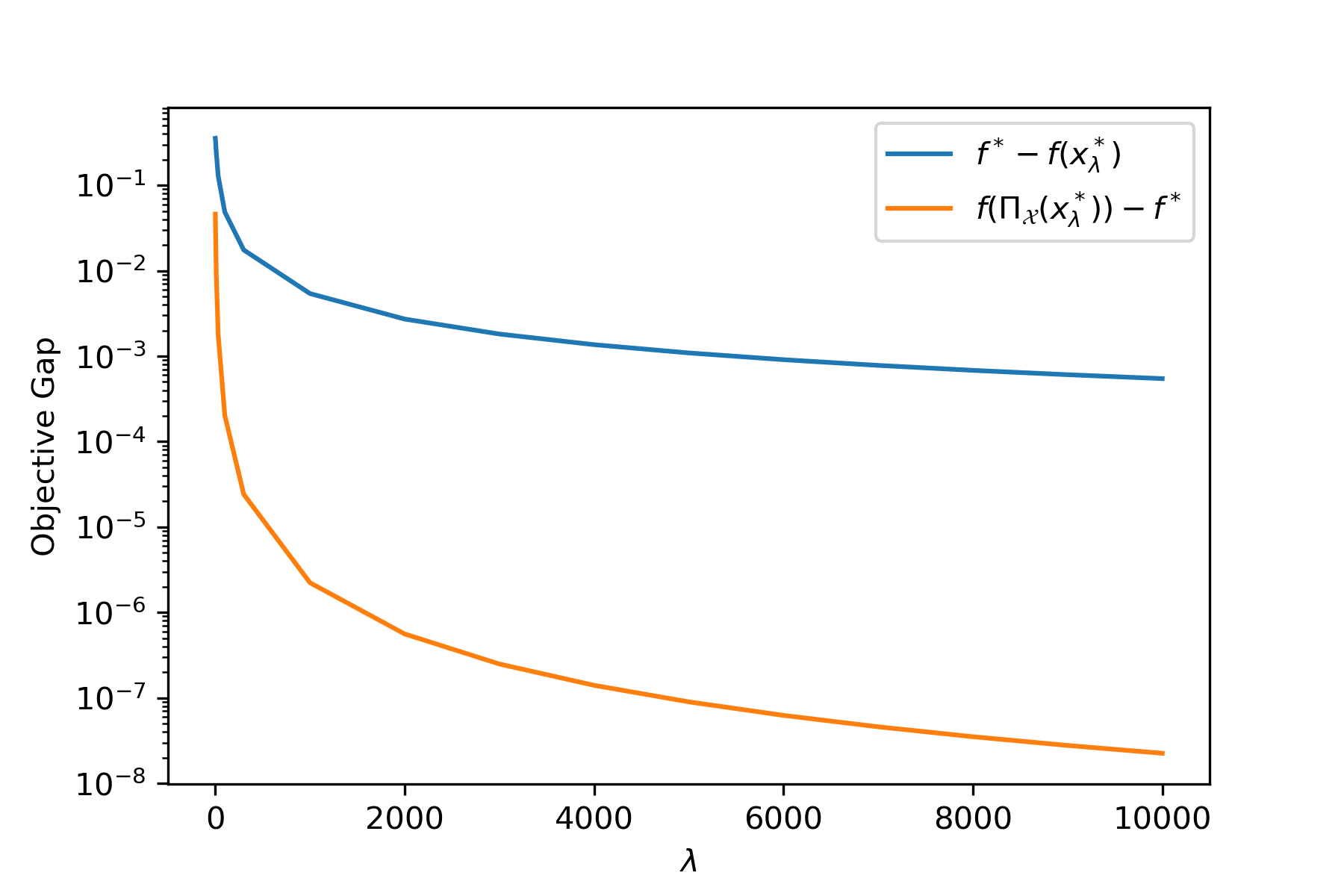}
\end{center}
\caption{(1) Squared distances from $x_\lambda^*$ to its projection $\PP(x_\lambda^*)$ and to the solution of the constrained problem $x^*$ for different $\lambda$; (2) functional values gaps: $f^* - f_\lambda^*$ and $f(\PP(x_\lambda^*)) - f^*$. In this experiment, only $f^* - f_\lambda^*$ shows dependency on $\lambda$ as $\Theta(\frac{1}{\lambda})$, the other quantities converge as $O(\frac{1}{\lambda^2})$.}
\label{fig:nonconvex_obj_rate}
\end{figure}

The values of $f$ and $f + \lambda h$, where $\lambda = 1,000$, are shown in Figure~\ref{fig:nonconvex_obj}.

\begin{figure}[h]
\begin{center}
    \includegraphics[scale=0.40]{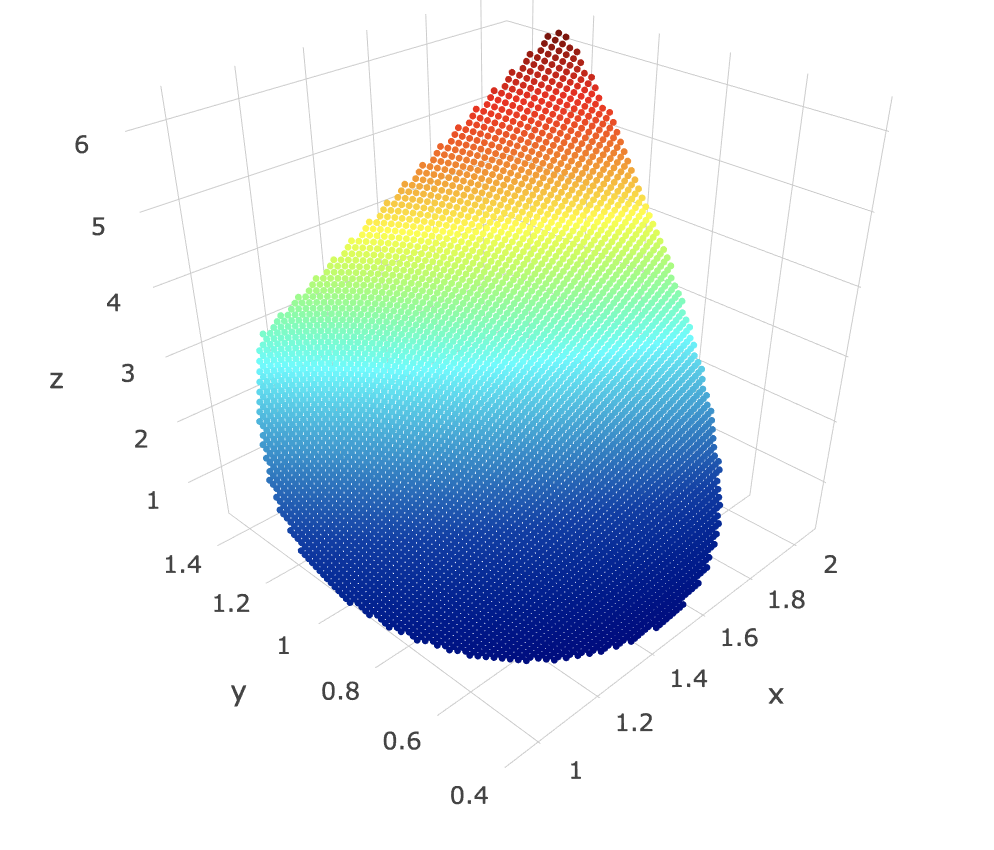}
    \includegraphics[scale=0.40]{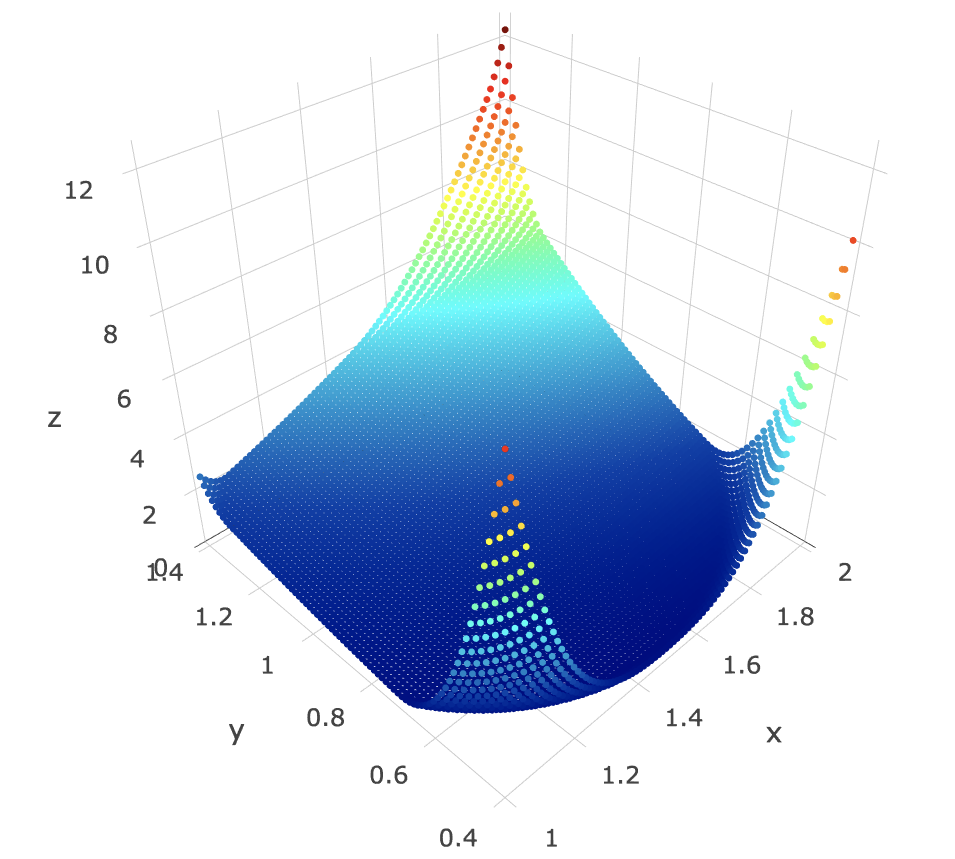}
\end{center}
\caption{(1) Values of function $f(x,y) = x^2y^2$ constrained to the intersection of sets $\cX_1 = \{(x,y) \;:\; x + 4y \le 7\}$ and $\cX_2 = \{(x,y) \;:\; (x - 1.5)^2 + (y - 1)^2\le 1/3 \}$; (2) values of the reformulated objective $f+\lambda h$ with $\lambda = 1,000$.}
\label{fig:nonconvex_obj}
\end{figure}

\clearpage

\section{Additional Experiments} \label{sec:exp2}

This section provides extra plots with the same datasets and algorithms, but this time we measure infeasibility and values of the objective function without projecting iterates. Infeasibility is measured as squared distance from an iterate to its projection: $\|x-\Pi_{\cX}(x)\|^2$. Since some intermediate iterates are not necessarily feasible, their objective value might be below the optimal value $f(x^*)$. This makes the corresponding plots less meaningful, nonetheless we still provide them to facilitate understanding of the obtained results.
\begin{table}[ht]
  \begin{center}
  \begin{tabular}{c@{\quad}cccc}
    &
     $m=40$  & $m=60$& $m=100$ \\
\tiny    Infeasibility
      & \raisebox{-\totalheight / 2}{\includegraphics[scale=0.30]{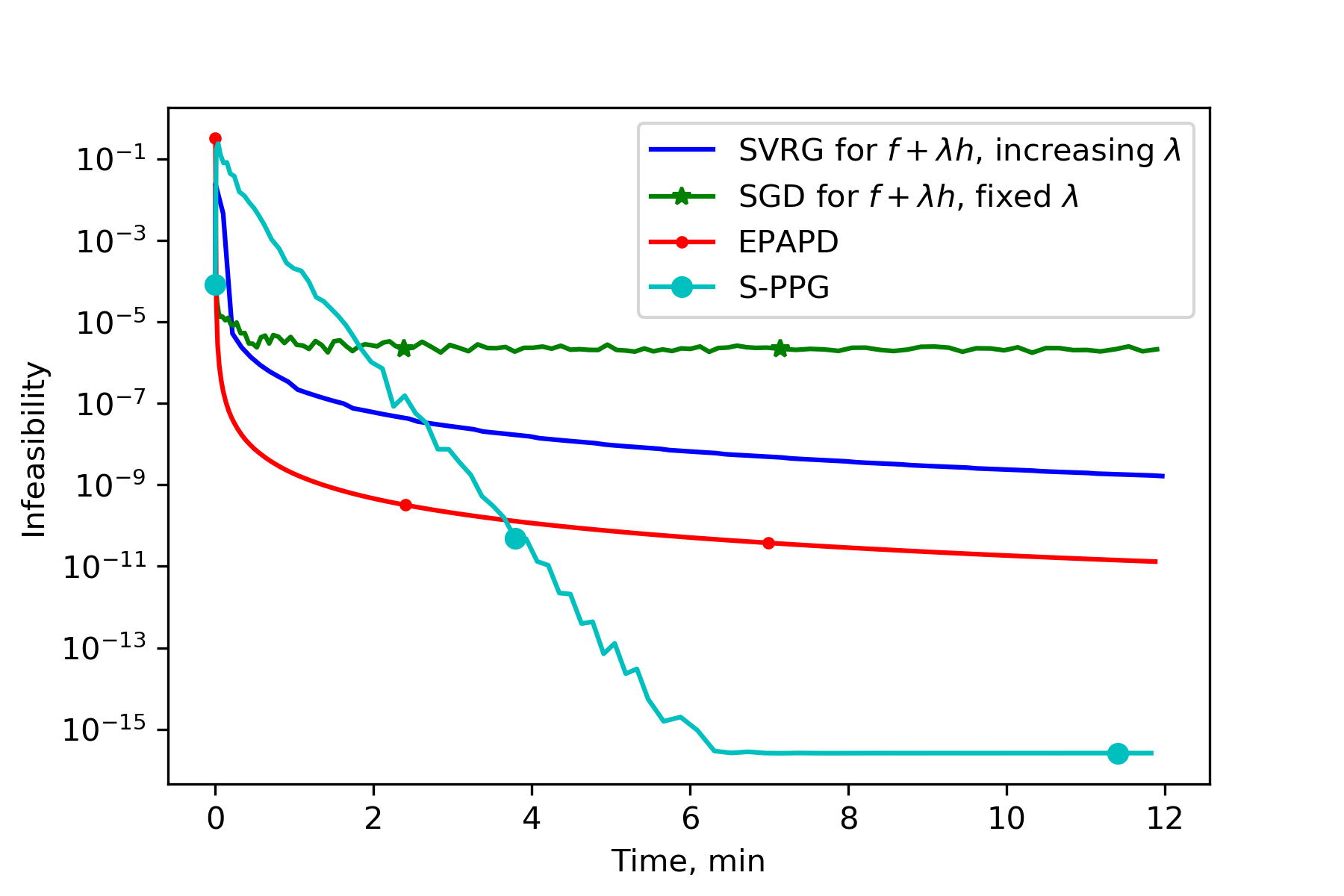}}
      & \raisebox{-\totalheight / 2}{\includegraphics[scale=0.30]{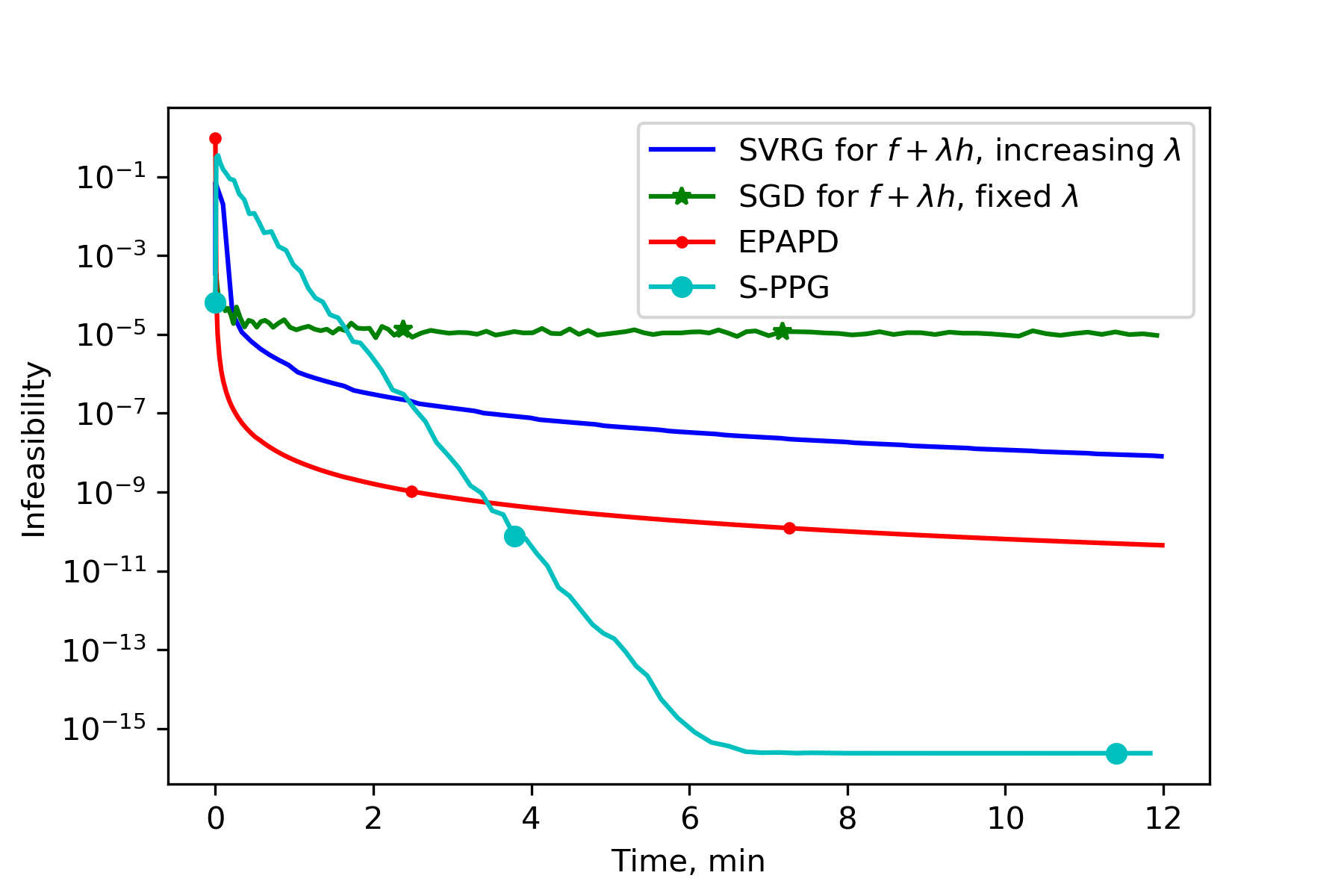}}
      & \raisebox{-\totalheight / 2}{\includegraphics[scale=0.30]{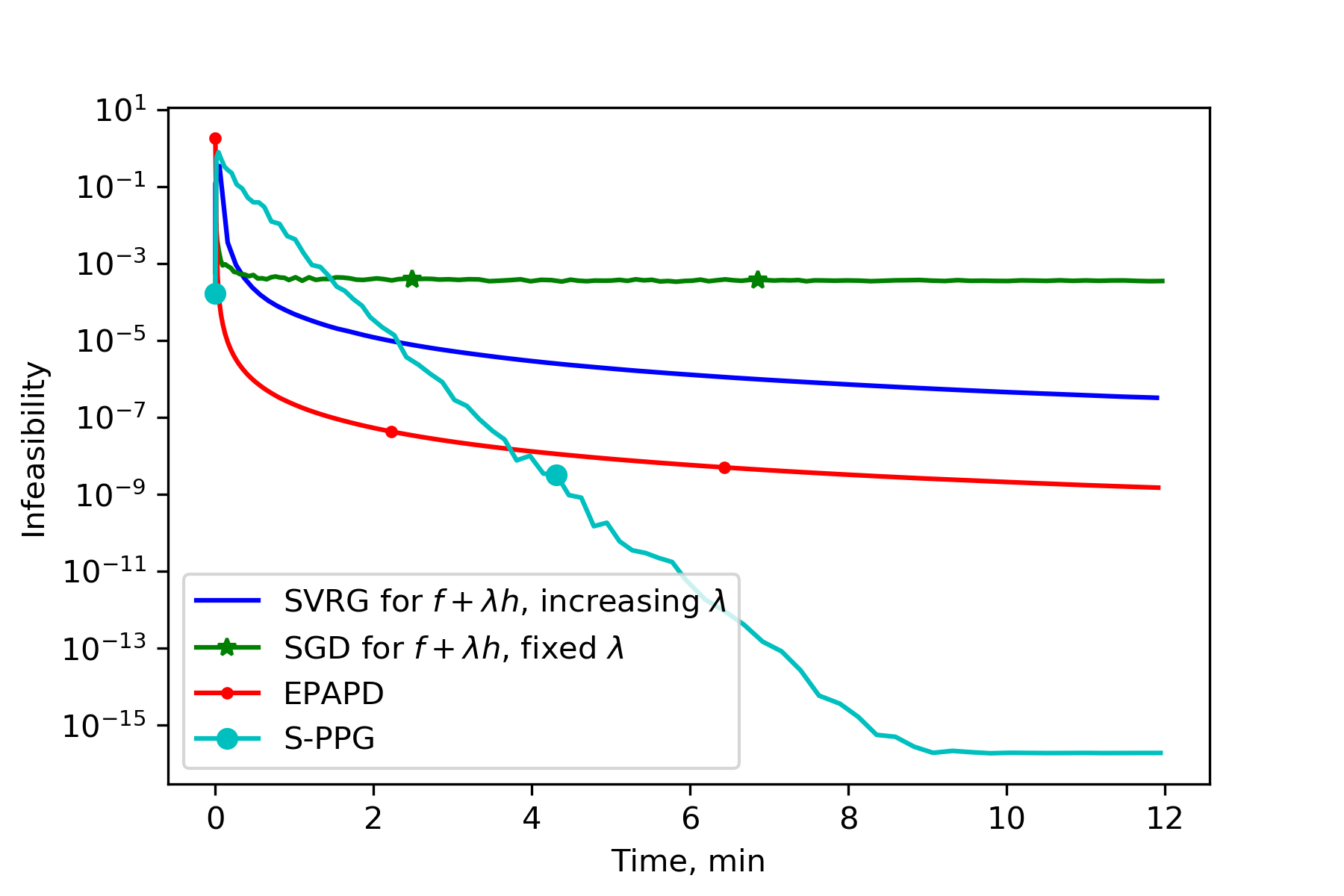}}\\ 
      
    \makecell{\tiny  Objective value \\ \tiny  of iterates \\ \tiny  without\\ \tiny  extra projection} 
      & \raisebox{-\totalheight / 2}{\includegraphics[scale=0.30]{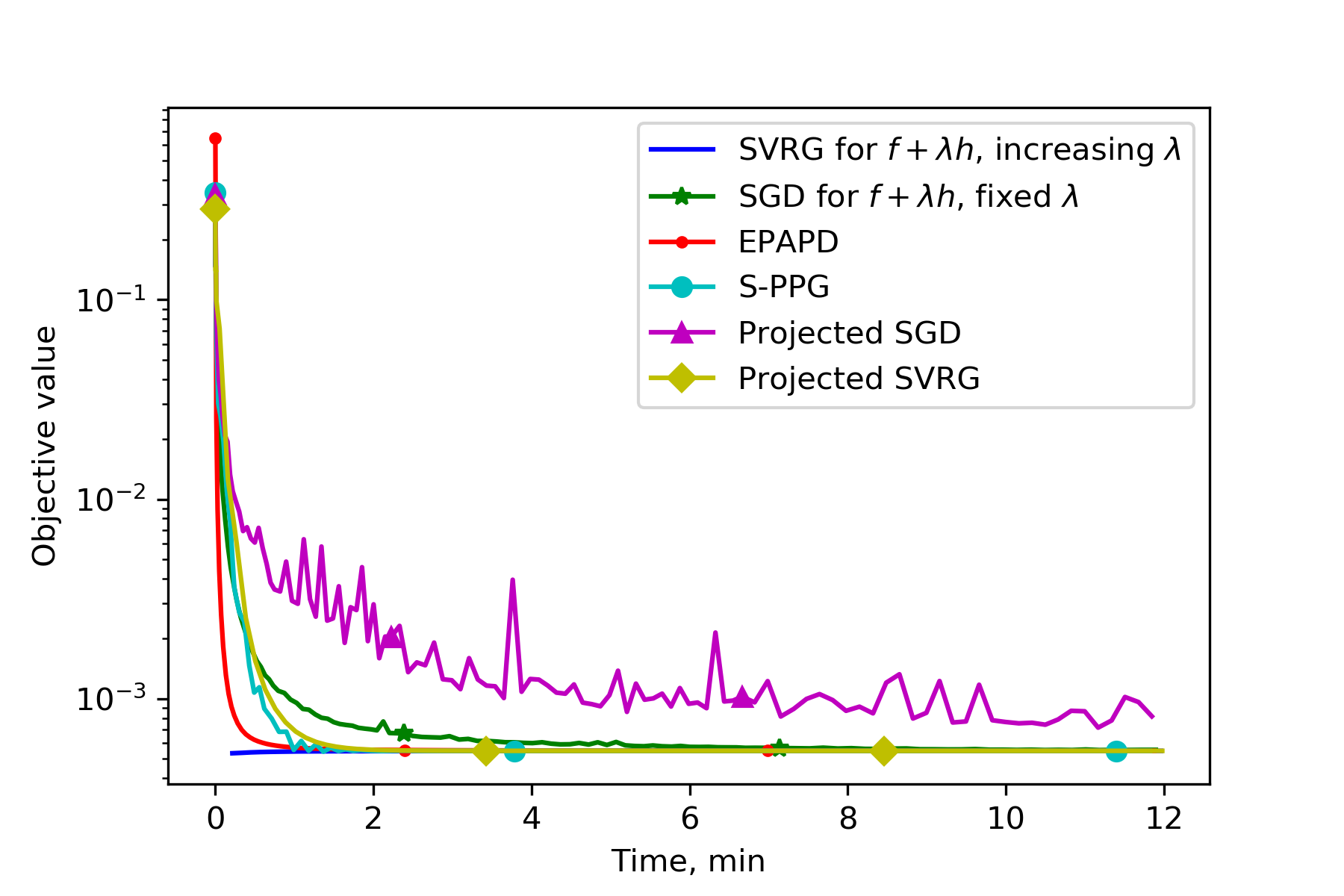}}
      & \raisebox{-\totalheight / 2}{\includegraphics[scale=0.30]{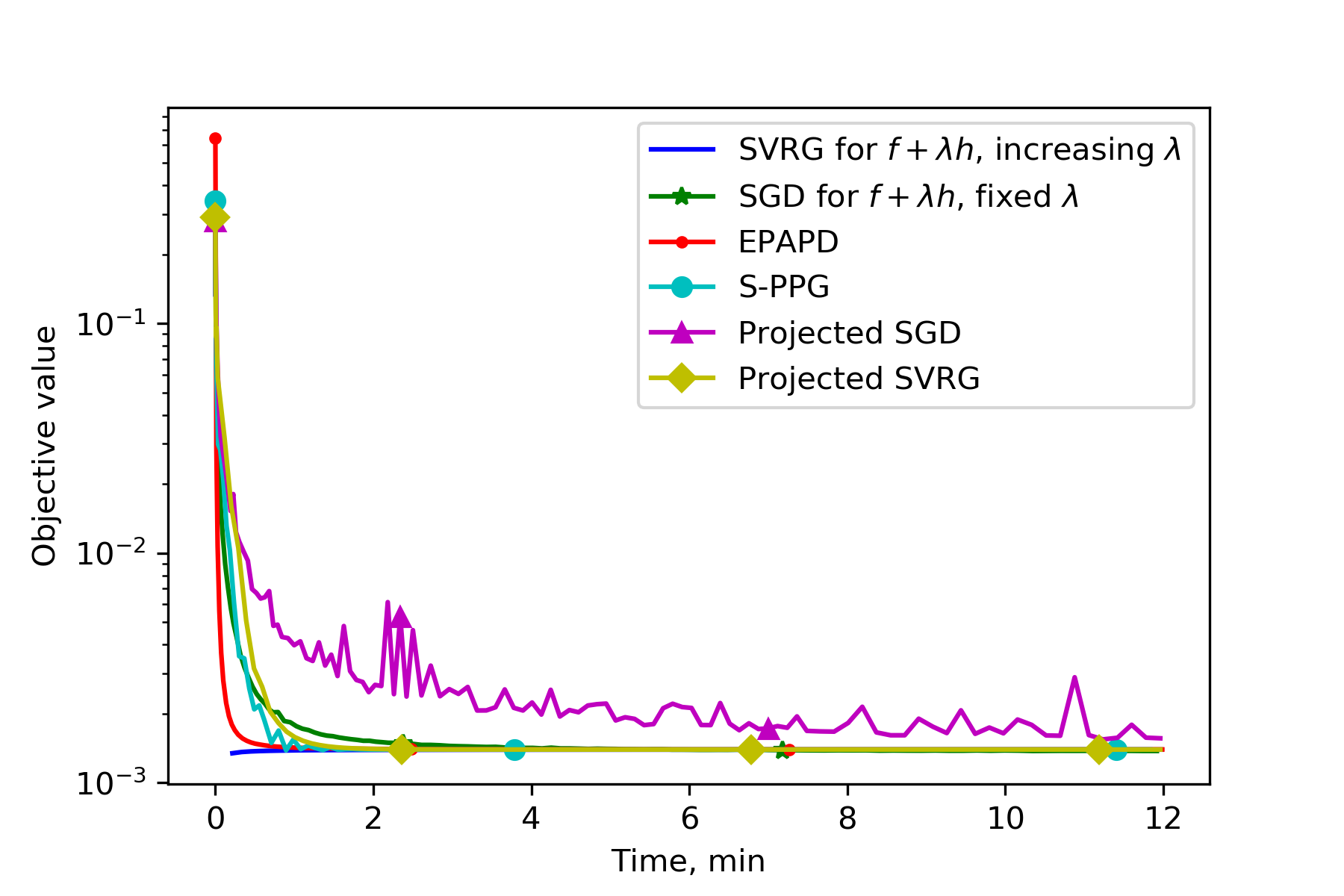}}
      & \raisebox{-\totalheight / 2}{\includegraphics[scale=0.30]{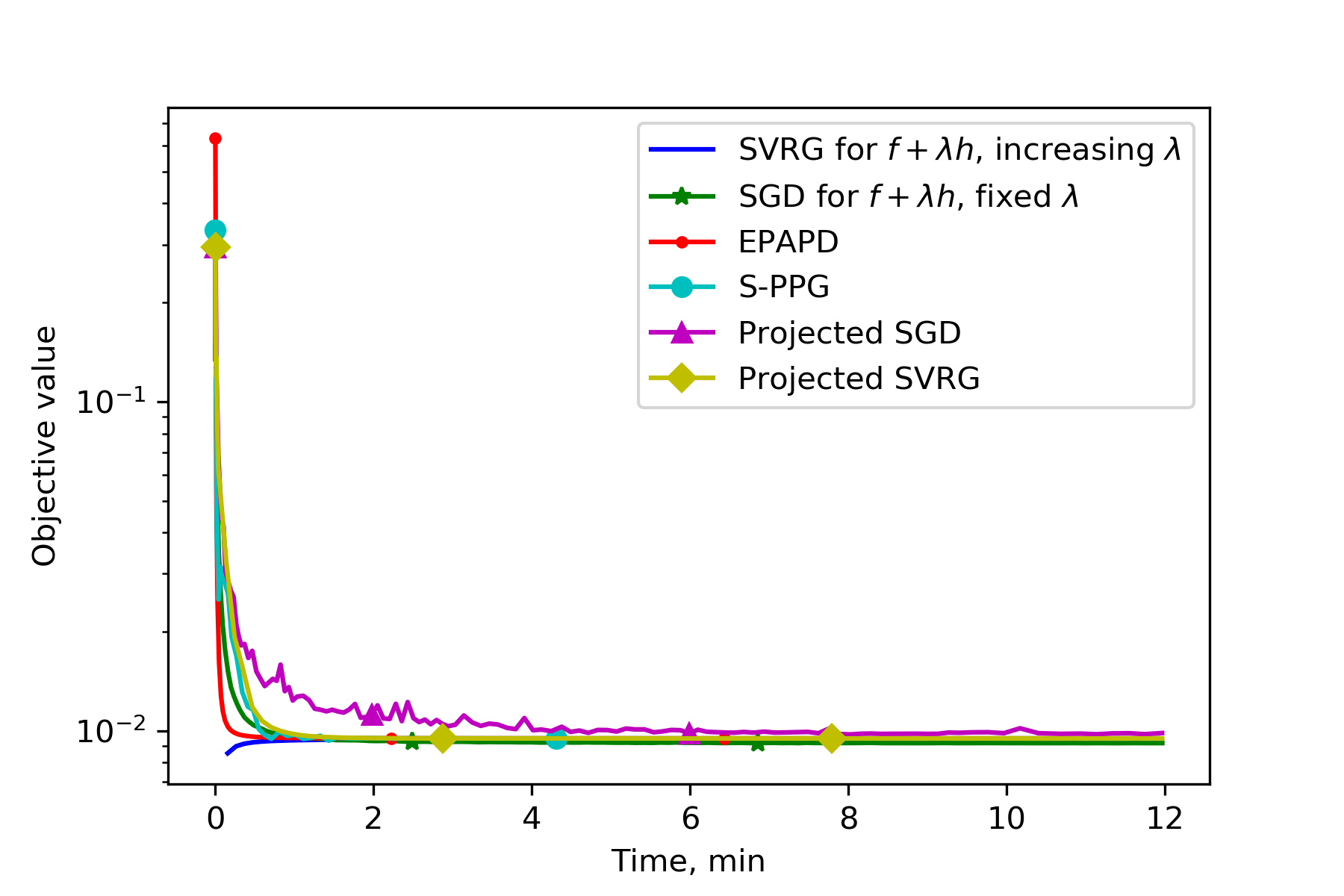}}
  \end{tabular}
  \end{center}
  \caption{A1a dataset.}\label{tab:a1a_second}
\end{table}

\begin{table}[ht]
  \begin{center}
  \begin{tabular}{c@{\quad}cccc}
    & $m=40$  & $m=60$& $m=100$ \\
\tiny      Infeasibility
      & \raisebox{-\totalheight / 2}{\includegraphics[scale=0.30]{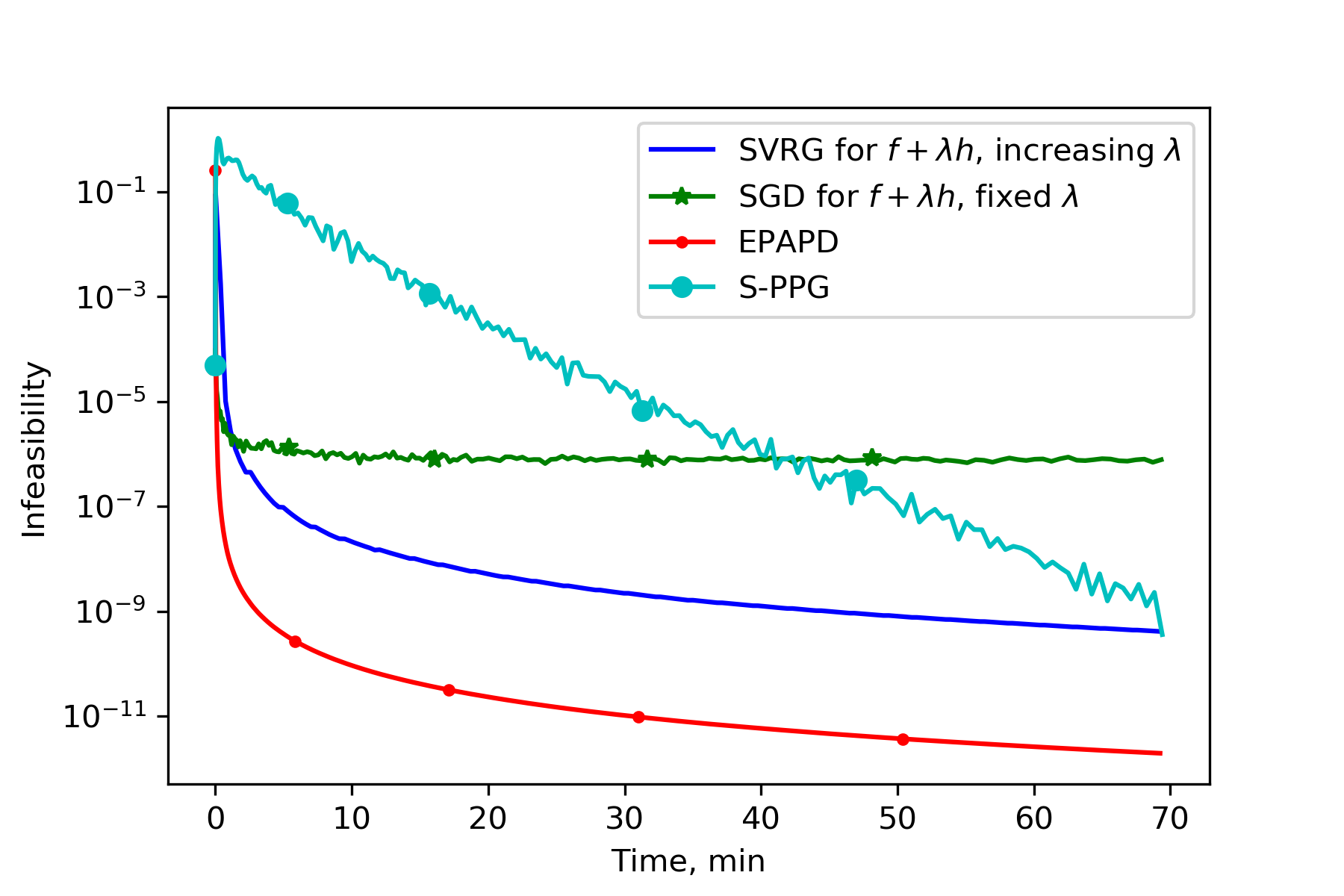}}
      & \raisebox{-\totalheight / 2}{\includegraphics[scale=0.30]{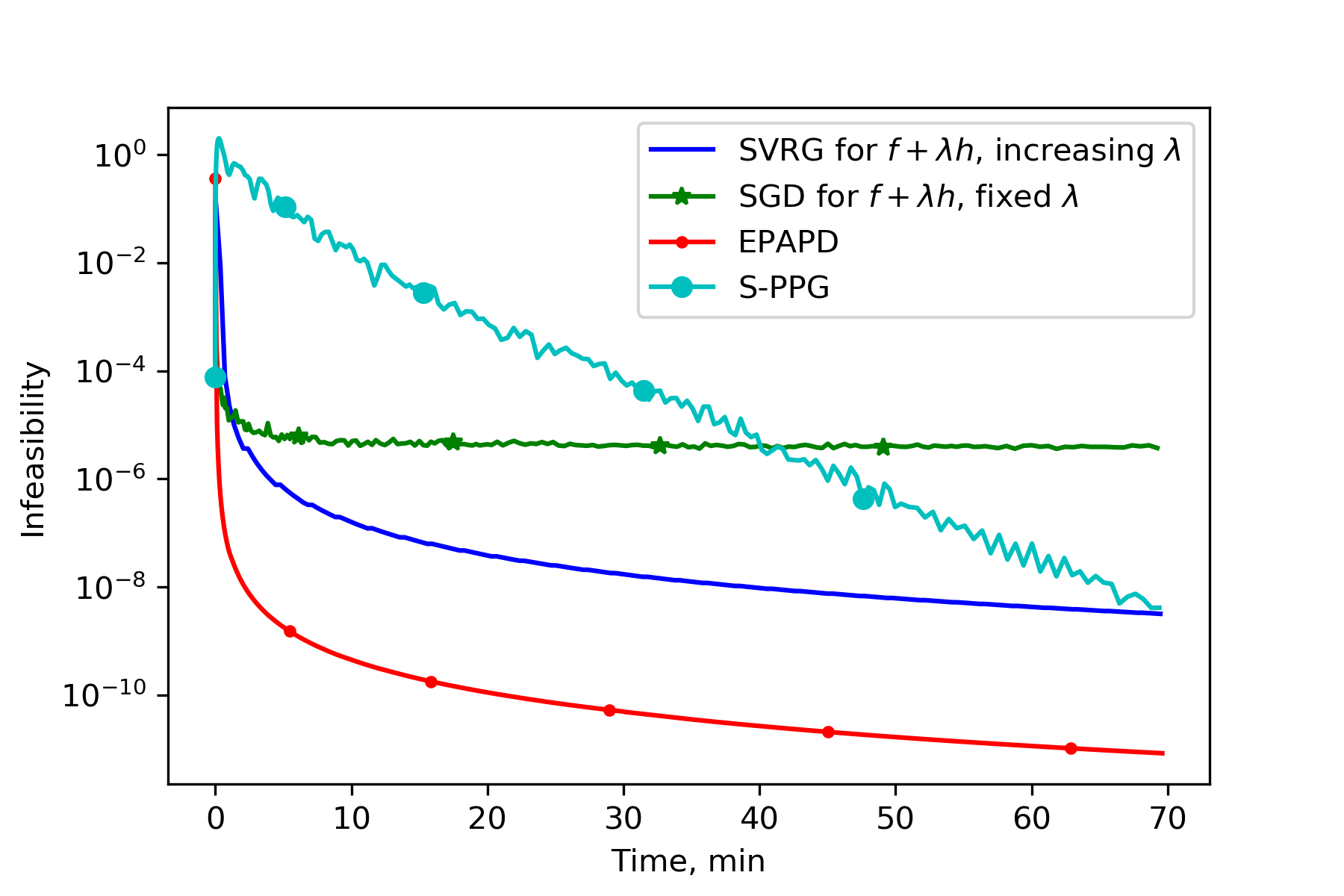}}
      & \raisebox{-\totalheight / 2}{\includegraphics[scale=0.30]{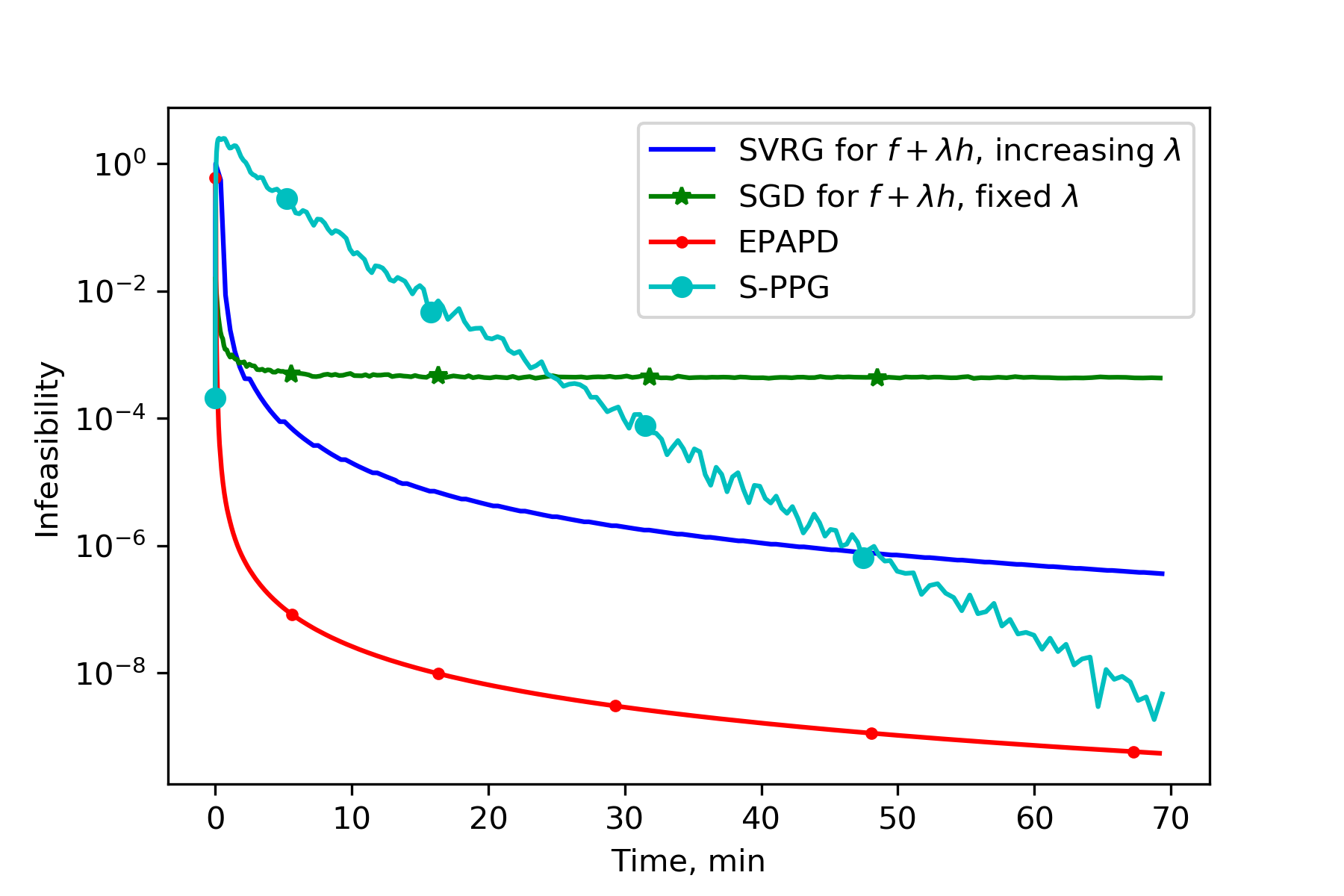}}\\ 
      
    \makecell{\tiny  Objective value \\ \tiny  of iterates \\ \tiny  without\\ \tiny  extra projection} 
      & \raisebox{-\totalheight / 2}{\includegraphics[scale=0.30]{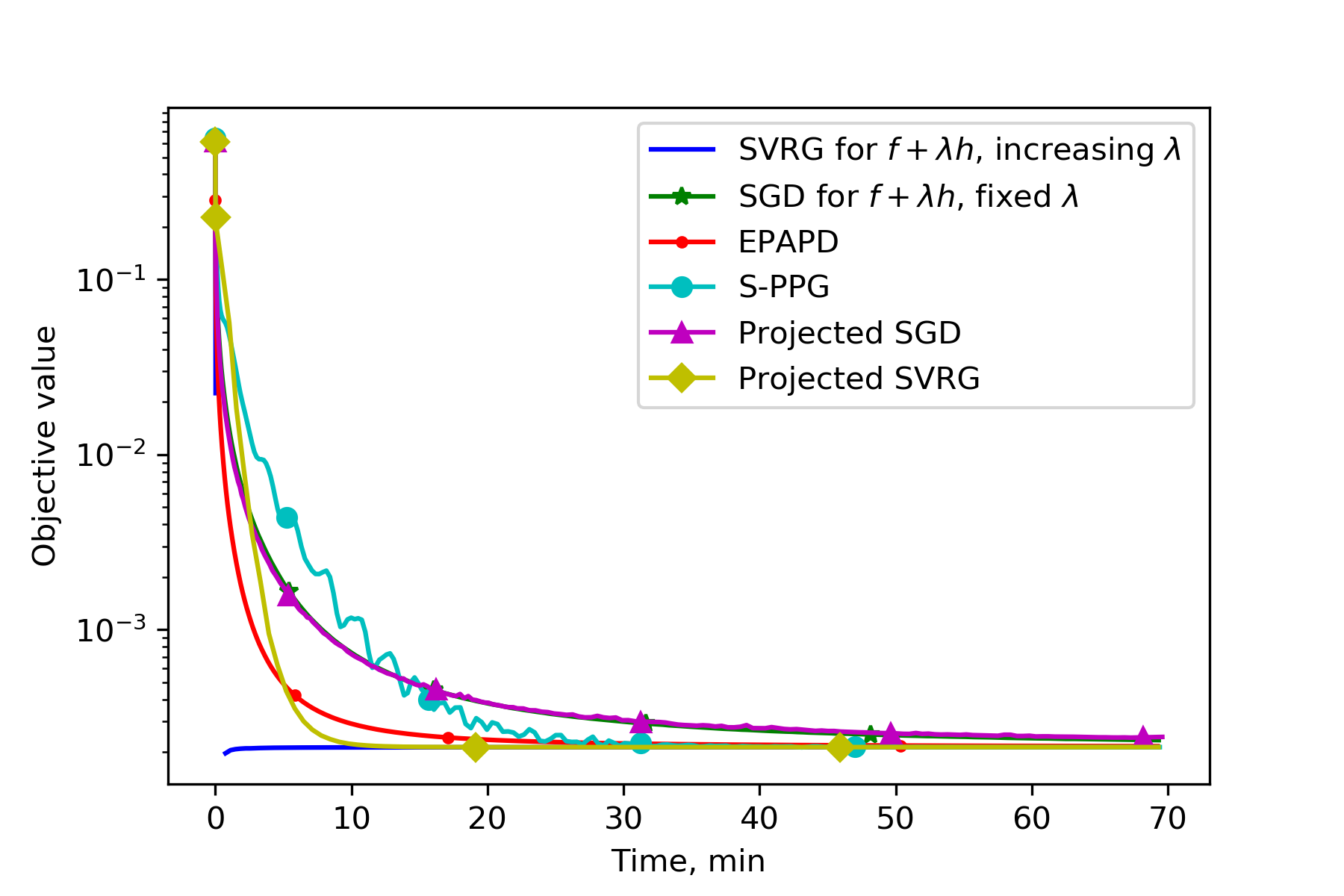}}
      & \raisebox{-\totalheight / 2}{\includegraphics[scale=0.30]{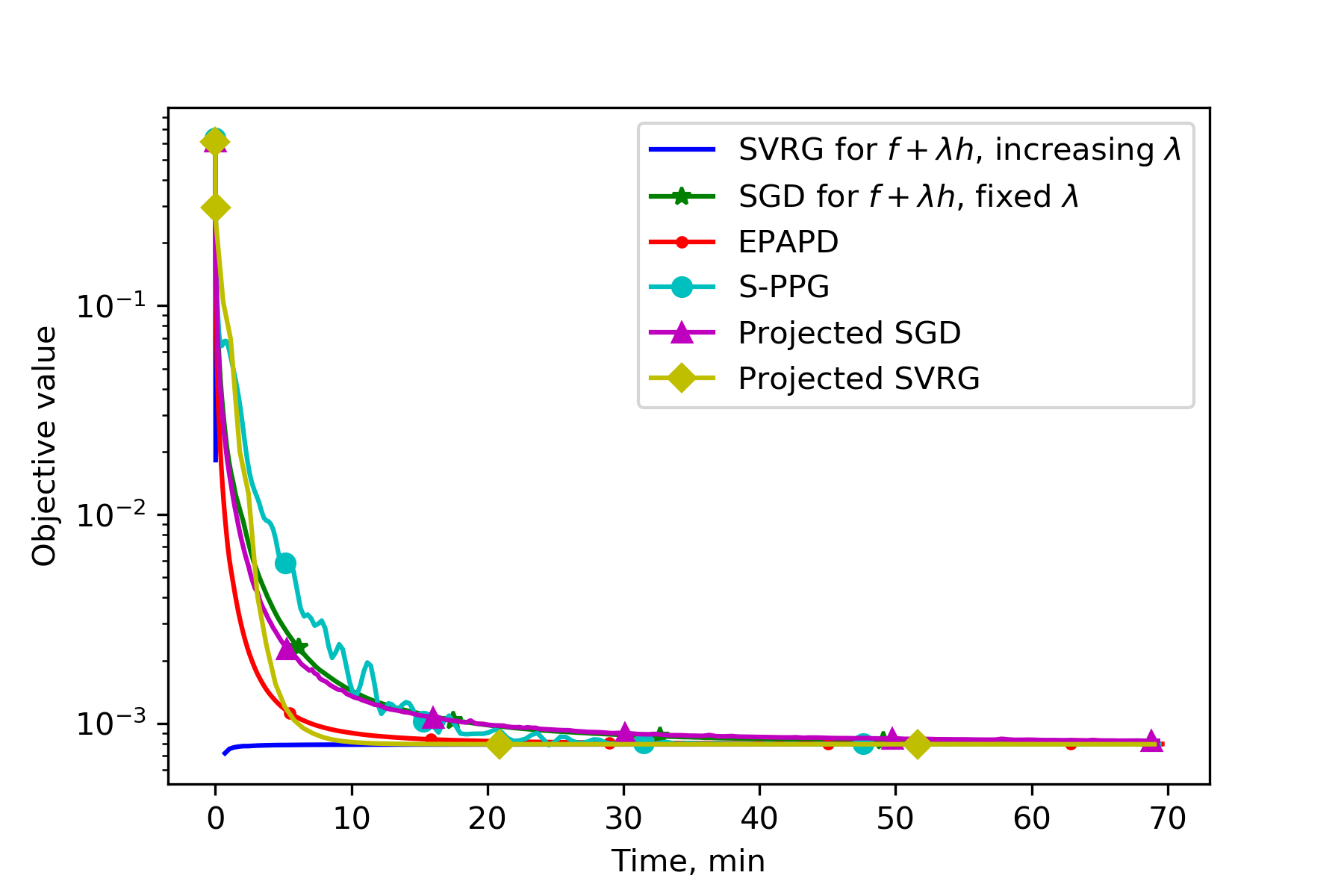}}
      & \raisebox{-\totalheight / 2}{\includegraphics[scale=0.30]{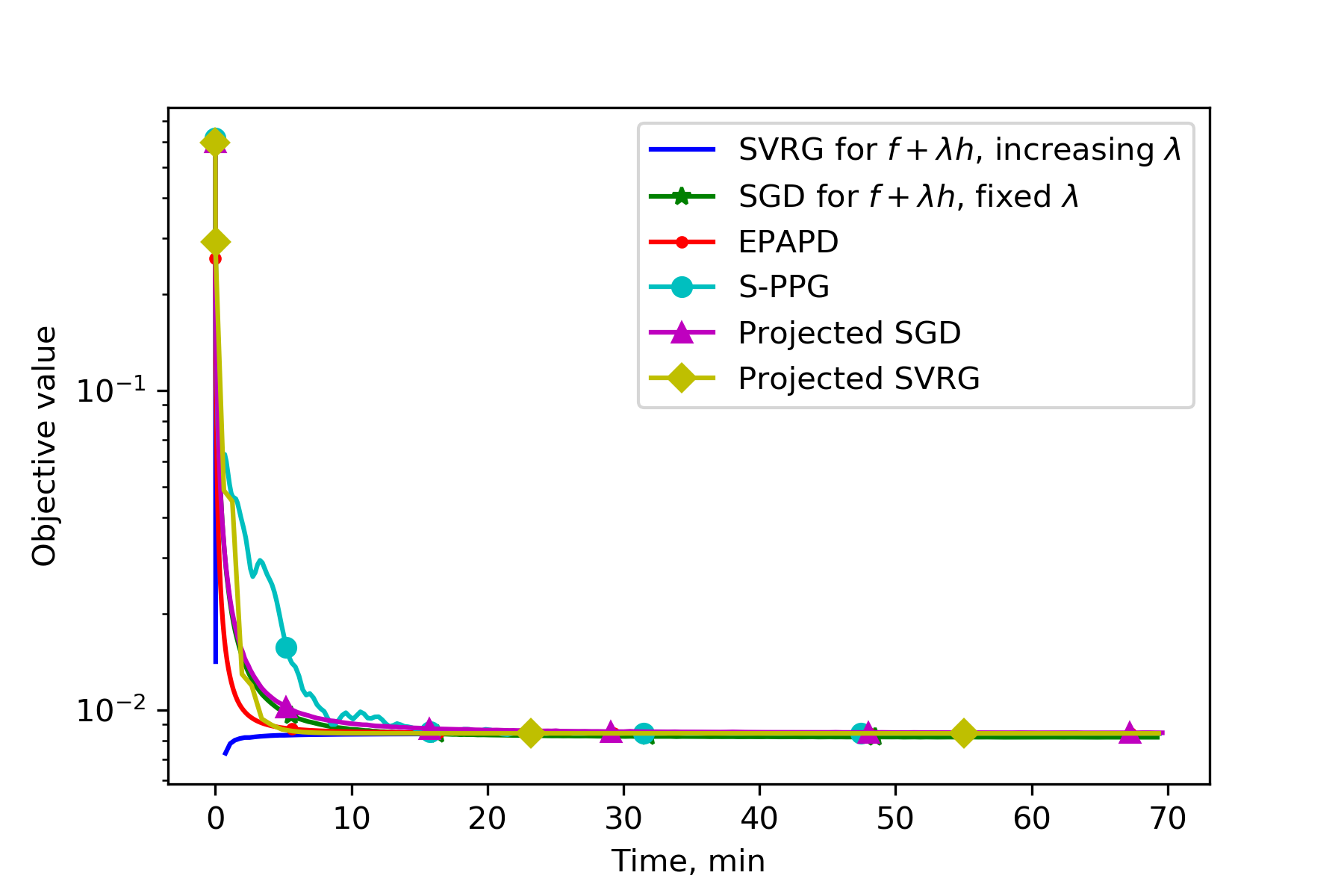}}
  \end{tabular}
  \end{center}
  \caption{Mushrooms dataset.}\label{tab:mushrooms_second}
\end{table}

\begin{table}[ht]
  \begin{center}
  \begin{tabular}{c@{\quad}cccc}
    & $m=100$  & $m=200$& $m=400$ \\
  \tiny    Infeasibility
      & \raisebox{-\totalheight / 2}{\includegraphics[scale=0.30]{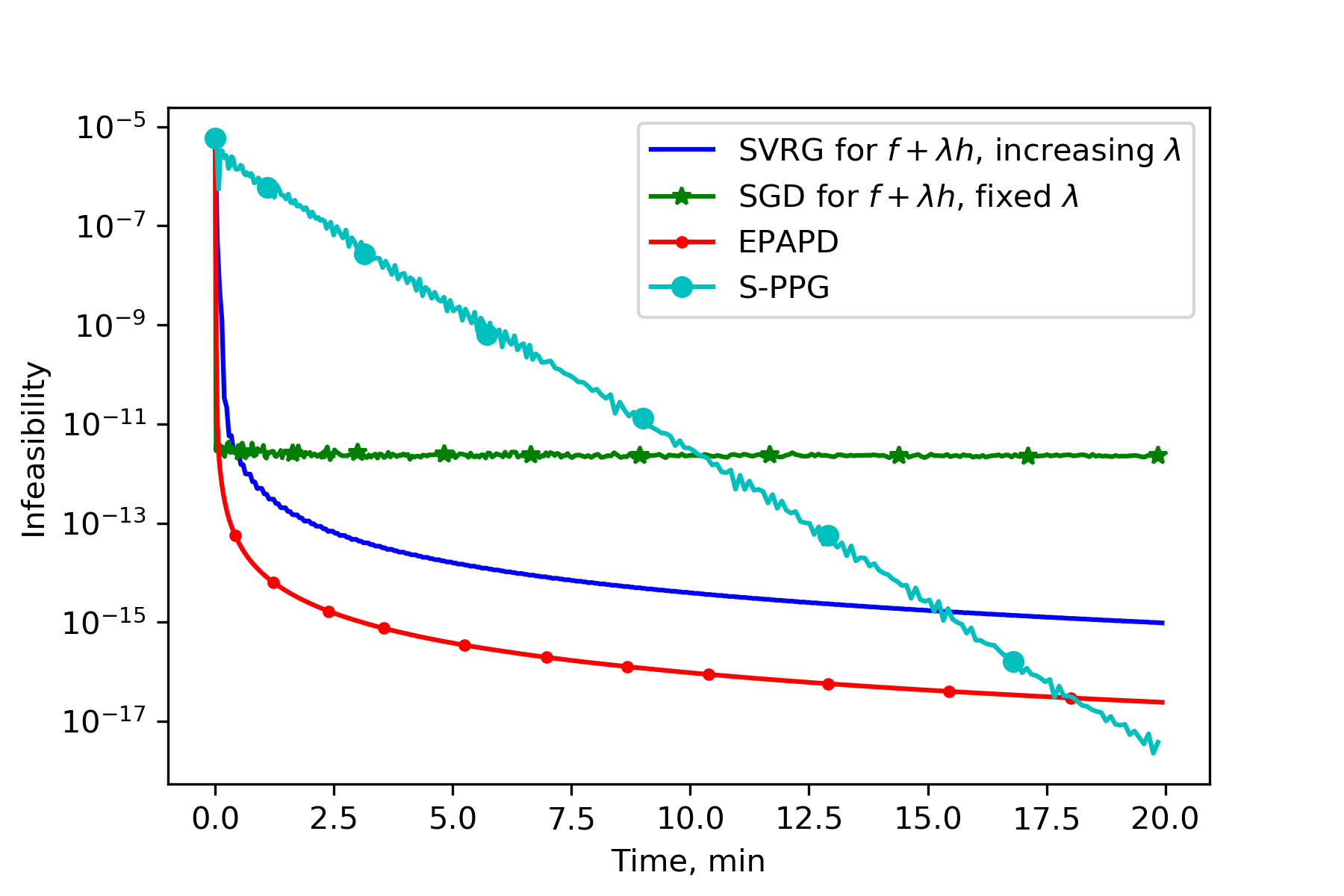}}
      & \raisebox{-\totalheight / 2}{\includegraphics[scale=0.30]{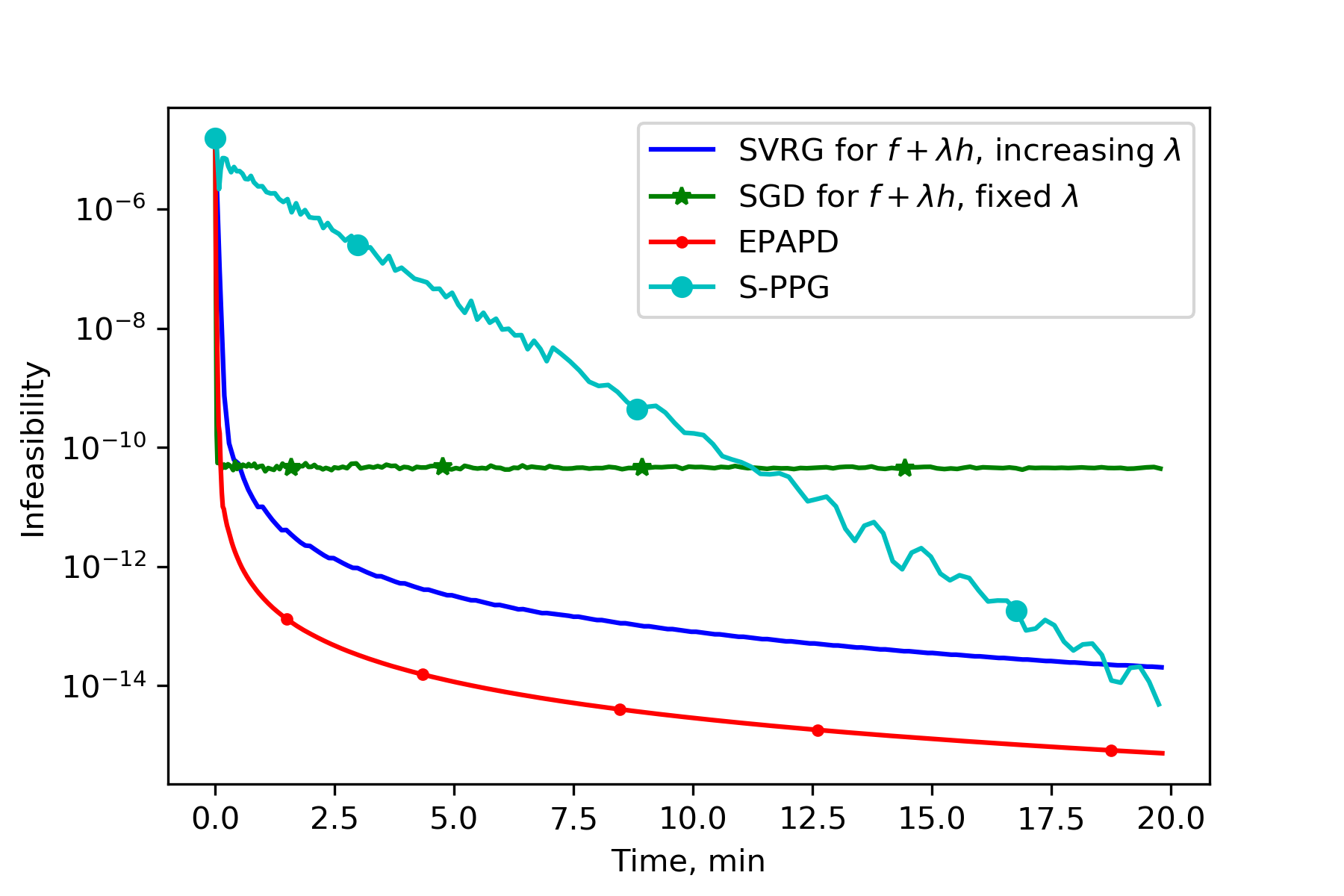}}
      & \raisebox{-\totalheight / 2}{\includegraphics[scale=0.30]{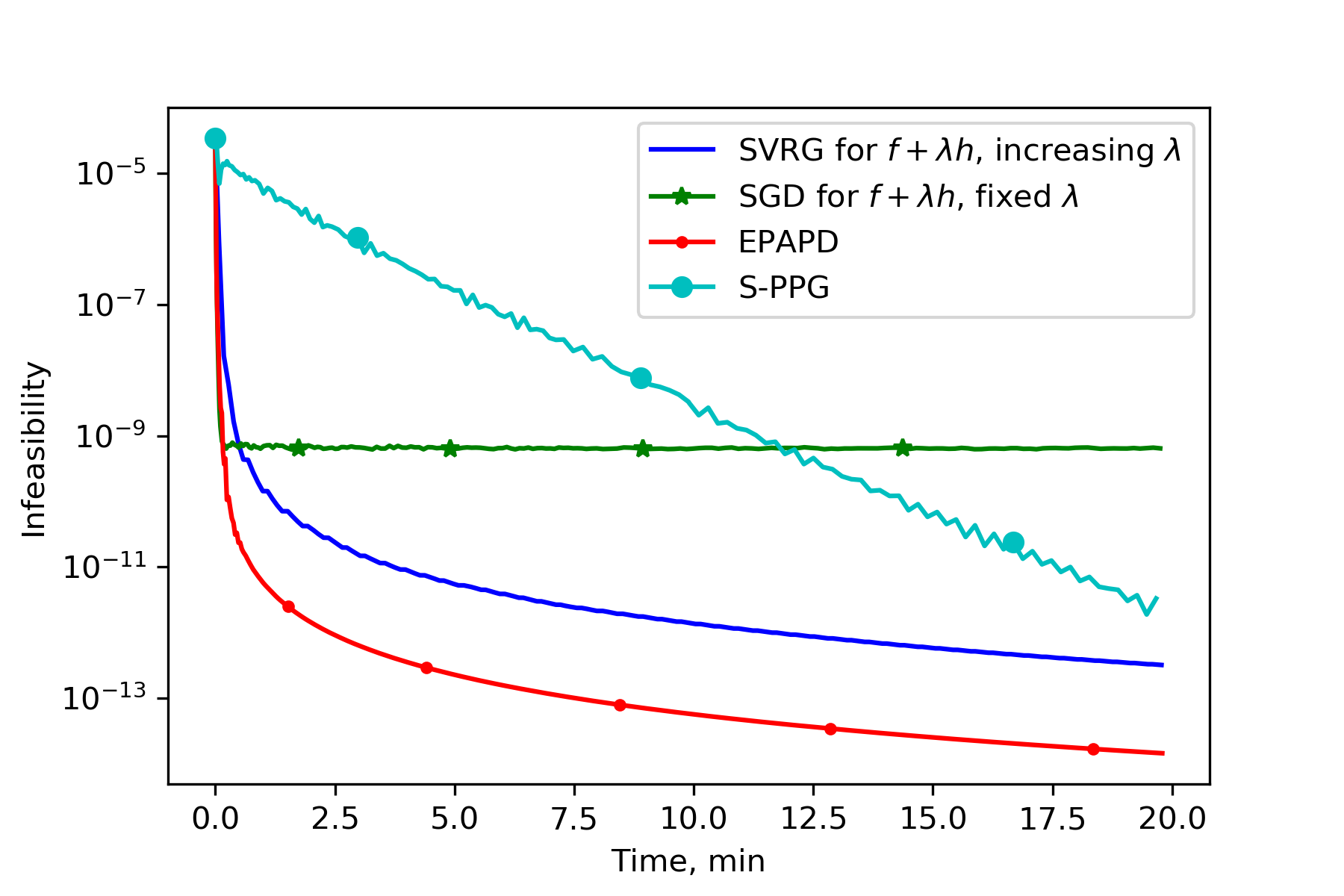}}\\ 
      
    \makecell{\tiny  Objective value \\ \tiny  of iterates \\ \tiny  without\\ \tiny  extra projection} 
      & \raisebox{-\totalheight / 2}{\includegraphics[scale=0.30]{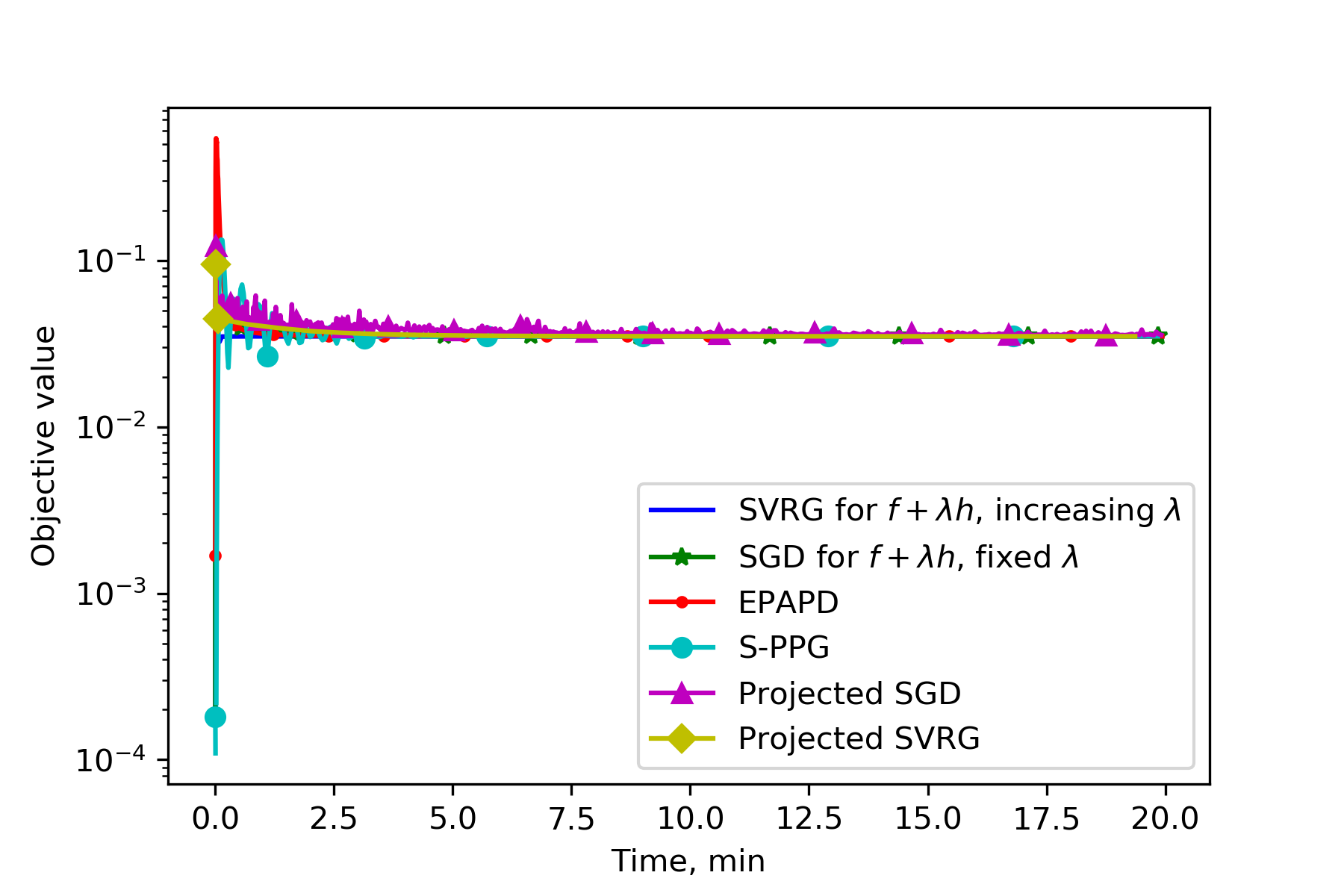}}
      & \raisebox{-\totalheight / 2}{\includegraphics[scale=0.30]{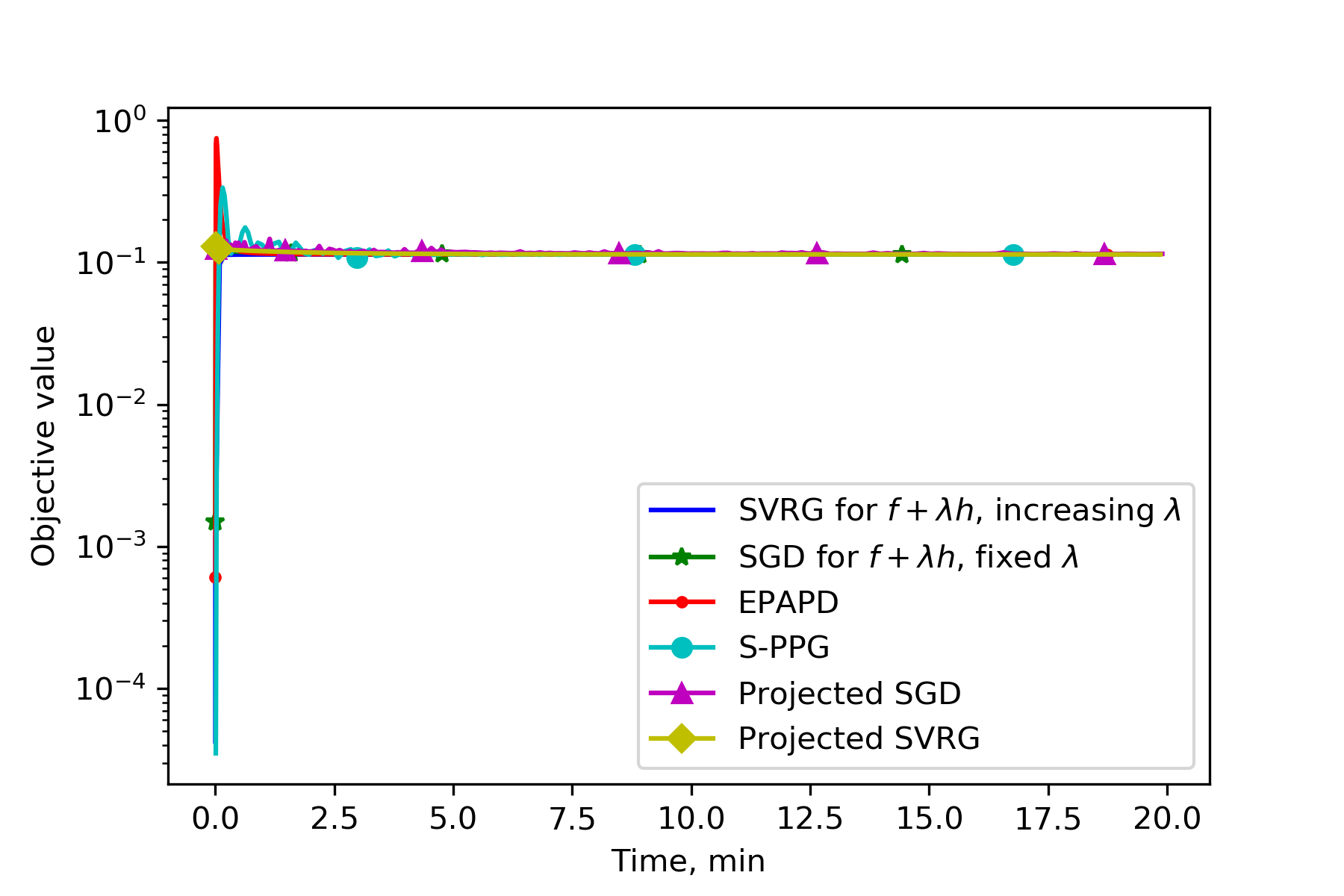}}
      & \raisebox{-\totalheight / 2}{\includegraphics[scale=0.30]{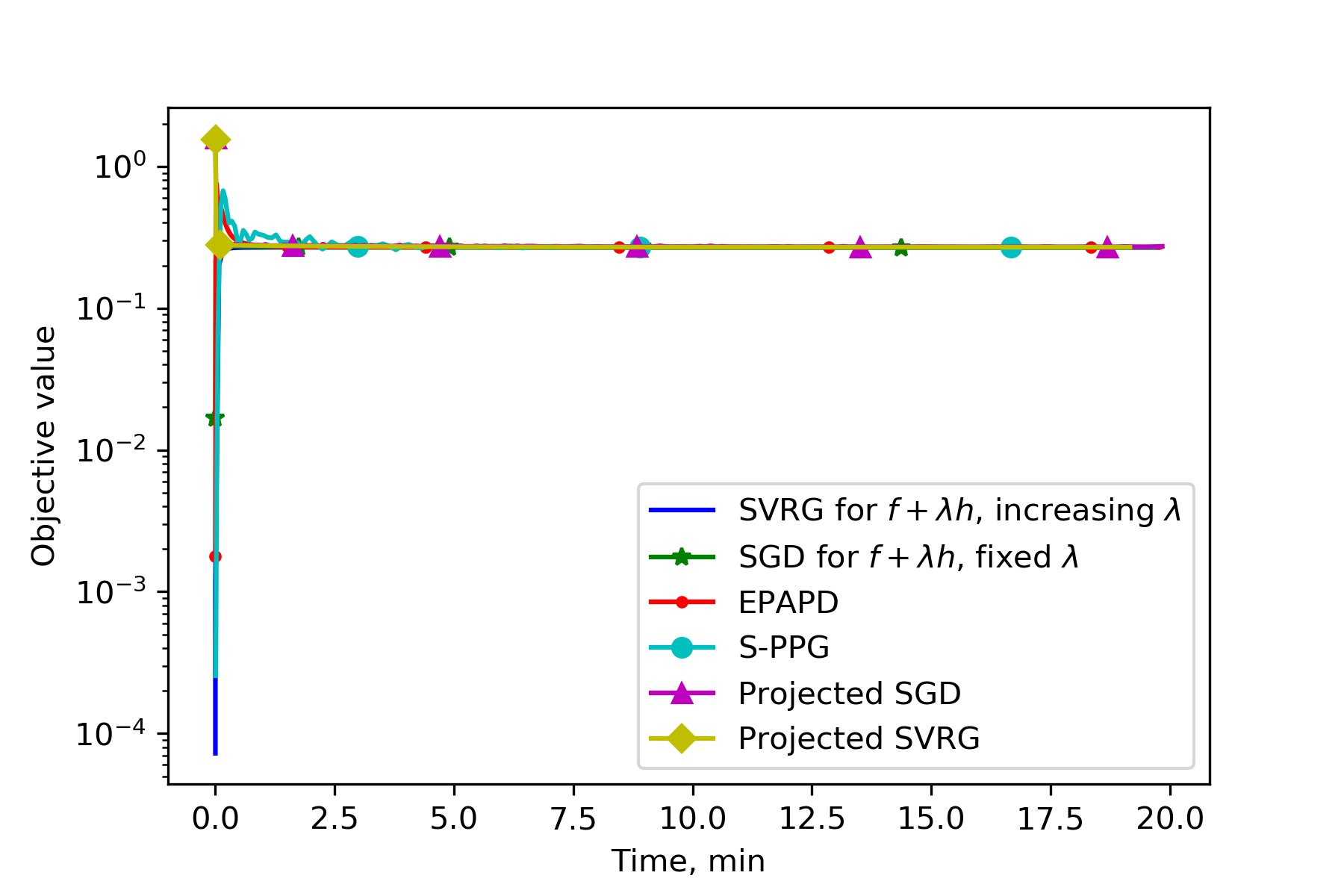}}
  \end{tabular}
  \end{center}
  \caption{Madelon dataset.}\label{tab:madelon_second}
\end{table}

\newpage
\section{Linear Regularity without Convexity} \label{sec:LR-suppl}
In this section we consider examples of nonconvex sets $\cX_1,\dots,\cX_m$ that satisfy the linear regularity assumption. Some of the examples are very general and suggest that using penalties to randomize constraints can be used in a variety of applications.

\begin{example}
    Consider that $\cX_j = \bigcup_{i=1}^{n_i} C_{j}^i$, where each $C_{j}^i$ is convex, and for any selection of indices $i_1, \dotsc, i_m$ the intersection $\bigcap_{j=1}^m C_{j}^{i_j}$ is either empty or satisfies $\bigcap_{j=1}^m \ri C_{j}^{i_j} \neq \emptyset$, where $\ri C$ denotes the relative interior of $C$ or just $C$ itself if it is a polyhedron. 
\end{example}

Note that the union of convex sets does not have to be convex. For instance, the union of $\{x\}$ and $\{y\}$, where $x\neq y\in\R^d$, is nonconvex.

\begin{proof}
Fix any $x$. It follows from $\PP(x)\in \cX$ that $\PP(x)\in \cX_j$ for any $j$, and, thus, there is an index $i_j$ such that $\PP(x)\in C_j^{i_j}$. Furthermore, the condition $\bigcap_{j=1}^m \ri C_{j}^{i_j} \neq \emptyset$ implies linear regularity of sets $C_1^{i_1}, \dotsc, C_m^{i_m}$ \cite{bauschke1999strong} with some constant $\gamma_{i_1,\dotsc, i_m} > 0$. Defining $\gamma$ as the minimal $\gamma_{i_1,\dotsc, i_m}$ over those $i_1,\dotsc, i_m$ that give non-empty intersection of sets $C_{1}^{i_1}, \dotsc, C_{m}^{i_m}$ gives us a lower bound on the linear regularity constant. As the minimum of a finite number of positive numbers this lower bound is positive.
\end{proof}

\begin{example}
    Let $\cX_j$ be the unions of subspaces and half-spaces $\cX_j = \bigcup_{i=1}^{n_j} \{a_i^T x \& b_i \}$, where $\min_j n_j>1$ and $\&$ can be either $'='$ or $'\le'$. 
\end{example}
\begin{proof}
	This is a simple corollary of the previous example.
\end{proof}

\begin{example}
    Consider $\cX_j\eqdef \{x\in \R^d \;:\; a_j^Tx \in D_j \}$, where $a_j\neq 0$ and $D_j$ is a (possibly nonconvex) set, e.g., $D_j$ can be equal to $\mathbb{Z}$ (the lattice of integers).
\end{example}
\begin{proof}
    Firstly, let us show that the specified sets might not be convex. If $D_j$ is not convex, then there exist $p, q\in D_j$ and $\alpha\in(0, 1)$ such that $\alpha p + (1 - \alpha)q \not\in D_j$. Since $a_j\neq 0 $, there exist $x, y$ such that $a_j^T x=p$ and $a_j^T y = q$. Consequently, $a_j^T (\alpha x + (1 - \alpha)y) = \alpha p + (1-\alpha)q \not\in D_j$, so set $\cX_j$ is also nonconvex.
    
    Now, let us show linear regularity. For any point $w$ its projection $\PP(w)$ onto $\cX$ belongs to $\cX_j$ for all $j$, so $\exists p_j\in D_j$ such that $a_j^T\PP(w) = p_j$. Therefore, $\PP(w)$ is also equal to the projection of $w$ onto $\widetilde\cX \coloneqq\bigcap_{j=1}^m\widetilde\cX_j\coloneqq \bigcap_{j=1}^m\{a_j^Tx = p_j\}$. Thus, it satisfies
    \begin{align}
        \lambda_{\text{min}}^+\|w - \PP(w)\|^2 
        &= \lambda_{\text{min}}^+\|w - \Pi_{\widetilde\cX}(w)\|^2\nonumber\\
        &\le \frac{1}{n}\sum \|w - \Pi_{\widetilde \cX_j}(w)\|^2 \nonumber\\
        &= \frac{1}{n}\sum \|w - \Pi_{\cX_j}(w)\|^2,\label{eq:lin_reg_of_dictionaries}
    \end{align}
    where $\lambda_{\text{min}}^+\coloneqq \lambda_{\text{min}}^+(A)$ is the smallest positive eigenvalue of matrix $A$ that is constructed by using vectors $\{a_j\}_{j=1}^m$ as rows (see~\cite{hoffman2003approximate} for more details). Since this number is independent of what choices $p_j\in D_j$ yielded the projection, inequality \eqref{eq:lin_reg_of_dictionaries} holds for all $w$.
\end{proof}

\begin{example}
\label{ex:one_set_is_intersection}
    Regardless of any other properties of sets $\cX_1,\dotsc, \cX_m$, if any of them is equal to their intersection $\cX = \bigcap_{j=1}^m \cX_j$, linear regularity holds with $\gamma\ge \frac{1}{m}$.
\end{example}
\begin{proof}
    Since $\cX_k$ is exactly $\cX$ for some $k$, the projection of $x$ onto it coincides with the one onto $\cX$ and using non-negativity of all other penalties we derive the claim.
\end{proof}

\begin{example}
    Let $\cX$ be the set of all matrices from $\R^{d_1\times d_2}$ with at most $s < d_1d_2$ nonzero components and projection be induced by Frobenius distance. Choose $t\ge 0$ and let sets $Q_1, \dotsc, Q_m$ be all possible subsets of positions $\{(i,j)\}$ with cardinality bigger than $s + t$. Then $\cX=\bigcap_{j=1}^m \cX_j$, where $\cX_j$ is the set of matrices with at most $s$ nonzeros in block $Q_j$, and linear regularity holds with $\gamma\ge \frac{1}{m}$.
\end{example}
\begin{proof}
    $\cX$ is nonconvex, because matrix $\frac{1}{d_1d_2}\sum_{i=1,j=1}^{d_1,d_2}e_i e_j^T$, where $e_i$ is the $i$-th basis vector, is not from this space, while every summand in the convex combination is.

    Assume we have a matrix $M$ and we project it onto $\cX$. If did not have more than $s$ nonzeros, its projection is equal to itself and we do not need to proove a bound for it, so let us assume it has more than $s$ nonzeros. The projection result, then, has exactly $s$ nonzeros and as we use Frobenius distance, they will be at the positions of elements with largest absolute value. Blocks $Q_1, \dotsc, Q_m$ cover all posible subsets of such indices, so there exist $k$ for which $\Pi_{\cX}(M) = \Pi_{\cX_j}(M)$ and 
    \[
    \frac{1}{m}\sum_{j=1}^m\|x-\Pi_{\cX_j}(x)\|^2\ge \frac{1}{m}\|x-\Pi_{\cX_j}(x)\|^2 \ge \frac{1}{m}\|x-\Pi_{\cX}(x)\|^2.
    \]
\end{proof}

\end{document}